\documentclass[preprint]{elsarticle}

\usepackage[ddmmyyyy]{datetime}
\usepackage{geometry}
\usepackage{lineno,hyperref,enumerate}
\nolinenumbers
\usepackage{float}
\usepackage{amsfonts}
\usepackage{mathrsfs}
\usepackage{amssymb}
\usepackage{amsmath}
\usepackage{amsthm}
\usepackage{mathabx}
\usepackage{appendix}
\usepackage{xcolor}
\usepackage{graphicx}
\usepackage[linesnumbered,lined,boxed,commentsnumbered,ruled]{algorithm2e}
\usepackage[textfont=scriptsize,labelfont=bf]{caption}
\usepackage[textfont=scriptsize,labelfont=scriptsize]{subcaption}
\usepackage{multirow}
\usepackage{booktabs}
\usepackage{tikz}
\usetikzlibrary{arrows,decorations.pathmorphing,backgrounds,positioning,fit,petri,graphs}
\usepackage{FMT}
\usepackage{FMTmicros}
\allowdisplaybreaks
\usepackage{rotating}

\newcommand{\filtbk}{{\boldsymbol{\eta}}}
\newcommand{{\fmt}}{{F$\mfd$T}}
\makeatletter
\def\ps@pprintTitle{%
 \let\@oddhead\@empty
 \let\@evenhead\@empty
 \def\@oddfoot{}%
 \let\@evenfoot\@oddfoot}
\makeatother

\begin{document}

\begin{frontmatter}

\title{Tight framelets and fast  framelet filter bank transforms on manifolds}

\author[addressCityU]{Yu Guang Wang\fnref{fn0}}
\ead{yuguang.wang@unsw.edu.au}
\author[addressCityU]{Xiaosheng Zhuang\fnref{fn0}\corref{corau}}
\ead{xzhuang7@cityu.edu.hk}

\address[addressCityU]{Department of Mathematics, City University of Hong Kong, Tat Chee Avenue, Kowloon Tong, Hong Kong}

\cortext[corau]{Corresponding authors.}

\fntext[fn0]{The research  of X. Z. and Y. G. W., and the work described in this paper was partially supported by a grant from the Research Grants Council of the Hong Kong Special Administrative Region, China (Project No.  CityU 11300717) and grants from City University of Hong Kong (Project No.: 7200462 and 7004445).}

\begin{abstract}
Tight framelets on a smooth and compact Riemannian manifold $\mfd$ provide a tool of multiresolution analysis for data from  geosciences, astrophysics, medical sciences, etc. This work investigates the construction, characterizations, and applications of tight framelets on such a manifold $\mfd$.  Characterizations of the tightness of a sequence of framelet systems for $\Lpm{2}$ in both the continuous and semi-discrete settings are provided. Tight framelets associated with framelet filter banks on $\mfd$ can then be easily designed and fast framelet filter bank transforms on $\mfd$ are shown to be realizable with nearly linear  computational complexity. Explicit construction  of tight framelets on the sphere $\sph{2}$ as well as numerical examples are given.
\end{abstract}


\end{frontmatter}

\section{Introduction and motivation}
\label{sec:intro}

In the era of information technologies, the rapid development of modern high-tech devices, for example, a super computer, PC, smart phone, wearable and VR/AR device, is driven internally by  Moore's Law \cite{wiki_Moore} which contributes to the exponential growth of the computational power, while externally stimulated by the tremendous need of both the public and individual parties in processing massive data from finance, economy, geology, bio-information, cosmology, medical sciences and so on. It has been noticed that Moore's Law is slowing down due to the constrains of the physical law \cite{wiki_IntelTickTock} but the volume of data is dramatically increasing. Dealing with Big Data is becoming a crucial part of an individual person, party, government and country.

Real-world data often inherit high-dimensionality such as data from a surveillance system, seismology, climatology. High-dimensional data are typically concentrated on a low-dimensional manifold \cite{RoSa2000,TeDeLa2000}, for instance, the sphere in remote sensing and CMB data \cite{Bennett_etal2013}, more complex surfaces in brain imaging \cite{ThTo1996}, and discrete graph data from social and traffic networks \cite{ShWiHoVa2015}.
Analysis and learning tools on manifolds hence play an increasingly important role in machine learning and statistics.

The key to successful manifold learning lies in that data on a manifold may exhibit high complexity on one hand while they are highly sparse at a certain domain via an appropriate multiscale representation system on the other hand. Sparsity within such representations, stemming from computational harmonic analysis, enables efficient analysis and processing of high-dimensional and massive data.


Multiresolution  analysis  in general are designed for data in the Euclidean space $\Rd$, $d\ge1$, for example, a signal in $\Rone$, an image in $\Rd[2]$ and a video in $\Rd[3]$. Multiscale representation systems in $\Rd$ including wavelets, framelets, curvelets, shearlets, etc., which are capable of capturing the sparsity of data, have been well-developed and widely used, see e.g. \cite{BoKuZh2015,CaDeDoYi2006,Chui1992,CoDaVi1993,Daubechies1992,Mallat2009, Meyer1990}.  The core of the classical framelet (and wavelet) construction relies on the extension principles such as unitary extension principle (UEP) \cite{RoSh1997}, oblique extension principle (OEP) and mixed extension principle (MEP) \cite{DaHaRoSh2003}. The extension principles associate framelet systems with filter banks, which enables fast algorithmic realizations for the framelet transforms and applications, see e.g. \cite{DaHaRoSh2003,HaZhZh2016,Mallat2009}. The fast algorithms that include the filter bank decomposition and reconstruction of a representation system which uses convolution and FFT achieve computational complexity in proportion to the size of the input data (up to a log factor).

Different from on Euclidean domains, multiscale representation systems and their corresponding fast algorithmic realizations on a general compact manifold are less studied. One of the reasons is that the operators of translation and dilation for classical wavelet and framelet systems in $\R^d$ can not be in parallel extended to general manifolds. We have to look for alternative approaches. One possible approach is based on the central idea  behind wavelet analysis on $\Rd$: the time domain operators have  their equivalences in the Fourier domain. The tight framelet construction on a manifold of this paper, which uses orthogonal polynomials and localized kernels, is closely related to this approach. The main idea is that a sequence of orthogonal polynomials plays the role of a Fourier basis and can be used to define a  localized kernel from which ``translation'' and ``dilation'' can be obtained. Such an approach can be seen in Fischer, Mhaskar and Prestin in \cite{FiPr1997, MhPr2004}, where they show that wavelets or polynomial frames can be extended to general domains including intervals and  spheres. Coifman, Maggioni, Mhaskar and Dong \cite{CoMa2006,Dong2017,MaMh2008,Mhaskar2010} consider more general cases, for which diffusion wavelets, diffusion polynomial frames and wavelet tight frames on manifolds and graphs are constructed.

Besides orthogonal polynomials and localized kernels on $\mfd$, our characterization and construction of tight framelets on $\mfd$ also rely on (nonhomogeneous) affine systems
\[
\mathsf{AS}_J(\{\varphi; \psi^1,\ldots,\psi^r\}) =\{\varphi_{J,k} \setsep k\in\mathcal{I}_J\}\cup\{\psi^n_{j,k} \setsep k\in\mathcal{J}_j, n = 1,\ldots,r, j\ge J \}, \quad J\in\Z,
\]
where $\{\varphi;\psi^1,\ldots,\psi^n\}$ is a set of generators and the subscripts $j$ and $k$ encode certain ``dilation'' and ``translation'' information with $\mathcal{I}_j,\mathcal{J}_j$ being the index sets at scale $j$.
In the classical wavelet analysis, the wavelets or framelets $\varphi_{j,k}:=2^{j/2}\varphi(2^j\cdot-k)$ and $\psi^n_{j,k}:=2^{j/2}\psi^n(2^j\cdot-k)$, $k\in\Z$ in $\Rone$ are defined by dilation and translation associated with a set $\{\varphi;\psi^1,\ldots,\psi^r\}$ of generators in $L_2(\R)$.
One of the fundamental problems in classical wavelet analysis is to  construct an affine system $\mathsf{AS}_J(\{\varphi; \psi^1,\ldots,\psi^r\})$ that will form an orthonormal basis, a Riesz basis or a frame for $\Lpm{2}$.
In the frame theory, such a system is called a \emph{framelet system}, the elements of which are called \emph{framelets}. \emph{Tight framelets} refer to elements of a framelet system with equal lower and upper frame bounds. See \cite{Daubechies1992,DaHaRoSh2003}.

The construction of affine systems of wavelets in $\Rd$ has been studied in \cite{RoSh1997}. Sequences of affine systems are studied in \cite{Han2010, Han2012} and later extended to affine shear systems in \cite{HaZh2015,Zhuang2015,Zhuang2016}. Han \cite{Han2010, Han2012} shows that the sequences of affine systems are of fundamental importance in the analysis and construction of framelet systems, for example, in the MRA, the filter bank structure and the extension principles \cite{ChHeSt2002,DaHaRoSh2003,RoSh1997}. More discussions refer to \cite{ChHeSt2002, DaHaRoSh2003,Han2010,Han2012,HaZh2015, RoSh1997,Zhuang2016} and references therein.
Adopting the framework of sequences of affine systems \cite{Han2010,Han2012} and the approach of orthogonal polynomials in \cite{Dong2017,FiPr1997,MhPr2004}, we show that a sequence   of tight frames  for $L_2(\mfd)$, called \emph{continuous tight framelet (system)} $\cfrsys(\Psi;\mfd)$, can be constructed based on a framelet generating set $\Psi=\{\scala;\scalb^1,\ldots,\scalb^r\}$ on $\Rone$ and an orthonormal eigen-pair  $\{(\eigvm,\eigfm)\}_{\ell=0}^\infty$ on $\mfd$. See Section~\ref{subsec:cont_framelets} for details.

For computation and application, we discretize the continuous tight framelets by using a sequence of polynomial-exact quadrature rules $\mathcal{Q}:=\{\QN\}_{j\ge J}$ on $\mfd$. This leads to a simple approach of constructing \emph{(semi-discrete) tight framelets} $\frsys(\Psi,\QQ;\mfd)$ for $\Lpm{2}$. We show that if the framelet generating set $\Psi=\{\scala;\scalb^1,\ldots,\scalb^r\}$ is associated with a filter bank $\filtbk=\{\maska;\maskb[1],\ldots,\maskb[r]\}$, see \eqref{eq:refinement}, the characterization conditions of $\filtbk$ for the tightness of semi-discrete framelets on $\mfd$ are greatly simplified, which  facilitates the design and application of the tight framelets.

By exploiting the refinement structure for the filters in \eqref{eq:refinement} and the properties of the tight frame $\frsys(\Psi,\QQ;\mfd)$, we can design the \emph{framelet filter bank decomposition algorithm}
and the \emph{framelet filter bank reconstruction algorithm}, where
the decomposition uses (discrete) convolutions with filters in the filter bank $\filtbk$ and downsampling operations, and the reconstruction uses convolutions and upsampling operations.
Figure~\ref{fig:algo:fb} depicts one-level decomposition and reconstruction at scale $j$. Since convolution is equivalent with discrete Fourier transforms on $\mfd$, the decomposition and reconstruction can be implemented \emph{fast} using fast discretet Fourier transforms (FFTs) on $\mfd$. We then call the decomposition and reconstruction \emph{fast framelet filter bank transforms} ({\fmt}s). The (multi-level) {{\fmt}} algorithms are recursive one-level framelet filter bank transforms, see Section~\ref{sec:fast.algo} for details. The {\fmt}s provide a tool for efficient multiscale data analysis on $\mfd$.

\begin{figure}[htb]
\begin{minipage}{\textwidth}
\centering
\begin{center}
\begin{tikzpicture}[nonterminal/.style={rectangle, minimum size=6mm, very thick, draw=red!50!black!50,top color=white, bottom color=red!50!black!20,font=\itshape},
terminal/.style={rectangle,minimum size=6mm,rounded corners=1mm,very thick,draw=black!50,top color=white,bottom color=black!20,font=\ttfamily},
sum/.style={circle,minimum size=1mm,very thick,draw=black!50,top color=white,bottom color=black!20,font=\ttfamily},
skip loop/.style={to path={-- ++(0,#1) -| (\tikztotarget)}},
hv path/.style={to path={-| (\tikztotarget)}},
vh path/.style={to path={|- (\tikztotarget)}},
,>=stealth',thick,black!50,text=black,
every new ->/.style={shorten >=1pt},
graphs/every graph/.style={edges=rounded corners}]
\matrix[row sep=1mm,column sep=6mm] {
& & \node (ma1) [terminal] {$\dconv\maska^\star$}; & \node (ds1) [terminal] {$\downsmp \hspace{-0.5mm}$}; & \node (pr1) [nonterminal] {processing}; & \node (us1) [terminal] {$\upsmp \hspace{-0.5mm} $}; & \node (ma2) [terminal] {$\dconv\maska$}; & & \\
\node (in) [nonterminal] {input}; & \node (p1) [coordinate] {}; & & & & & &\node (plus) [sum] {$+_{r}$}; & \node (out) [nonterminal] {output};\\
& & \node (mb1) [terminal] {$\dconv(\maskb[n])^\star$}; & & \node (pr2) [nonterminal] {processing}; & & \node (mb2) [terminal] {$\dconv\maskb[n]$}; & & \\
};

\graph [use existing nodes] {
ma1 -> ds1 -> pr1 -> us1 -> ma2;
mb1 -> pr2 -> mb2;
in -- p1;
p1 ->[vh path]{ ma1, mb1 };
{ma2, mb2 } -> [hv path] plus -> out;
};
\end{tikzpicture}
\end{center}
\vspace{-1mm}
\begin{minipage}{0.8\textwidth}
\caption{One-level framelet filter bank decomposition and reconstruction based on a filter bank $\{\maska;\maskb[1],\ldots,\maskb[r] \}$ at scale $j$. Here the filters $\maskb[n]$ range over $n=1,\ldots,r$ and the node $+_r$ sums over the low-pass filtered coefficient sequence and all $r$ high-pass filtered coefficient sequences.}
\label{fig:algo:fb}
\end{minipage}
\end{minipage}
\end{figure}
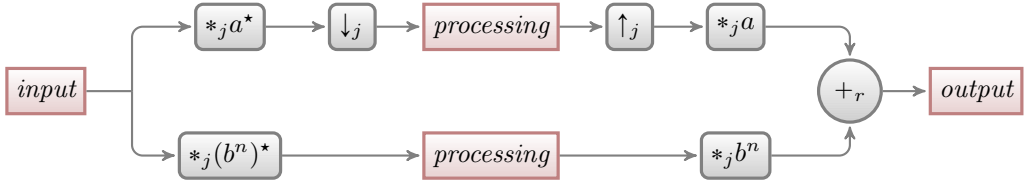


Before we proceed to detail the construction of continuous and semi-discrete tight framelets on $\mfd$ and their discretization in Section~\ref{sec:tight.framelets}, we state the major contributions of the paper in the following aspects.
\begin{enumerate}[(1)]
\item \emph{Sequences of framelet systems on a manifold.} Most of literature on frames and tight frames on manifolds only consider a fixed system $\frsys[0](\Psi,\QQ;\mfd)$ with two framelet generators, i.e. $\Psi=\{\scala;\scalb\}$, see e.g. \cite{MaMh2008, Mh2005, NaPeWa2006-2}. As far as we are concerned, there is no literature on the investigation of a sequence $\{\frsys(\Psi,\QQ;\mfd)\setsep J\ge {\ord[0]}\}$ of framelet systems on a compact Riemannian manifold $\mfd$ for some ${\ord[0]}\in\Z$ and for $\Psi=\{\scala;\scalb^1,\ldots,\scalb^r\}$ with multiple framelet generators.
    In this paper, we introduce sequences of framelet systems on a manifold and provide a complete characterization (equivalence conditions) of a sequence of framelet systems to be a sequence of tight frames in $\Lpm{2}$, which greatly simplifies the construction of tight framelets on $\mfd$. Moreover, with the flexible number of framelet generators, one can  separate the ``frequency domain'' in a more careful way that enables more sophisticated data analysis on different ``frequency'' ranges  (see Examples~\ref{ex:1} and \ref{ex:2}), which  are important in application, such as denoinsing or inpainting on a manifold.

\item \emph{MRA structure and filter banks association.} From the equivalence relations in Theorem~\ref{thm:framelet.tightness}, a sequence of tight frames for $\Lpm{2}$ has a multiresolution (MRA) structure for $\Lpm{2}$. The MRA structure is then naturally associated with a filter bank, which helps to design a fast realization of the framelet transforms on $\mfd$. We should point out that the papers \cite{CoMa2006, LeMh2008,MaMh2008,Mhaskar2010} focus on the characterization with respect to $\Psi=\{\scala;\scalb\}$ and with no filter bank associated. Dong \cite{Dong2017} considers $\cfrsys[0](\Psi;\mfd)$ with FIR (finite impulse response) filter banks whose masks have fully supported  Fourier series, which makes it impossible to involve polynomial-exact quadrature rules on $\mfd$ for discretization. In the paper, we provide a complete characterization of a sequence of tight framelets for $\Lpm{2}$ in terms of the associated filter bank in both the FIR and IIR (infinite impulse response or band-limited) cases, and also demonstrate that using band-limited filter banks enables the discretization of the continuous framelets via polynomial-exact quadrature rules and the efficient implementation of the framelet filter bank transforms.

\item \emph{Unitary extension principle and quadrature rules on a manifold.} The equivalence conditions of (iv) and (v) in Theorem~\ref{thm:framelet.tightness} for a sequence of tight framelets in $\Lpm{2}$ in terms of the associated framelet generators $\Psi=\{\scala;\scalb^1,\ldots,\scalb^r\}$ and the associated filter bank $\filtbk=\{\maska;\maskb[1],\ldots,\maskb[r]\}$ provide a \emph{new} unitary extension principle (UEP) for $\Lpm{2}$, which is a non-trivial generalization of classical unitary extension principle \cite{DaHaRoSh2003, RoSh1997} for $\Lpm[\R]{2}$. The conditions \eqref{thmeq:fr.scal.j.j1} and \eqref{thmeq:fr.tightness.a.mask.cond} are new as far as we are concerned. These two equivalence conditions not only simplify the construction of tight framelets for $\Lpm{2}$, but also give the connection of tight framelets with quadrature rules for numerical integration on a compact Riemannian manifold.

\item \emph{Fast framelet filter bank transforms on manifolds.} The fast algorithmic realization for framelet filter bank transforms on a general compact Riemannian manifold is new as far as we are concerned. Assuming FFT on $\mfd$, which holds for many important manifolds including torus, sphere and Grassmannian, we demonstrate that the fast framelet transforms on a manifold proposed in this paper have (up to a log factor) the linear computational complexity and the low redundancy rate (or the low data complexity). The computational complexity and the redundancy rate are both in proportion to the size of the input data, and are independent of the decomposition level. We remark that we focus on fast algorithms on smooth manifolds rather than on graphs, which is another important problem to explore. A smooth manifold and a graph have a fundamental difference although the latter can be embedded into a smooth Riemannian manifold: a smooth manifold has nice geometric properties with explicitly known form of orthonormal systems which can be exploited for the design of fast discrete Fourier transforms; a graph only has the topological structure (see e.g. \cite{Si2006}) and the analysis heavily relies on the spectral graph theory \cite{Ch1997}. Dong \cite{Dong2017} and Hammond et al. \cite{HaVaGr2011} studied the algorithms of wavelet transforms (WFTG and SGWT) for graph data based on spectral graph theory. However, as their transforms have no downsampling process, the redundancy rate and the computational complexity increase exponentially with respect to the decomposition level.

\end{enumerate}

The remaining of the paper is organized as follows. In Section~\ref{sec:tight.framelets}, we provide a complete characterization for a sequence of framelet systems  to be a sequence of tight frames in $\Lpm{2}$ in both the continuous and semi-discrete scenarios. We show that polynomial-exact quadrature rules on $\mfd$ give a simple way of constructing semi-discrete tight framelets in $\Lpm{2}$. In Section~\ref{sec:fast.algo}, for tight framelets associated with a filter bank and a sequence of polynomial-exact quadrature rules on $\mfd$, we describe the multi-level framelet filter bank decomposition and reconstruction algorithms. We give fast framelet filter bank transforms ({\fmt}s) with nearly linear computational  complexity and low redundancy rate based on the fast algorithms for discrete Fourier transforms (FFTs) on $\mfd$. In Section~\ref{sec:fmtS2} we construct framelets on the sphere $\sph{2}$ with two high passes ($\maskb[1]$ and $\maskb[2]$).
Section~\ref{sec:numer} gives numerical examples for the {\fmt} algorithms on $\sph{2}$ using the nonequispaced fast spherical Fourier transforms (NFSFTs) of Keiner, Kunis and Potts \cite{KeKuPo2007}. Final remarks  are given in the last section.

\section{Tight framelets on manifolds}
\label{sec:tight.framelets}
In this section, we give a complete characterization for a sequence of  framelet systems to be a sequence of continuous tight framelets for $\Lpm{2}$ and show that the discretization of continuous tight framelets using quadrature rules can achieve semi-discrete tight framelets for $\Lpm{2}$.

Throughout the paper, we assume that the manifold $\mfd$ has the following properties.
\begin{enumerate}[(1)]
\item The manifold $\mfd$ is a $d$-dimensional compact, connected, and smooth Riemannian manifold with smooth boundary (possibly empty) for $d\ge2$  equipped with a probability measure $\memf$ $(\memf(\mfd)=1)$.  The space $\Lpm{2}:=\Lpm[\mfd,\memf]{2}$ is the space of complex-valued square integrable functions on $\mfd$ with respect to $\memf$ endowed with the $L_{2}$-norm $\norm{f}{\Lpm{2}}:=\left(\int_{\mfd}|f(\PT{x})|^{2}\dmf{x}\right)^{1/2}$ for $f\in \Lpm{2}$. Note that $\Lpm{2}$ is a Hilbert space with inner product $\InnerL{f,g}:=\InnerL[\Lpm{2}]{f,g}:=\int_{\mfd}f(\PT{x})\conj{g(\PT{x})}\dmf{x}$, $f,g\in\Lpm{2}$, where $\conj{g}$ is the complex conjugate to $g$.

\item $\{\eigfm\}_{\ell=0}^\infty$ and $\{\eigvm\}_{\ell=0}^\infty$ are two seqeunces. The  sequence $\{\eigfm\}_{\ell=0}^\infty\subset \Lpm{2}$ is an orthonormal basis for $\Lpm{2}$ with $\eigfm[0]\equiv1$; i.e. $\InnerL{\eigfm[\ell],\eigfm[\ell']} = \delta_{\ell,\ell'}$, where $\delta_{\ell,\ell'}$ is the \emph{Kronecker delta} with $\delta_{\ell,\ell'}=1$ if $\ell=\ell'$ and $0$ otherwise, and the sequence $\{\eigvm\}_{\ell=0}^\infty\subset\R$ is a nondecreasing sequence of nonnegative numbers satisfying $0=\eigvm[0]\le \eigvm[1]\le \cdots$ and $\lim_{\ell\rightarrow\infty} \eigvm = \infty$. The sequence  $\{(\eigfm,\eigvm)\}_{\ell=0}^\infty$ is said to be an \emph{orthonormal eigen-pair} for $\Lpm{2}$. A typical example of $\{(\eigfm,\eigvm)\}_{\ell=0}^\infty$ is the set of pairs of the eigenfunctions and eigenvalues of the \emph{Laplace-Beltrami operator} $\LBm$ on $\mfd$ satisfying $\LBm\eigfm = -\eigvm^{2}\eigfm$ for $\ell\in\N_0:=\N\cup\{0\}$.
\end{enumerate}

Since $\{(\eigfm,\eigvm)\}_{\ell=0}^\infty$ is an orthonormal eigen-pair for $\Lpm{2}$, the (generalized) \emph{Fourier coefficients} $\Fcoem{f}, \ell\in\N_0$ of a function $f\in\Lpm{2}$ can be defined to be  $\Fcoem{f} := \InnerL{f,\eigfm}$, $\ell\in\N_0$.
Then any function  $f\in \Lpm{2}$  has the Fourier expansion $f=\sum_{\ell=0}^\infty\Fcoem{f}\eigfm$ in $\Lpm{2}$ and Parsevel's identity  $\|f\|_{\Lpm{2}}^2 = \sum_{\ell=0}^{\infty} |\Fcoem{f}|^2$ holds.

To construct framelets on $\mfd$, we let
\[
\Psi:=\{\scala; \scalb^1,\ldots,\scalb^r\}\subset L_1(\R),
\]
a set of generating functions, or \emph{(framelet) generators},
where $L_1(\R)$ is the space of absolutely integrable functions  on $\Rone$ with respect to the Lebesgure measure. The \emph{Fourier transform} $\FT{\gamma}$ of a function $\gamma\in \Lpm[\R]{1}$ is $\FT{\gamma}(\xi):=\int_{\R}\gamma(t)e^{-2\pi i t\xi} \IntD{t}$, $\xi\in\Rone$ (with abuse of notation). The Fourier transform on $\Lpm[\R]{1}$ can be naturally extended to the $L_{2}$ space $\Lpm[\R]{2}$ of square integrable functions on $\Rone$. As wavelets and framelets in $\Rd$, the set of generators $\Psi$ is associated with  a \emph{(framelet) filter bank}
\[
\filtbk:=\{\maska; \maskb[1],\ldots,\maskb[r]\}\subset l_1(\Z):=\{h=\{h_k\}_{k\in\Z}\subset \C \setsep \sum_{k\in\Z} |h_k|<\infty \}
\]
by the following relation:
\begin{equation}
\label{eq:refinement}
    \FT{\scala}(2\xi) = \FS{\maska}(\xi)\FT{\scala}(\xi),\quad
    \FT{\scalb^n}(2\xi) = \FS{\maskb[n]}(\xi)\FT{\scala}(\xi),\quad n=1,\ldots,r, \; \xi\in\Rone,
\end{equation}
where for a \emph{filter (or mask)} $\mask=\{\mask_k\}_{k\in\Z}\subset\C$, the \emph{Fourier series} $\FT{h}$ is defined to be the $1$-periodic function $\FT{\mask}(\xi):=\sum_{k\in\Z}\mask_k e^{-2\pi i k\xi}$, $\xi\in\R$. Again, we abuse the ``hat'' notation, but one can easily tell the difference of Fourier coefficients $\Fcoem{f}$, Fourier transform $\FT{\gamma}$ and Fourier series $\FT{h}$ from the context.  The first equation in \eqref{eq:refinement} is said to be the \emph{refinement equation} with $\scala$ being the \emph{refinable function} associated with the \emph{refinement mask} $\maska$ (or \emph{low-pass filter} in electrical engineering). The functions $\scalb^n$ are framelet generators associated with \emph{framelet masks} (or \emph{high-pass filters}) $\maskb[n]$, $n=1,\ldots,r$, which can be derived via  extension principles \cite{DaHaRoSh2003,RoSh1997}.

In this paper, the symbols $f, g, u, q, \boldsymbol{\varphi},\boldsymbol{\psi}$ are reserved for functions defined on $\mfd$, the  symbols
$\alpha,\beta,\gamma$ are for functions on $\R$, the symbols $a,b,h$ are for filters (masks), and
$\fracoev[],\frbcoev[]{}$ in Seciton~\ref{sec:fast.algo} are for framelet coefficient sequences.
\subsection{Continuous framelets}\label{subsec:cont_framelets}

In this subsection, we define continuous framelet systems and give some equivalence conditions of a sequence of continuous framelet systems to be a sequence of tight frames in $L_2(\mfd)$.

Maggioni and Mhaskar \cite[Theorem 4.1]{MaMh2008} proved that when the associated filter function $\gamma$ has regularity depending on some constant $s_{1}$, $s_{2}>0$, the kernel
\begin{equation}\label{eq:kernel.K}
	K_{\gamma,N}(\PT{x},\PT{y}):=\sum_{\ell=0}^\infty \gamma\left(\frac{\eigvm}{N}\right) \conj{\eigfm(\PT{y})}\eigfm(\PT{x})
\end{equation}
is well-localized:
\begin{equation}\label{kernel:decay}
|K_{\gamma,N}(\PT{x},\PT{y})|\le \frac{c\: N^{s_1}}{\max\{1,(N\rho(\PT{x},\PT{y}))^{s_2}\}},
\end{equation}
where $s_1$ and $s_2$ satisfy $0<s_{1}<s_{2}$, the constant $c$ depends only on $\gamma$ and the manifold $\mfd$ itself, and $\rho:\mfd\times \mfd\rightarrow\R$ is a quasi-metric on $\mfd$. The inequality \eqref{kernel:decay} means that the kernel $K_{\gamma,N}(\cdot,\PT{y})$ is localized around a fixed $\PT{y}\in\mfd$ as a function of the first argument: the larger $N$, the more concentrated $K_{\gamma,N}(\cdot,\PT{y})$ around $\PT{y}$. This localized kernel in \eqref{eq:kernel.K} can then be used to define ``dilation''  and ``translation'' of a function on $\mfd$.

For $j\in\Z$ and $\PT{x},\PT{y}\in\mfd$, the \emph{continuous framelet elements} $\cfra(\PT{x})$ and $\cfrb{n}(\PT{x})$ on $\mfd$ at scale $j$ are the \emph{filtered Bessel kernels} (or summability kernels, reproducing kernels, Mercer kernels, see e.g. \cite{BrDiSaSlWaWo2014,MaMh2008,Zhou2003}), given
by
\begin{equation}\label{eq:intro.cfr}
\begin{aligned}
    \cfra(\PT{x}) := & K_{\FT{\scala},2^j}(\PT{x},\PT{y})=   \sum_{\ell=0}^{\infty} \FT{\scala}\left(\frac{\eigvm}{2^{j}}\right)\conj{\eigfm(\PT{y})}\eigfm(\PT{x}),&\\
    \cfrb{n}(\PT{x}) := &K_{\FT{\scalb^n},2^j}(\PT{x},\PT{y})=  \sum_{\ell=0}^{\infty} \FT{\scalb^n}\left(\frac{\eigvm}{2^{j}}\right)\conj{\eigfm(\PT{y})}\eigfm(\PT{x}),&\quad n = 1,\ldots,r.
\end{aligned}
\end{equation}
The framelet elements $\cfra(\PT{x})$ and $\cfrb{n}(\PT{x})$ correspond to the ``dilation'' operation at scale $j$ and the ``translation'' at a point $\PT{y}\in\mfd$ of wavelets in $\Rd$. The \emph{continuous framelet system} $\cfrsys(\Psi):=\cfrsys(\Psi;\mfd)$ on $\mfd$ (starting at a scale $J\in\Z$) is then a (nonhomogeneous) affine system \cite{Han2010,Han2012} given by
\begin{equation}\label{eq:wav.sys}
    \cfrsys(\Psi) = \mathsf{CFS}_J(\{\scala;\scalb^{1},\dots,\scalb^{r}\}) := \{\cfra[\ord,\PT{y}] \setsep \PT{y}\in\mfd\}\cup\{\cfrb{1},\dots,\cfrb{r} \setsep \PT{y}\in\mfd, \:j\ge \ord\}.
\end{equation}
The continuous framelet system $\cfrsys(\Psi)$ is said to be a \emph{(continuous) tight frame} for $\Lpm{2}$ if $\cfrsys(\Psi)\subset\Lpm{2}$ and if, in $L_2$ sense,
\begin{equation}\label{eq:intro.cfr.tight.f}
    f = \int_{\mfd} \InnerL{f,\cfra[\ord,\PT{y}]}\cfra[\ord,\PT{y}] \dmf{y} + \sum_{j=\ord}^{\infty}\sum_{n=1}^{r} \int_{\mfd}\InnerL{f,\cfrb{n}}\cfrb{n} \dmf{y}\quad \forall f\in \Lpm{2},
\end{equation}
or equivalently,
\begin{equation}\label{eq:intro.cfr.tight.coeff}
\norm{f}{\Lpm{2}}^{2} = \int_{\mfd} \bigl|\InnerL{f,\cfra[\ord,\PT{y}]}\bigr|^{2} \dmf{y} + \sum_{j=\ord}^{\infty}\sum_{n=1}^{r} \int_{\mfd}
\bigl|\InnerL{f,\cfrb{n}}\bigr|^{2}\dmf{y}\quad \forall f\in\Lpm{2}.
\end{equation}
The elements in $\cfrsys(\Psi)$ are said to be \emph{(continuous) tight framelets} for $L_2(\mfd)$. We also say $\cfrsys(\Psi)$ (continuous) tight framelets if no confusion arises, similar to the treatment for ``classical wavelets'', see \cite{Daubechies1992,DaHaRoSh2003}.

The following theorem gives equivalence conditions of a sequence $\{\cfrsys(\Psi)\}_{\ord=\ord[0]}^{\infty}$ of continuous framelet systems in \eqref{eq:wav.sys} to be a sequence of  tight frames for $\Lpm{2}$.

 \begin{theorem}\label{thm:cfr.tightness}
Let $J_0\in\Z$ be an integer and $\Psi:=\{\scala;\scalb^1,\ldots,\scalb^r\}\subset\Lpm[\R]{1}$ with $r\ge1$ be a set of framelet generators associated with a filter bank $\filtbk:=\{\maska; \maskb[1],\ldots, \maskb[r]\}\subset l_1(\Z)$ satisfying \eqref{eq:refinement}. Define continuous framelet system $\cfrsys(\Psi), J\ge J_0$ as in  \eqref{eq:wav.sys} with framelets $\cfra$ and $\cfrb{n}$ in \eqref{eq:intro.cfr}. Suppose $\cfra$ and $\cfrb{n}$ are functions in $\Lpm{2}$ for all $\PT{y}\in\mfd$,  $n=1,\ldots,r$, and $j\ge J_0$. Then, the following statements are equivalent.
\begin{enumerate}[{\rm(i)}]
\item The continuous framelet system $\cfrsys(\Psi)$ is a tight frame for $\Lpm{2}$ for all $\ord\ge\ord[0]$, i.e. \eqref{eq:intro.cfr.tight.f} holds for all $\ord\ge \ord[0]$.

\item  For all $f\in\Lpm{2}$, the following identities hold:
\begin{align}
          &   \lim_{j\to\infty}\normau{\int_{\mfd} \InnerL{f,\cfra}\cfra\: \dmf{y}-f}{\Lpm{2}}=0,\label{thmeq:cfr.prj.lim}\\
          &  \int_{\mfd} \InnerL{f,\cfra[j+1,\PT{y}]}\cfra[j+1,\PT{y}]\: \dmf{y}
             = \int_{\mfd} \InnerL{f,\cfra}\cfra\: \dmf{y}
             +\int_{\mfd} \sum_{n=1}^{r}\InnerL{f,\cfrb{n}}\cfrb{n}\: \dmf{y},\;\; j\ge\ord[0].\label{thmeq:cfr.pr.j.j1}
\end{align}

\item For  all $f\in\Lpm{2}$, the following identities hold:
\begin{align}
           & \lim_{j\to\infty}\int_{\mfd} \bigl|\InnerL{f,\cfra}\bigr|^{2}\: \dmf{y}   = \norm{f}{\Lpm{2}}^{2},\label{thmeq:cfr.coe.lim}\\
            & \int_{\mfd} \bigl|\InnerL{f,\cfra[j+1,\PT{y}]}\bigr|^{2}\: \dmf{y}
            = \int_{\mfd} \bigl|\InnerL{f,\cfra}\bigr|^{2}\:\dmf{y} + \int_{\mfd} \sum_{n=1}^{r}\bigl|\InnerL{f,\cfrb{n}}\bigr|^{2}\:\dmf{y},\quad j\ge\ord[0].&\label{thmeq:cfr.coe.j.j1}
\end{align}

\item The generators in $\Psi$ satisfy
\begin{align}
  & \lim_{j\to\infty}\Big|\FT{\scala}\left(\frac{\eigvm}{2^{j}}\right) \Big|= 1,\quad \ell\ge0.\label{thmeq:cfr.scala.lim}\\
            & \left|\FT{\scala}\left(\frac{\eigvm}{2^{j+1}}\right)\right|^{2}
                = \left|\FT{\scala}\left(\frac{\eigvm}{2^{j}}\right)\right|^{2} + \sum_{n=1}^{r}\left|\FT{\scalb^{n}}\left(\frac{\eigvm}{2^{j}}\right)\right|^{2},\quad \ell\ge0,\; j\ge\ord[0].\label{thmeq:cfr.scal.j.j1}
 \end{align}

\item The refinable function $\scala$ satisfies \eqref{thmeq:cfr.scala.lim}  and the filters in the filter bank $\filtbk$ satisfy
\begin{align}
&  \left|\FS{\maska}\left(\frac{\eigvm}{2^{j}}\right)\right|^{2} + \sum_{n=1}^{r} \left|\FS{\maskb[n]}\left(\frac{\eigvm}{2^{j}}\right)\right|^{2} = 1 \quad \forall \ell\in\sigma_\scala^j:=\left\{\ell\in\N_0 \setsep \FT{\scala}\left(\frac{\eigvm}{2^j}\right) \neq 0\right\}\mbox{~and~} \forall j\ge {\ord[0]}+1.
\label{thmeq:cfr.tightness.a.mask.cond}
 \end{align}
\end{enumerate}
\end{theorem}

\begin{proof}
(i)$\Longleftrightarrow$(ii). We define projections  $\cfrpra{j}$ and $\cfrprb{j}$, $n=1,\dots,r$ as
\begin{equation}\label{eqs:proj.cfr}
    \cfrpra{j}(f) := \int_{\mfd} \InnerL{f,\cfra}\cfra\: \dmf{y},\quad
    \cfrprb{j}(f) := \int_{\mfd} \InnerL{f,\cfrb{n}}\cfrb{n}\: \dmf{y},\quad f\in\Lpm{2}.
\end{equation}
Since $\cfrsys(\Psi)$ is a tight frame for $\Lpm{2}$ for all $\ord\ge\ord[0]$,
\begin{equation*}
\begin{aligned}
\label{eq:f.cfrpra.cfrprab.J}
  f = \cfrpra{\ord}(f) + \sum_{j=\ord}^{\infty}\sum_{n=1}^{r} \cfrprb{j}(f)
    = \cfrpra{\ord+1}(f) + \sum_{j=\ord+1}^{\infty}\sum_{n=1}^{r} \cfrprb{j}(f)
\end{aligned}
\end{equation*}
for all $f\in\Lpm{2}$ and for all $J\ge {\ord[0]}$.
Thus, for $\ord\ge\ord[0]$, in $L_2$ sense,
\begin{equation}\label{eq:cfr.pr.J.J1}
 \cfrpra{\ord+1}(f) =  \cfrpra{\ord}(f) + \sum_{n=1}^{r} \cfrprb{\ord}(f),
\end{equation}
which shows \eqref{thmeq:cfr.pr.j.j1}. Then, recursively using \eqref{eq:cfr.pr.J.J1} gives
\begin{equation}\label{eq:cfrpra.m1}
  \cfrpra{m+1}(f) = \cfrpra{\ord}(f) + \sum_{j=\ord}^{m}\sum_{n=1}^{r} \cfrprb{j}(f)
\end{equation}
for all $m\ge J$ and $J\ge {\ord[0]}$. Now forcing $m\to\infty$ gives, in $L_{2}$ sense
\begin{equation*}
  \lim_{m\to\infty}\cfrpra{m+1}(f) = \cfrpra{\ord}(f) + \sum_{j=\ord}^{\infty}\sum_{n=1}^{r} \cfrprb{\ord}(f)=f,
\end{equation*}
which is  \eqref{thmeq:cfr.prj.lim}. Consequently, (i)$\Longrightarrow$(ii).
Conversely, by \eqref{thmeq:cfr.pr.j.j1}, follows \eqref{eq:cfrpra.m1}.
Forcing $m\to\infty$ in \eqref{eq:cfrpra.m1} together with \eqref{thmeq:cfr.prj.lim} gives \eqref{eq:intro.cfr.tight.f}. Thus, (ii)$\Longrightarrow$(i).

(ii)$\Longleftrightarrow$(iii). The equivalence between (ii) and (iii) follows from the polarization identity.

(ii)$\Longleftrightarrow$(iv). By \eqref{eq:intro.cfr} and the orthonormality of $\eigfm$, we obtain
\begin{equation*}\label{eq:cfr.coeff}
    \InnerL{f,\cfra} = \sum_{\ell=0}^{\infty} \overline{\FT{\scala}\left(\frac{\eigvm}{2^{j}}\right)}\Fcoem{f}\:\eigfm(\PT{y}), \quad
    \InnerL{f,\cfrb{n}} = \sum_{\ell=0}^{\infty} \conj{\FT{\scalb^{n}}\left(\frac{\eigvm}{2^{j}}\right)}\Fcoem{f}\:\eigfm(\PT{y}).
\end{equation*}
This together with \eqref{eqs:proj.cfr} and \eqref{eq:intro.cfr} gives, for $j\ge\ord[0]$ and $n=1,\dots,r$, the Fourier coefficients for the projections $\cfrpra{j}(f)$ and $\cfrprb{j}(f)$:
\begin{equation}\label{eq:Fcoe.cfrpr}
    \Fcoem{\left(\cfrpra{j}(f)\right)}
    = \left|\FT{\scala}\left(\frac{\eigvm}{2^{j}}\right)\right|^{2} \Fcoem{f},\quad
    \Fcoem{\left(\cfrprb{j}(f)\right)}
    = \left|\FT{\scalb^{n}}\left(\frac{\eigvm}{2^{j}}\right)\right|^{2} \Fcoem{f},\quad  \ell\in\N_0,
\end{equation}
which implies that \eqref{thmeq:cfr.pr.j.j1} is equivalent to \eqref{thmeq:cfr.scal.j.j1} by the Riesz-Fisher theorem. On the other hand, by \eqref{eq:Fcoe.cfrpr} and Parseval's identity, we obtain
\begin{equation}\label{eq:norm.cfrpra.L2err.Fcoe}
  \normb{\cfrpra{j}(f)-f}{\Lpm{2}}^{2} = \sum_{\ell=0}^{\infty}\left(\left|\FT{\scala}\left(\frac{\eigvm}{2^{j}}\right)\right|^{2} - 1 \right)^{2} |\Fcoem{f}|^{2}.
\end{equation}
When the left-hand side of \eqref{eq:norm.cfrpra.L2err.Fcoe} tends to zero as $j\to\infty$, every term in the sum of the right-hand side in \eqref{eq:norm.cfrpra.L2err.Fcoe} must tend to zero as $j\rightarrow \infty$; i.e.  $\lim_{j\to\infty}\FT{\scala}(2^{-j}\eigvm)= 1$ for each $\ell\ge0$. Thus, \eqref{thmeq:cfr.prj.lim}$\Longrightarrow$\eqref{thmeq:cfr.scala.lim}.
Conversely, by the continuity of $\FT{\scala}$ at zero and  Lebesgue's dominated convergence theorem, we see that
 if  $\lim_{j\rightarrow\infty}\FT{\scala}(2^{-j}\eigvm)= 1$ for each $\ell\ge0$, then $\lim_{j\to\infty}\normb{\cfrpra{j}(f)-f}{\Lpm{2}}^{2}=0$. This shows \eqref{thmeq:cfr.scala.lim}$\Longrightarrow$\eqref{thmeq:cfr.prj.lim}. Thus, (ii)$\Longleftrightarrow$(iv).

(iv)$\Longleftrightarrow$(v).  By the relation in \eqref{eq:refinement}, it can be obtained that for $\ell\ge0$ and $j\ge\ord[0]$,
\begin{align*}
     \left|\FT{\scala}\left(\frac{\eigvm}{2^{j}}\right)\right|^{2} + \sum_{n=1}^{r}\left|\FT{\scalb^{n}}\left(\frac{\eigvm}{2^{j}}\right)\right|^{2}
    = \left(\left|\FT{\maska}\left(\frac{\eigvm}{2^{j+1}}\right)\right|^{2} + \sum_{n=1}^{r}\left|\FT{\maskb[n]}\left(\frac{\eigvm}{2^{j+1}}\right)\right|^{2}\right)\left|\FT{\scala}\left(\frac{\eigvm}{2^{j+1}}\right)\right|^{2}.
\end{align*}
This shows that \eqref{thmeq:cfr.scal.j.j1} is equivalent to \eqref{thmeq:cfr.tightness.a.mask.cond}. Therefore, (iv)$\Longleftrightarrow$(v).
\end{proof}
\begin{remark}{\rm
Tightness of $\cfrsys(\Psi)$ is usually proved for a fixed $J$ under some sufficient conditions that imply but are not equivalent to item (iv) or (v) of Theorem~\ref{thm:cfr.tightness}, see \cite[Theorem 3]{MhPr2004} for the case $r=1$ and $J=0$ with no filter bank associated, and \cite[Theorem 2.1]{Dong2017} for the case of $J=0$ and $r\ge1$ with filter bank associated. The characterization in Theorem~\ref{thm:cfr.tightness} gives a full picture of the relationship among the tightness of a sequence of framelet systems $\cfrsys(\Psi)$, $J\ge J_0$, the framelet  generating set $\Psi$ and  the filter bank $\filtbk$. They are the counterparts of classical tight framelets in $\R^d$, see \cite{Han2010, Han2012}.
}
\end{remark}

\begin{remark}
{\rm The statements (iv) and (v) in Theorem~\ref{thm:cfr.tightness} show that the tightness of continuous framelet system $\cfrsys(\Psi)$ can be reduced to a simple identity in \eqref{thmeq:cfr.scal.j.j1} or \eqref{thmeq:cfr.tightness.a.mask.cond}, where \eqref{thmeq:cfr.scal.j.j1} holds for any classical tight frame generated by $\Psi$ for $\Lpm[\R]{2}$ and \eqref{thmeq:cfr.tightness.a.mask.cond} holds for any filter bank with the perfect reconstruction property. This simplifies the construction of continuous tight frames on the manifold $\mfd$. On the other hand, the condition of \eqref{thmeq:cfr.tightness.a.mask.cond} is weaker than that for $\Lpm[\R]{2}$ as we do not require the downsampling condition for the filter bank, see e.g. \cite{DaHaRoSh2003,RoSh1997}. A direct consequence is that the conditions (iv) and (v) in Theorem~\ref{thm:cfr.tightness} can be easily satisfied by frequency splitting techniques when only generators or filters of band-limited functions are needed, see \cite{HaZhZh2016,HaZh2015} and the remarks following Theorem~\ref{thm:framelet.tightness} in Subsection~\ref{subsec:discreteframelets}.}
\end{remark}


In Theorem~\ref{thm:cfr.tightness}, the condition that $\cfra$ and $\cfrb{n}$ in \eqref{eq:intro.cfr} are functions in $\Lpm{2}$ is automatically satisfied from the band-limited property of $\scala$ and $\scalb^n$, i.e. $\supp\FT\scala$ and $\supp\FT{\scalb^n}$ are finite, when the summation in \eqref{eq:intro.cfr} is taken over finite terms. On the other hand, when $\scala,\scalb^n$ are not band-limited, a mild condition on the decay of $\FT\scala$ guarantees that $\cfra$ and $\cfrb{n}$ in \eqref{eq:intro.cfr} are functions in $\Lpm{2}$, which is a  consequence of
Weyl's asymptotic formula \cite{Chavel1984,Weyl1912} and Grieser's uniform bound of eigenfunctions \cite{Grieser2002} as stated  in the following lemma.

For two real sequences $\{A_{\ell}\}_{\ell=0}^{\infty}$ and $\{B_{\ell}\}_{\ell=0}^{\infty}$, the symbol $A_{\ell}\asymp B_{\ell}$ means that there exist positive constants $c,c^{\prime}$ independent of $\ell$ such that $c^{\prime}B_{\ell} \le A_{\ell}\le c B_{\ell}$ for all $\ell\ge0$.
\begin{lemma}\label{lem:estimate.eigvm.eigfm} Let $d\ge2$ and $\mfd$ be a $d$-dimensional smooth and compact Riemannian manifold with smooth boundary. Let
$\{(\eigfm,\eigvm)\}_{\ell=0}^\infty$ be the orthonormal eigen-pairs of the {Laplace-Beltrami operator} $\LBm$ on $\mfd$, i.e. $\LBm\eigfm = -\eigvm^{2}\: \eigfm,\; \ell\ge0$, with  $\eigfm[0]\equiv1$. Then,
\begin{equation*}
    \eigvm \asymp \ell^{\frac{1}{d}},\quad \norm{\eigfm}{\Lpm{\infty}} \le c_1\:|\eigvm|^{\frac{d-1}{2}},\quad \ell\ge0,
\end{equation*}
where the constant $c_1$ depends only on the dimension $d$.
\end{lemma}

Lemma~\ref{lem:estimate.eigvm.eigfm} implies the following result (see \cite{Dong2017}) which shows that for any $\PT{y}\in\mfd$, the continuous framelets $\cfra$ and $\cfrb{n}$, $n=1,\dots,r$, are in $\Lpm{2}$ under a mild decay assumption on $\FT{\scala}$.
\begin{proposition}
\label{prop:cfr.L2}
Let the conditions of Lemma~{\rm\ref{lem:estimate.eigvm.eigfm}} be satisfied. Let $\Psi:=\{\scala;\scalb^1,\ldots,\scalb^r\}\subset \Lpm[\R]{1}$ with $r\ge1$ be a set of framelet generators associated with a filter bank $\filtbk:=\{\maska; \maskb[1],\ldots, \maskb[r]\}\subset l_1(\Z)$ satisfying \eqref{eq:refinement}. Let $\cfra$ and $\cfrb{n}$ be the continuous framelets given in \eqref{eq:intro.cfr}. Suppose $s>d-1/2$ and
\begin{equation}
\label{eq:decay:phi}
|\FT{\scala}(\xi)|\le c_{0}\: (1+|\xi|)^{-s}\quad  \forall \xi\in\R.
\end{equation}
Then, for any $j\in\Z$,
\[
\sup_{\PT{y}\in\mfd}\|\cfra\|_{\Lpm{2}}<\infty \mbox{~and~}
\sup_{\PT{y}\in\mfd}\|\cfrb{n}\|_{\Lpm{2}}<\infty,\quad n = 1,\ldots,r.
\]
\end{proposition}
\begin{proof}
Fix $j\in\Z$. By Parseval's identity and the estimates in Lemma~\ref{lem:estimate.eigvm.eigfm}, the squared $L_{2}$-norm of $\cfra$ is
\begin{align*}
  \normb{\cfra}{\Lpm{2}}^{2}
  = \sum_{\ell=0}^{\infty} \left|\FT{\scala}\left(\frac{\eigvm}{2^{j}}\right)\right|^{2}\left|\conj{\eigfm(\PT{y})}\right|^{2}
  &\le (c_1c_{0})^2\sum_{\ell=0}^{\infty} \left(1+\Bigl|\frac{\eigvm}{2^{j}}\Bigr|\right)^{-2s}\times\eigvm^{{d-1}}\\
  &\le \tilde{c} \sum_{\ell=0,\eigvm\neq0}^{\infty} \eigvm^{-(2s-(d-1))}
  \le \tilde{c}\sum_{\ell=1}^{\infty} \ell^{-\frac{2s-(d-1)}{d}}<\infty,
\end{align*}
where the last inequality follows from the assumption $s>d-1/2$.
Since $\{\maska;\maskb[1],\ldots,\maskb[r]\}\subset l_1(\Z)$, the Fourier series $\FT\maska$, $\FT{\maskb[1]}$, $\ldots$, $\FT{\maskb[r]}$ are all bounded Fourier series. By the  relations in {\eqref{eq:refinement}}, all $\FT{\scalb^n}$  have the same decay property as $\FT\scala$ in \eqref{eq:decay:phi}. The finiteness for the $L_2$-norm of $\cfrb{n}$ then follows from the same argument for $\cfra$ as above.
\end{proof}

\subsection{Semi-discrete framelets}\label{subsec:discreteframelets}

In order to efficiently process a data set on a manifold, one needs the discrete version of the continuous framelets in \eqref{eq:intro.cfr}. A natural way to discretize the continuous framelets on $\mfd$ is to use  quadrature rules (for numerical integration). In this subsection, we show how to use quadrature rules to discretize the continous framelets in \eqref{eq:intro.cfr}.

Let
\[
\QN:=\QN^{(j)}:=\{(\wN,\pN) \in\R\times\mfd \setsep k=0,\dots,N_{j}\}
\]
be a set of pairs at  scale $j$ with $N_j$ weights $\wN\in\R$ and $N_j$ points $\pN\in\mfd$. When $\QN$ is used for numerical integration on $\mfd$, we say $\QN$ a \emph{quadrature rule} on $\mfd$. We use the quadrature rules $\QN$ and $\QN[{N_{j+1}}]$ to discretize the integrals for the continuous framelets $\cfra$ and $\cfrb{n}$ as functions of $\PT{y}$ on $\mfd$ in \eqref{eq:intro.cfr.tight.f}.
The \emph{(semi-discrete) framelets} $\fra(\PT{x})$ and $\frb{n}(\PT{x})$ (with abuse of notation) at scale $j$ are then defined as
\begin{equation}
\label{eq:intro.fra.frb}
\begin{aligned}
    \fra(\PT{x}) &:= \sqrt{\wN}\:\cfra[j,\pN](\PT{x}) = \sqrt{\wN}\sum_{\ell=0}^{\infty} \FT{\scala}\left(\frac{\eigvm}{2^{j}}\right)\conj{\eigfm(\pN)}\eigfm(\PT{x}),\\
    \frb{n}(\PT{x}) &:= \sqrt{\wN[j+1,k]}\:\cfrb[j,{\pN[j+1,k]}]{n}(\PT{x}) = \sqrt{\wN[j+1,k]}\sum_{\ell=0}^{\infty} \FT{\scalb^{n}}\left(\frac{\eigvm}{2^{j}}\right)\conj{\eigfm(\pN[j+1,k])}\eigfm(\PT{x}),\quad n =1,\ldots,r.
\end{aligned}
\end{equation}
Here, the weights $\wN$ in \eqref{eq:intro.fra.frb} need not be non-negative. The square roots of weights are purely needed to satisfy the tightness of the framelets.
The discretization of the integral for $\cfrb{n}$ uses the nodes from $\QN[N_{j+1}]$ as $\cfrb{n}$ is in the scale $j+1$, which can be understood from the point of view of multiresolution analysis. This will be clear later when we discuss the band-limited property of $\scalb^n$.

Let $\QQ:=\{\QN\}_{j\ge J}$. The \emph{(semi-discrete) framelet system} $\frsys(\Psi,\QQ):=\frsys(\Psi,\QQ;\mfd)$ on $\mfd$ (starting at a scale $J\in\Z$) is a (nonhomogeneous) affine system defined to be
\begin{equation}\label{eq:intro.frsys}
    \frsys(\Psi,\QQ) := \frsys(\Psi,\QQ;\mfd) :=\{\fra[\ord,k] \setsep k=0,\dots,N_{\ord}\}\cup \{\frb{} \setsep k=0,\dots,N_{j+1}, \:j\ge \ord\}.
\end{equation}
The framelet system $\frsys(\Psi,\QQ)$ is said to be a \emph{(semi-discrete) tight frame} for $\Lpm{2}$ if $\frsys(\Psi,\QQ)\subset \Lpm{2}$ and if, in $L_2$ sense,
\begin{equation}\label{eq:intro.fr.tight.f}
    f = \sum_{k=0}^{N_J} \InnerL{f,\fra[\ord,k]}\fra[\ord,k]  + \sum_{j=\ord}^{\infty} \sum_{k=0}^{N_{j+1}}\sum_{n=1}^{r}\InnerL{f,\frb{n}}\frb{n} \quad \forall f\in \Lpm{2},
\end{equation}
or equivalently,
\begin{equation*}
\label{thmeq:framelet.tightness}
 \norm{f}{\Lpm{2}}^{2} = \sum_{k=0}^{N_{\ord}} \bigl|\InnerL{f,\fra[\ord,k]}\bigr|^{2} + \sum_{j=\ord}^{\infty} \sum_{k=0}^{N_{j+1}} \sum_{n=1}^{r}\bigl|\InnerL{f,\frb{n}}\bigr|^{2}\quad \forall f\in\Lpm{2}.
\end{equation*}
The elements in $\frsys(\Psi,\QQ)$ are then said to be \emph{(semi-discrete) tight framelets} for $L_2(\mfd)$. We also say $\frsys(\Psi,\QQ)$ (semi-discrete) tight framelets.

The following theorem gives equivalence conditions of a sequence $\{\frsys(\Psi,\QQ)\}_{\ord=\ord[0]}^{\infty}$ of (semi-discrete) framelet systems  in \eqref{eq:intro.fr.tight.f} to be a sequence of  tight frames for $\Lpm{2}$. The equivalence relations lead to framelet transforms on a manifold and a way to  constructing filter banks for a tight framelet system.

\begin{theorem}
\label{thm:framelet.tightness}
Let $J_0\in\Z$ be an integer and $\Psi:=\{\scala;\scalb^1,\ldots,\scalb^r\}\subset \Lpm[\R]{1}$ with $r\ge1$ be a set of  framelet generators associated with a filter bank $\filtbk:=\{\maska; \maskb[1],\ldots, \maskb[r]\}\subset l_1(\Z)$ satisfying \eqref{eq:refinement}. Let $\QQ=\{\QN\}_{j\ge {\ord[0]}}$  be a sequence of quadrature rules  $\QN:=\QN^{(j)}:=\{(\wN,\pN)\in\R\times\mfd \setsep k=0,\dots,N_{j}\}$.
 Define (semi-discrete) framelet system $\frsys(\Psi,\QQ)=\frsys(\Psi,\QQ;\mfd)$, $J\ge J_0$ as in  \eqref{eq:intro.frsys}
with framelets $\fra$ and $\frb{n}$ given by \eqref{eq:intro.fra.frb}. Suppose elements in $\frsys(\Psi,\QQ)$ are all functions in $\Lpm{2}$.  Then, the following statements are equivalent.
 \begin{enumerate}[{\rm(i)}]
 \item   The  framelet system $\frsys(\Psi,\QQ)$ is a tight frame for $\Lpm{2}$ for any $\ord\ge\ord[0]$, i.e. \eqref{eq:intro.fr.tight.f} holds for all $\ord\ge \ord[0]$.

\item For all $f\in\Lpm{2}$, the following identities hold:
\begin{align}
            & \lim_{j\to\infty}\normB{\sum_{k=0}^{N_{j}} \InnerL{f,\fra}\fra-f}{\Lpm{2}}=0,\label{thmeq:fr.prj.lim}\\
            & \sum_{k=0}^{N_{j+1}} \InnerL{f,\fra[j+1,k]}\fra[j+1,k] = \sum_{k=0}^{N_{j}} \InnerL{f,\fra}\fra + \sum_{k=0}^{N_{j+1}} \sum_{n=1}^{r}\InnerL{f,\frb{n}}\frb{n},\quad j\ge\ord[0].\label{thmeq:framelet.pr.j.j1}
 \end{align}

 \item For all $f\in\Lpm{2}$, the following identities hold:
\begin{align}
            & \lim_{j\to\infty} \sum_{k=0}^{N_{j}} \bigl|\InnerL{f,\fra}\bigr|^{2}=\norm{f}{\Lpm{2}}^{2},\label{thmeq:framelet.coe.lim}\\
            & \sum_{k=0}^{N_{j+1}} \bigl|\InnerL{f,\fra[j+1,k]}\bigr|^{2}=\sum_{k=0}^{N_{j}} \bigl|\InnerL{f,\fra}\bigr|^{2} + \sum_{k=0}^{N_{j+1}}\sum_{n=1}^{r}\bigl|\InnerL{f,\frb{n}}\bigr|^{2} ,\quad j\ge\ord[0].\label{thmeq:framelet.coe.j.j1}
\end{align}

\item The generators in $\Psi$ and the sequence of sets $\QN$ satisfy
\begin{align}
  & \lim_{j\to\infty}\overline{\FT{\scala}\left(\frac{\eigvm}{2^{j}}\right)}{\FT{\scala}}\left(\frac{\eigvm[\ell']}{2^{j}}\right)\QU{\ell'}(\QN) = \delta_{\ell,\ell'},\label{thmeq:fr.scala.lim}\\
            &                 \overline{\FT{\scala}\left(\frac{\eigvm}{2^{j}}\right)}{\FT{\scala}}\left(\frac{\eigvm[\ell']}{2^{j}}\right)\QU{\ell'}(\QN)=\left[\overline{\FT{\scala}\left(\frac{\eigvm}{2^{j+1}}\right)}{\FT{\scala}}\left(\frac{\eigvm[\ell']}{2^{j+1}}\right)-\sum_{n=1}^r \overline{\FT{\scalb^n}\left(\frac{\eigvm}{2^{j}}\right)}{\FT{\scalb^n}}\left(\frac{\eigvm[\ell']}{2^{j}}\right)\right]\QU{\ell'}(\QN[N_{j+1}])\label{thmeq:fr.scal.j.j1}
 \end{align}
for all $\ell,\ell'\ge0$ and $j\ge\ord[0]$,
 where
 \begin{equation}\label{eq:sum:wt:u}
 \QU{\ell'}(\QN):=\sum_{k=0}^{N_j}\wN\eigfm(\pN)\overline{\eigfm[\ell'](\pN)}.
\end{equation}

\item The refinable function $\scala$, the filters in the filter bank $\filtbk$, and the sequence of quadrature rules $\QN$ satisfy \eqref{thmeq:fr.scala.lim} and
\begin{align}
&\left[\overline{\FT{\maska}\left(\frac{\eigvm}{2^{j}}\right)}{\FT{\maska}}\left(\frac{\eigvm[\ell']}{2^{j}}\right)\QU{\ell'}(\QN[N_{j-1}])
+\sum_{n=1}^r\overline{\FT{\maskb[n]}\left(\frac{\eigvm}{2^{j}}\right)}{\FT{\maskb[n]}}\left(\frac{\eigvm[\ell']}{2^{j}}\right)\QU{\ell'}(\QN[N_{j}])\right] = \QU{\ell'}(\QN[N_{j}]),  \quad \forall (\ell,\ell')\in \sigma_{\scala,\overline{\scala}}^j, \label{thmeq:fr.tightness.a.mask.cond}
 \end{align}
 for all $j\ge {\ord[0]}+1$, where
\begin{equation}\label{eq:sigma.scala}
 \sigma_{\scala,\overline{\scala}}^j:=\left\{(\ell,\ell')\in\N_0\times \N_0 \setsep \overline{\FT{\scala}\left(\frac{\eigvm}{2^{j}}\right)}{\FT{\scala}}\left(\frac{\eigvm[\ell']}{2^{j}}\right)\neq0\right\}.
\end{equation}
\end{enumerate}
In particular, if for all $j\ge J_0$, the sum $\QU{\ell'}(\QN[N_{j}])$ satisfies
\begin{equation}\label{def:fmtQ}
\QU{\ell'}(\QN[N_{j}]) = \delta_{\ell,\ell'}\quad \forall (\ell,\ell')\in \sigma_{\scala,\overline{\scala}}^j,
\end{equation}
then the above items {\rm(iv)} and {\rm(v)} reduce to the items {\rm(iv)} and {\rm(v)} in Theorem~{\rm\ref{thm:cfr.tightness}} respectively.
\end{theorem}

\begin{proof}
We skip the proofs of equivalence among the statements (i) -- (iii), which are similar to those of Theorem~\ref{thm:cfr.tightness}, and only show the equivalence among the statements (iii) -- (v) as follows.

(iii) $\Longleftrightarrow$ (iv). For $f\in\Lpm{2}$, by the formulas in \eqref{eq:intro.fra.frb} and the orthonormality of $\eigfm$, we obtain
\begin{equation}\label{eq:fr.coeff}
    \InnerL{f,\fra} = \sqrt{\wN}\sum_{\ell=0}^{\infty} \overline{\FT{\scala}\left(\frac{\eigvm}{2^{j}}\right)}\Fcoem{f}\:\eigfm(\pN[j,k]), \quad
    \InnerL{f,\frb{n}} = \sqrt{\wN[j+1,k]}\sum_{\ell=0}^{\infty} \conj{\FT{\scalb^{n}}\left(\frac{\eigvm}{2^{j}}\right)}\Fcoem{f}\:\eigfm(\pN[j+1,k]).
\end{equation}
It then follows
\[
\begin{aligned}
\sum_{k=0}^{N_j}  \bigl|\InnerL{f,\fra[j,k]}\bigr|^{2}
&=\sum_{k=0}^{N_j} \wN\left|\sum_{\ell=0}^{\infty} \overline{\FT{\scala}\left(\frac{\eigvm}{2^{j}}\right)}\Fcoem{f}\:\eigfm(\pN[j,k])\right|^2 \\
&=\sum_{\ell=0}^\infty\sum_{\ell'=0}^\infty
\Fcoem{f}\overline{\Fcoem[\ell']{f}}\;\overline{\FT{\scala}\left(\frac{\eigvm}{2^{j}}\right)}{\FT{\scala}}\left(\frac{\eigvm[\ell']}{2^{j}}\right)\sum_{k=0}^{N_j}\wN\eigfm(\pN)\overline{\eigfm[\ell'](\pN)}\\
&=\sum_{\ell=0}^\infty\sum_{\ell'=0}^\infty\Fcoem{f}\overline{\Fcoem[\ell']{f}}\;
\overline{\FT{\scala}\left(\frac{\eigvm}{2^{j}}\right)}{\FT{\scala}}\left(\frac{\eigvm[\ell']}{2^{j}}\right)\QU{\ell'}(\QN)\\
&=\sum_{\ell=0}^{\infty} |\Fcoem{f}|^2 \left|\FT{\scala}\left(\frac{\eigvm}{2^{j}}\right)\right|^2\QU{\ell}(\QN)
+\sum_{\ell=0}^{\infty}\sum_{\ell'=0,\ell' \neq \ell}^\infty \Fcoem{f}\overline{\Fcoem[\ell']{f}}\:\overline{\FT{\scala}\left(\frac{\eigvm}{2^{j}}\right)}{\FT{\scala}}\left(\frac{\eigvm[\ell']}{2^{j}}\right)\QU{\ell'}(\QN) .
\end{aligned}
\]
This is true for all $f\in\Lpm{2}$, which gives the equivalence between \eqref{thmeq:framelet.coe.lim} and \eqref{thmeq:fr.scala.lim}.
On the other hand, from \eqref{eq:fr.coeff}, we observe that the formula \eqref{thmeq:framelet.coe.j.j1} can be rewritten as
\[
\begin{aligned}
&\sum_{\ell=0}^{\infty}\sum_{\ell'=0}^\infty \Fcoem{f}\overline{\Fcoem[\ell']{f}}\:\overline{\FT{\scala}\left(\frac{\eigvm}{2^{j+1}}\right)}{\FT{\scala}}\left(\frac{\eigvm[\ell']}{2^{j+1}}\right)\QU{\ell'}(\QN[N_{j+1}])\\
=&\sum_{\ell=0}^{\infty}\sum_{\ell'=0}^\infty \Fcoem{f}\overline{\Fcoem[\ell']{f}}
\left[\overline{\FT{\scala}\left(\frac{\eigvm}{2^{j}}\right)}{\FT{\scala}}\left(\frac{\eigvm[\ell']}{2^{j}}\right)
\QU{\ell'}(\QN)+\sum_{n=1}^r\overline{\FT{\scalb^n}\left(\frac{\eigvm}{2^{j}}\right)}{\FT{\scalb^n}}\left(\frac{\eigvm[\ell']}{2^{j}}\right)\QU{\ell'}(\QN[N_{j+1}])
\right]
  \quad \forall f\in\Lpm{2},
\end{aligned}
\]
which is equivalent to  \eqref{thmeq:fr.scal.j.j1}.

(iv) $\Longleftrightarrow$ (v). By  \eqref{eq:refinement}, we have
\[
\begin{aligned}
&\overline{\FT{\scala}\left(\frac{\eigvm}{2^{j-1}}\right)}{\FT{\scala}}\left(\frac{\eigvm[\ell']}{2^{j-1}}\right)\QU{\ell'}(\QN[N_{j-1}])
+\sum_{n=1}^r\overline{\FT{\scalb^n}\left(\frac{\eigvm}{2^{j-1}}\right)}{\FT{\scalb^n}}\left(\frac{\eigvm[\ell']}{2^{j-1}}\right)\QU{\ell'}(\QN[N_{j}])
\\=&
\left[\overline{\FT{\maska}\left(\frac{\eigvm}{2^{j}}\right)}{\FT{\maska}}\left(\frac{\eigvm[\ell']}{2^{j}}\right)\QU{\ell'}(\QN[N_{j-1}])
+\sum_{n=1}^r\overline{\FT{\maskb[n]}\left(\frac{\eigvm}{2^{j}}\right)}{\FT{\maskb[n]}}\left(\frac{\eigvm[\ell']}{2^{j}}\right)\QU{\ell'}(\QN[N_{j}])\right]
\overline{\FT{\scala}\left(\frac{\eigvm}{2^{j}}\right)}{\FT{\scala}}\left(\frac{\eigvm[\ell']}{2^{j}}\right),
\end{aligned}
\]
which implies \eqref{thmeq:fr.scal.j.j1}$\Longleftrightarrow$\eqref{thmeq:fr.tightness.a.mask.cond} and thus proves the equivalence  between (iv) and (v).

In particular, if \eqref{def:fmtQ} is satisfied, then in view of $\sigma^{j}_{\scala,\conj{\scala}}\subset\sigma^{j+1}_{\scala,\conj{\scala}}$ which is due to  the refinement relation in \eqref{eq:refinement}, we see that  \eqref{thmeq:fr.scala.lim}, \eqref{thmeq:fr.scal.j.j1} and \eqref{thmeq:fr.tightness.a.mask.cond} reduce to \eqref{thmeq:cfr.scala.lim}, \eqref{thmeq:cfr.scal.j.j1} and  \eqref{thmeq:cfr.tightness.a.mask.cond} in Theorem~\ref{thm:cfr.tightness} respectively.  We are done.
\end{proof}

\begin{remark}[Unitary Extension Principle]
{\rm
The items (iv) and (v) in Theorem~\ref{thm:framelet.tightness} can be regarded as the \emph{unitary extension principle (UEP)} for  $\Lpm{2}$. In $\Lpm[\R]{2}$, the filter bank $\filtbk=\{\maska;\maskb[1],\ldots,\maskb[r]\}$ associated with $\Psi=\{\scala;\scalb^1,\ldots,\scalb^r\}\subset \Lpm[\R]{2}$ by \eqref{eq:refinement} is said to satisfy the \emph{UEP} (see \cite{DaHaRoSh2003,RoSh1997}) for $\Lpm[\R]{2}$ if for $a.e.\; \xi\in\R$,
\begin{subequations}
\label{eq:UEP}
\begin{align}
&|\FS\maska(\xi)|^2+\sum_{n=1}^r|\FS\maskb(\xi)|^2 = 1,
\label{eq:UEP1}\\
&\overline{\FS\maska(\xi)}\FS\maska\left(\xi+\frac{1}{2}\right) +\sum_{n=1}^r\overline{\FS\maskb(\xi)}\FS\maskb\left(\xi+\frac{1}{2}\right) =0. \label{eq:UEP2}
\end{align}
\end{subequations}
The UEP conditions in \eqref{eq:UEP} together with a decay condition on $\FT{\scala}$ imply the tightness of a framelet system generated from $\Psi$ through dilation and translation in $\Lpm[\R]{2}$, see e.g. \cite{Daubechies1992}. By Theorem~\ref{thm:cfr.tightness}, only the condition \eqref{eq:UEP1} is needed to construct continuous tight frame $\cfrsys(\Psi;\mfd)$ for $\Lpm{2}$. To ensure the tightness of the semi-discrete tight framelet system $\frsys(\Psi,\QQ)$ in $\Lpm{2}$, the condition \eqref{thmeq:fr.tightness.a.mask.cond} is needed. This can be viewed as a generalization of UEP on the manifold $\mfd$. The condition \eqref{thmeq:fr.tightness.a.mask.cond} seems more complicated than those in \eqref{eq:UEP}. However, \eqref{thmeq:fr.tightness.a.mask.cond} brings more flexibility in practice for the construction of semi-discrete tight frames for $\Lpm{2}$ as will be discussed below.}
\end{remark}

\begin{remark}[Quadrature Rules]
{\rm  Theorem~\ref{thm:framelet.tightness} provides a natural connection to the design of polynomial-exact quadrature rules on $\mfd$. It
shows that using suitable quadrature rules on $\mfd$ is critical to the tightness and the multiresolution structure for framelets $\frsys(\Psi,\QQ)$.
The sum $\QU{\ell'}(\QN)$  in \eqref{eq:sum:wt:u}
 is a discrete version of the integral of the product of $\eigfm$ and $\eigfm[\ell']$ by the quadrature rule $\QN$.
Suppose the refinable function is normalized so that $\FT{\scala}(0)=1$. Then, by the orthonormality of the eigenfunctions $\eigfm$,  the formula \eqref{thmeq:fr.scala.lim} is saying that the error of the numerical integration approximated by the framelet quadrature rule $\QN$ converges to zero as $j\to\infty$, that is,
\[
\lim_{j\rightarrow\infty}\QU{\ell'}(\QN)=\lim_{j\rightarrow\infty} \sum_{k=0}^{N_j}\eigfm(\pN)\overline{\eigfm[\ell'](\pN)} = \int_\mfd \eigfm(\PT{x})\overline{\eigfm[\ell'](\PT{x})}d\mu(\PT{x})= \InnerL{\eigfm,\eigfm[\ell']} = \delta_{\ell,\ell'}.
\]
This is satisfied  by  most of the quadrature rules, for example, QMC designs on the sphere \cite{BrDiSaSlWaWo2014}, lattice rules and low-discrepancy points on the unit cube \cite{DiKuSl2013}.
}
\end{remark}

\begin{remark}{\rm
For simplicity, one may consider using the same quadrature rule for all scales in practice \cite{Dong2017}, i.e. $\QN \equiv \QN[N]$ for all $j$, the equations \eqref{thmeq:fr.scal.j.j1} and \eqref{thmeq:fr.tightness.a.mask.cond} are simplified without the term  $\QU{\ell'}$. This, however, leads to that the data complexity (or the redundancy rate) increases exponentially in the level of decomposition and is thus not desirable.
}
\end{remark}

We next discuss how to achieve condition in \eqref{def:fmtQ}. For $n\in\No$, the space $\polyspm:=\spann\{\eigfm, \overline{\eigfm} \setsep \eigvm\le n\}$ is said to be the \emph{(orthogonal diffusion) polynomial space} of degree $n$ on $\mfd$ and an element of $\polyspm$ is said to be a polynomial of degree $n$.
In Lemma~\ref{lem:estimate.eigvm.eigfm}, Corollary~\ref{cor:framelet.tightness2}, and Theorem~\ref{thm:dec:rec} below, we assume that the product of two polynomials is still a polynomial, that is, there exists a (minimal) integer $c\ge 2$ such that
\begin{equation}
\label{eq:prodAssump}
q_1q_2\in\Pi_{c\cdot n}\quad \forall q_1,q_2\in\Pi_n.
\end{equation}
This assumption holds true for a general compact Riemannian manifold when the orthonormal eigen-pair $\{(\eigfm,\eigvm)\}_{\ell=0}^\infty$ is of a certain operator, such as the Laplace-Beltrami operator, see e.g. \cite[Theorem~A.1]{FiMh2011}. When $\mfd$ is the unit sphere $\sph{d}$ or the  torus $\torus{d}$ for $d\ge1$, the assumption of \eqref{eq:prodAssump} holds with $c=2$ for the orthonormal eigen-pair of the Laplace-Beltrami operator.

For $n\ge0$, a quadrature rule $\QN[N,n]:=\QN[N]:=\{(\wH[k],\pH{k})\}_{k=0}^{N}$ on $\mfd$ is said to be a \emph{polynomial-exact quadrature rule} of degree $n$ if
\begin{equation}
\label{eq:QN}
    \int_{\mfd}q(\PT{x})\dmf{x} =  \sum_{k=0}^{N} \wH[k]\: q(\pH{k})		\quad \forall q\in\polyspm[n].
\end{equation}
Here, we use $\QN[N,n]$ to emphasize the degree $n$ of the exactness of $\QN[N]$.
Since $\Pi_n$ has the finite dimension, the quadrature rules $\QH[N,n]$ can be computed and pre-designed. For example, \cite{MhNaWa2001,SlWo2004} give the polynomial-exact quadrature rules on the two-dimensional sphere $\sph{2}$. For polynomial-exact quadrature rules on other manifolds, refer to e.g. \cite{Br_etal2014,HaSa2004,KoRe2015}.

The following lemma shows that if the generators of $\Psi$ are band-limited functions, the condition \eqref{def:fmtQ} can be easily satisfied.
\begin{lemma}
\label{lem:UQ}
Suppose \eqref{eq:prodAssump} holds. Let $\scala\in\Lpm[\R]{1}$ be a band-limited function such that $\supp\FT{\scala}\subseteq[0,1/c]$ with $c\ge2$ the integer in \eqref{eq:prodAssump}. Let $j\in\Z$ and $\QN=\{(\wN,\pN)\in\R\times\mfd \setsep k=0,\dots,N_{j}\}$ be a polynomial-exact quadrature rule  of degree $2^{j}$. Then $\QN$ satisfies \eqref{def:fmtQ}.
\end{lemma}
\begin{proof}
Since $\supp\FT{\scala}\subseteq[0,1/c]$, we obtain by \eqref{eq:sigma.scala}
\[
\sigma_{\scala,\overline{\scala}}^{j}\subseteq \{(\ell,\ell')\setsep (\eigvm,\eigvm[\ell'])\in[0,2^{j}/c)\times[0,2^{j}/c).
\]
This together with the assumption in \eqref{eq:prodAssump} gives $\eigfm\overline{\eigfm[\ell']}\in\Pi_{2^j}$ for $(\ell,\ell')\in \sigma_{\scala,\overline{\scala}}^{j}$. Consequently,
by the orthonormality of $\{\eigfm\}_{\ell=0}^\infty$ and that $\QN$ is a quadrature rule of degree $2^j$, for all $(\ell,\ell')\in \sigma_{\scala,\overline{\scala}}^{j}$, we obtain
\[
 \QU{\ell'}(\QN)=\sum_{k=0}^{N_j}\wN\:\eigfm(\pN)\overline{\eigfm[\ell'](\pN)}=\int_\mfd \eigfm(\PT{x})\overline{\eigfm[\ell'](\PT{x})} d\memf(\PT{x}) = \delta_{\ell,\ell'},
\]
which is \eqref{def:fmtQ}.
\end{proof}

The following corollary, which is an immediate consequence of Theorem~\ref{thm:framelet.tightness} and Lemma~\ref{lem:UQ},  shows that the tightness of a sequence of semi-discrete framelet systems $\frsys(\Psi,\QQ)$, $\ord\ge\ord[0]$ is equivalent to that of the corresponding sequence of continuous framelet systems $\cfrsys(\Psi)$, $\ord\ge\ord[0]$ if the quadrature rule $\QN$, $j\ge\ord[0]$ for $\frsys(\Psi,\QQ)$ is exact for polynomials of degree $2^{j}$.
\begin{corollary}\label{cor:framelet.tightness2}Let $J_0\in\Z$ be an integer and $\Psi:=\{\scala;\scalb^1,\ldots,\scalb^r\}\subset\Lpm[\R]{1}$ with $r\ge1$ a set of band-limited functions associated with a filter bank $\filtbk:=\{\maska; \maskb[1],\ldots, \maskb[r]\}\subset l_1(\Z)$ satisfying \eqref{eq:refinement}. Suppose that \eqref{eq:prodAssump} holds,   $\supp\FT{\scala}\subseteq[0,1/c]$ with $c\ge2$ the integer in \eqref{eq:prodAssump},  and $\QN$ is a (polynomial-exact) quadrature rule of degree $2^j$. Let $\QQ:=\{\QN\}_{j\ge {\ord[0]}}$  and define continuous framelet system $\cfrsys(\Psi), J\ge J_0$ as in  \eqref{eq:wav.sys} and
  semi-discrete framelet system $\frsys(\Psi,\QQ)$, $J\ge J_0$ as in  \eqref{eq:intro.frsys}.
Then, the framelet system $\frsys(\Psi,\QQ)$ is a tight frame for $\Lpm{2}$ for all $\ord\ge \ord[0]$ if and only if the   framelet system $\cfrsys(\Psi)$ is a tight frame for $\Lpm{2}$ for all $\ord\ge \ord[0]$.
 \end{corollary}


\section{Fast framelet filter bank transforms on $\mfd$}\label{sec:fast.algo}

 By \eqref{eq:intro.fra.frb} and \eqref{eq:intro.cfr}, the framelet $\fra$ in a framelet system $\frsys(\Psi;\QQ)$ can be written as a constant multiple of the kernel in \eqref{kernel:decay}: $\fra=\sqrt{\wN}K_{\FT{\scala},2^j}(\cdot,\pN)$. The $\fra$ is thus well-localized, concentrated at $\pN$ when $j$ is sufficiently large (see  Figure~\ref{fig:fr.S2}). As the convolution of a function $g$ in $\Lpm[\R]{1}$ with the delta function $\boldsymbol{\delta}$ which recovers $g$, the inner product $\InnerL{f,\fra[\ord,k]}$ of the framelet coefficient approximates the function value $f(\pN[\ord,k])$ as level $\ord$ is sufficiently high. In practice, we can thus regard the function values $f(\pN[\ord,k])$, $k = 0,\ldots,N_{\ord}$ on the manifold as the values of the framelet coefficients $\InnerL{f,\fra[\ord,k]}$, $k = 0,\ldots, N_{\ord}$ at scale $\ord$.

In this section, we discuss the multi-level framelet filter bank transforms
associated with a sequence of tight frames $\frsys(\Psi,\QQ)$ for $\Lpm{2}$. The transforms include the \emph{decomposition} and the \emph{reconstruction}: the decomposition of $\fracoev=(\fracoev[j,k])_{k=0}^{N_j} = (\InnerL{f,\fra[j,k]})_{k=0}^{N_j}$ into a coarse scale \emph{approximation coefficient sequence} $\fracoev[j-1]=(\InnerL{f,\fra[j-1,k]})_{k=0}^{N_{j-1}}$ and into the coarse scale \emph{detail coefficient sequences} $\frbcoev[j-1]{n}= (\frbcoev[j-1,k]{n} )_{k=0}^{N_j} = (\InnerL{f,\frb[j-1,k]{n}})_{k=0}^{N_j}$, $n=1,\ldots,r$, and  the reconstruction of  $\fracoev$, an inverse process, from the coarse scale approximations and details to fine scales.
We show that the decomposition and reconstruction algorithms for the framelet filter bank transforms can be implemented based on discrete Fourier transforms on $\mfd$.
Using fast discrete Fourier transforms (FFTs) on $\mfd$, we are able to develop fast algorithmic realizations for the multi-level framelet filter bank transforms  ({\fmt} algorithms).

\subsection{Multi-level framelet filter bank transforms}
\label{sec:decomp.reconstr}

The {\fmt} algorithms use convolution, downsampling and upsampling for data sequences on $\mfd$, as we introduce now.

Let $\{\QN\}_{j={\ord[0]}}^{\infty}$ be a sequence of quadrature rules on $\mfd$ with $\QN=\{(\wN,\pN) \in\R\times\mfd \setsep k=0,\dots,N_{j}\}$ a polynomial-exact quadrature rule of degree $2^j$, i.e. \eqref{eq:QN} holds with $\QN[N]$ replaced by $\QN$.
For an integer $N\in\N_0$, we denote by $l(N)$ the set of sequences supported on $[0,N]$. Let $\Lambda_j:=\dim \Pi_{2^j/c} = \#\{\ell\in\N_0: \eigvm\le 2^{j}/c\}$ with $c\ge2$ the minimal integer in \eqref{eq:prodAssump}. The following transforms (operators or operations) between sequences in $l(\Lambda_j)$ and sequences in $l(N_j)$ play an important role in describing and implementing the {\fmt} algorithms.

For $j\in\N_0$, the \emph{discrete Fourier transform} $\fft[j]: l(\Lambda_j)\rightarrow l(N_{j})$ for a sequence $\mathsf{\tilde{c}}=(\mathsf{\tilde c}_\ell)_{\ell=0}^{\Lambda_j}\in l(\Lambda_j)$ is defined as
\begin{equation}\label{def:fft}
(\fft[j] \mathsf{\tilde c})_k :=   \sum_{\ell=0}^{\Lambda_j}\mathsf{\tilde c}_\ell\: \sqrt{\wN[j,k]}\:\eigfm(\pN[j,k]),\quad k = 0,\ldots, N_{j}
\end{equation}
The sequence $\fft[j]\mathsf{\tilde c}$ is said to be a \emph{$(\Lambda_j,N_{j})$-sequence} and $\mathsf{\tilde c}$ is said to be the \emph{discrete Fourier coefficient sequence} of $\fft[j]\mathsf{\tilde c}$. Let $l(\Lambda_j,N_{j})$ the set of all $(\Lambda_j,N_{j})$-sequences.
The \emph{adjoint discrete Fourier transform} $\fft[j]^*: l(N_{j})\rightarrow l(\Lambda_j)$ for a sequence $\fracoev[]=(\fracoev[k])_{k=0}^{N_j}\in l(N_j)$ is defined by
\begin{equation}\label{def:adjfft}
(\fft[j]^*\fracoev[])_\ell :=    \sum_{k=0}^{N_{j}}  \fracoev[k]\: \sqrt{\wN[j,k]}\:\conj{\eigfm(\pN[j,k])}, \quad \ell=0,\dots,\Lambda_j .
\end{equation}
%
Since $\QN$ is a polynomial-exact quadrature rule of degree $2^j$, for every $(\Lambda_j,N_j)$-sequence $\fracoev[]$, there is a \emph{unique} sequence $\mathsf{\tilde c}\in l(\Lambda_j)$ such that $\fft \mathsf{\tilde c} = \fracoev[]$.
Hence, the notation $\dfcav[]:=\mathsf{\tilde c}=\fft^* \fracoev[]$ for the discrete Fourier coefficient sequence of a $(\Lambda_j,N_j)$-sequence $\fracoev[]$ is well-defined.

Let $\mask\in l_1(\Z)$ be a mask (filter).
The \emph{discrete convolution}  $\fracoev[] \dconv \mask$ of a sequence $\fracoev[]\in l(\Lambda_j,N_j)$  with a mask $\mask$ is a sequence in  $l(\Lambda_j,N_j)$ defined as
\begin{equation}\label{eq:dis.conv}
    (\fracoev[] \dconv \mask)_k := \sum_{\ell=0}^{\Lambda_j} \dfcav[\ell]\: {\FS{\mask}}\left(\frac{\eigvm}{2^{j}}\right)\sqrt{\wN}\:\eigfm(\pN),\quad k=0,\dots,N_{j}.
\end{equation}
As $(\widehat{\fracoev[] \dconv \mask})_\ell = \dfcav[\ell]\:{\FS{\mask}}\left(\frac{\eigvm}{2^{j}}\right)$ for $\ell\in\Lambda_j$, we have $\widehat{\fracoev[] \dconv \mask}\in l(\Lambda_j)$ and the definition \eqref{eq:dis.conv} is equivalent to  $\fracoev[] \dconv \mask  = \fft(\widehat{\fracoev[] \dconv \mask} )$.

The \emph{downsampling operator} $\downsmp:l(\Lambda_j,N_j)\rightarrow l(N_{j-1})$ for a  $(\Lambda_j,N_j)$-sequence $\fracoev[]$ is
\begin{equation}\label{eq:downsampling}
    (\fracoev[]\downsmp)_k := \sum_{\ell=0}^{\Lambda_{j}} \dfcav[\ell] \:\sqrt{\wN[j-1,k]}\:\eigfm(\pN[j-1,k]),\quad k=0,\dots,N_{j-1}.
\end{equation}
The \emph{upsampling operator} $\:\upsmp: l(\Lambda_{j-1},N_{j-1})\rightarrow l(\Lambda_j,N_j)$ for a $(\Lambda_{j-1},N_{j-1})$-sequence $\fracoev[]$ is
\begin{equation}\label{eq:upsampling}
(\fracoev[]\upsmp)_k :=
\sum_{\ell=0}^{\Lambda_{j-1}}\dfcav[\ell] \sqrt{\wN}\:\eigfm(\pN),\quad k = 0,\ldots, N_{j}.
\end{equation}

For a mask $\mask$, let $\mask^\star$ be the mask satisfying $\FT{\mask^\star}(\xi) = \overline{\FT\mask(\xi)}$, $\xi\in\R$.
The following theorem shows the framelet decomposition and reconstruction using the above convolution, downsampling and upsampling, under the condition that  $\QN$ is a polynomial-exact quadrature rule of degree $2^j$.
\begin{theorem}
\label{thm:dec:rec}
Let $J_0\in\Z$ be an integer and $\Psi:=\{\scala;\scalb^1,\ldots,\scalb^r\}\subset \Lpm[\R]{1}$ with $r\ge1$ a set of framelet generators associated with a filter bank $\filtbk:=\{\maska; \maskb[1],\ldots, \maskb[r]\}\subset l_1(\Z)$ satisfying \eqref{eq:refinement}. Let $\QQ=\{\QN\}_{j\ge {\ord[0]}}$ be a sequence of quadrature rules on $\mfd$ with
$\QN:=\QN^{(j)}:=\{(\wN,\pN)\in\R\times\mfd \setsep k=0,\dots,N_{j}\}$.
Define (semi-discrete) framelet system $\frsys(\Psi,\QQ)=\frsys(\Psi,\QQ;\mfd)$, $J\ge J_0$ as in  \eqref{eq:intro.frsys}. Suppose   that \eqref{eq:prodAssump} holds, $\supp\FT{\scala}\subseteq[0,1/c]$ with $c\ge2$ the minimal integer in \eqref{eq:prodAssump}, $\QN$ is exact for polynomials of degree $2^j$ for $j\ge J_0$, and \eqref{thmeq:cfr.tightness.a.mask.cond} holds.
Let $\fracoev=(\fracoev[j,k])_{k=0}^{N_j}$ and $\frbcoev{n}=(\frbcoev[j,k]{n})_{k=0}^{N_{j+1}}$, $n=1,\ldots,r$ be the approximation coefficient sequence and detail coefficient sequences of $f\in\Lpm{2}$ at scale $j$ given by
\begin{equation}\label{eq:coe.framelets}
    \fracoev[j,k] := \InnerL{f,\fra[j,k]},\; k = 0,\ldots, N_{j}, \quad\mbox{and}\quad \frbcoev[j,k]{n} := \InnerL{f,\frb[j,k]{n}}, \;  k = 0,\ldots, N_{j+1},\;  n=1,\dots,r,
\end{equation}
respectively. Then,
\begin{enumerate}[{\rm (i)}]
\item the coefficient sequence $\fracoev[j]$ is a  $(\Lambda_j,N_j)$-seqeunce and $\frbcoev[j]{n}, n=1,\ldots, r$, are  $(\Lambda_{j+1},N_{j+1})$-seqeunces for all $j\ge {\ord[0]}$;
\item  for any $j\ge {\ord[0]}+1$, the following decomposition relations hold:
\begin{equation}
\label{def:dec:tz}
  \fracoev[j-1] = (\fracoev\dconv \maska^\star)\downsmp,\quad \frbcoev[j-1]{n} = \fracoev\dconv (\maskb)^\star,\quad  n=1,\ldots,r;
\end{equation}
\item for any $j\ge {\ord[0]}+1$, the following reconstruction relation holds:
\begin{equation}\label{eq:reconstr.j.j1}
  \fracoev =  (\fracoev[j-1]\upsmp) \dconv \maska+\sum_{n=1}^r   \frbcoev[j-1]{n} \dconv \maskb.
\end{equation}
\end{enumerate}
\end{theorem}
\begin{proof}
For $\fracoev$ and $\frbcoev[j-1]{n}$ in \eqref{eq:coe.framelets}, by \eqref{eq:fr.coeff}, \eqref{eq:refinement} and $\supp\FT{\scala}\subseteq[0,1/c]$, we obtain $\supp\FT{\scalb^n}\subseteq[0,2/c]$ and
\[
\fracoev[j,k] =\sum_{\ell=0}^{\Lambda_j}\Fcoem{f}\: \conj{\FT{\scala}\left(\frac{\eigvm}{2^{j}}\right)}\sqrt{\wN}\: \eigfm(\pN[j,k]),\quad
\frbcoev[j-1,k]{n} =\sum_{\ell=0}^{\Lambda_j}\Fcoem{f}\: \conj{\FT{\scalb^n}\left(\frac{\eigvm}{2^{j-1}}\right)}\sqrt{\wN}\: \eigfm(\pN[j,k]).
\]
Hence, $\fracoev$ and $\frbcoev[j-1]{n}$, $n=1,\dots,r$ are all $(\Lambda_j,N_j)$-sequences with the discrete Fourier coefficients  $\dfcav[j]:=(\dfcav[j,\ell])_{\ell=0}^{\Lambda_{j}}$ and $\dfcbv[j-1]{n}:=(\dfcbv[j-1,\ell]{n})_{\ell=0}^{\Lambda_{j}}$  given by
\begin{equation*}
    \dfcav[j,\ell] = \:\Fcoem{f}\: \conj{\FT{\scala}\left(\frac{\eigvm}{2^{j}}\right)},\quad
     \dfcbv[j-1,\ell]{n}  = \:\Fcoem{f}\: \conj{\FT{\scalb^n}\left(\frac{\eigvm}{2^{j}}\right)},\quad \ell=0,\ldots,\Lambda_{j}.
\end{equation*}
Thus, item (i) holds.

We observe that $\fracoev[j-1]$ is  a $(\Lambda_{j-1},N_{j-1})$-sequence.
Using \eqref{eq:refinement}  gives,  for $k=0,\ldots, N_{j-1}$,
\begin{align*}
    \fracoev[j-1,k]&= \sum_{\ell=0}^{\Lambda_{j-1}}\Fcoem{f}\: \conj{\FT{\scala}\left(\frac{\eigvm}{2^{j-1}}\right)}\sqrt{\wN[j-1,k]}\:\eigfm(\pN[j-1,k])\notag\\
    &= \sum_{\ell=0}^{\Lambda_{j-1}}\Fcoem{f}\:\conj{\FT{\scala}\left(\frac{\eigvm}{2^{j}}\right)}\: \conj{\FT{\maska}\left(\frac{\eigvm}{2^{j}}\right)}\sqrt{\wN[j-1,k]}\: \eigfm(\pN[j-1,k])\\
    &= \sum_{\ell=0}^{\Lambda_{j}} \dfcav[j,\ell]\:\conj{\FT{\maska}\left(\frac{\eigvm}{2^{j}}\right)}\sqrt{\wN[j-1,k]}\: \eigfm(\pN[j-1,k])\\
    &= [(\fracoev \dconv \maska^\star)\downsmp](k).
\end{align*}
Similarly, for $k=0,\ldots, N_{j-1}$ and $n=1,\dots,r$,
\begin{equation*}
\frbcoev[j-1,k]{n}
=\sum_{\ell=0}^{\Lambda_{j}} \Fcoem{f}\: \conj{\FT{\scalb^{n}}\left(\frac{\eigvm}{2^{j-1}}\right)} \sqrt{\wN[j-1,k]}\: \eigfm(\pN[j-1,k])
=(\fracoev \dconv (\maskb)^\star)_k.
\end{equation*}
This proves \eqref{def:dec:tz}, thus, item (ii) holds.

Using $\fracoev[j-1]=(\fracoev\dconv\maska^{\star})\downsmp$ and $\frbcoev[j-1]{n}=\fracoev \dconv (\maskb)^\star$, we obtain
\[
\widetilde{\fracoev[]}: = (\fracoev[j-1]\upsmp)\dconv\maska + \sum_{n=1}^{r} \frbcoev[j-1]{n}\dconv\maskb
=(((\fracoev\dconv\maska^{\star})\downsmp)\upsmp)\dconv\maska
+ \sum_{n=1}^{r} (\fracoev \dconv (\maskb)^\star)\dconv \maskb.
\]
This with \eqref{eq:dis.conv}, \eqref{eq:downsampling}, \eqref{eq:upsampling} and \eqref{thmeq:cfr.tightness.a.mask.cond} gives
\[
\begin{aligned}
\widetilde{\fracoev[]}_k &= \sum_{\ell=0}^{\Lambda_j}\dfcav[j,\ell]\left( \left|\FT\maska\left(\frac{\eigvm}{2^j}\right)\right|^2+\sum_{n=1}^r
 \left|\FT\maskb\left(\frac{\eigvm}{2^j}\right)\right|^2\right) \sqrt{\wN}\:{\eigfm}(\pN)
\\&=\sum_{\ell=0}^{\Lambda_j}\dfcav[j,\ell] \sqrt{\wN}\:{\eigfm}(\pN)
\\&= \fracoev[j,k],\quad
\end{aligned}
\]
thus proving \eqref{eq:reconstr.j.j1}, which completes the proof.
\end{proof}

\medskip

Theorem~\ref{thm:dec:rec} gives the one-level framelet decomposition and reconstruction on $\mfd$, as illustrated by Figure~\ref{fig:algo:fb}.
Given a sequence $\fracoev[\ord]\in l(\Lambda_J,N_{J})$ with $J\ge {\ord[0]}\in\Z$, the \emph{multi-level framelet filter bank decomposition} from level $\ord$ to $\ord[0]$ is given by
\begin{equation*}
\label{def:dec:tz:J}
  \fracoev[j-1] = (\fracoev\dconv \maska^\star)\downsmp,\quad \frbcoev[j-1]{n} = \fracoev \dconv ({\maskb})^\star,\quad  n=1,\ldots,r,\quad j = J,\ldots, {\ord[0]}+1.
\end{equation*}
The corresponding multi-level framelet \emph{analysis operator}
\begin{equation*}
\label{def:analOp}
\analOp: l(\Lambda_J,\QN[N_{J}]) \rightarrow l(N_{J})^{1\times r}\times l(N_{J-1})^{1\times r}\times\cdots\times l(N_{{\ord[1]}})^{1\times r}\times l(N_{{\ord[0]}})
\end{equation*}
is defined as
\begin{equation}
\label{def:dec:tr:J}
\analOp\fracoev[\ord] = (\frbcoev[\ord-1]{1},\ldots, \frbcoev[\ord-1]{r}, \ldots, \frbcoev[{\ord[0]}]{1},\ldots,\frbcoev[{\ord[0]}]{r}, \fracoev[{\ord[0]}]),\quad \fracoev[\ord]\in l(\Lambda_J,N_{J}).
\end{equation}
For a sequence  $(\frbcoev[\ord-1]{1},\ldots, \frbcoev[\ord-1]{r}, \ldots, \frbcoev[{\ord[0]}]{1},\ldots,\frbcoev[{\ord[0]}]{r}, \fracoev[{\ord[0]}])$ of framelet coefficient sequences obtained from a multi-level decomposition,
the \emph{multi-level framelet filter bank reconstruction} is given by
\begin{equation*}
\label{def:rec:sd:J}
  \fracoev = (\fracoev[j-1]\upsmp) \dconv \maska + \sum_{n=1}^{r} \frbcoev[j-1]{n}\dconv \maskb,\quad j = {\ord[0]}+1,\ldots,J.
\end{equation*}
The corresponding multi-level framelet \emph{synthesis operator}
\begin{equation*}
\label{def:synOp}
\synOp:  l(N_{J})^{1\times r}\times l(N_{J-1})^{1\times r}\times\cdots\times l(N_{{\ord[1]}})^{1\times r}\times l(N_{{\ord[0]}}) \rightarrow l(\Lambda_J,N_{J})
\end{equation*}
is defined as
\[
\synOp( \frbcoev[\ord-1]{1},\ldots, \frbcoev[\ord-1]{r}, \ldots, \frbcoev[{\ord[0]}]{1},\ldots,\frbcoev[{\ord[0]}]{r}, \fracoev[{\ord[0]}]) = \fracoev[\ord].
\]
When the condition of Theorem~\ref{thm:dec:rec} is satisfied, the analysis and synthesis operators are invertible on $l(\Lambda_j,N_j)$ for any $j\ge {\ord[0]}$, i.e. $\synOp\analOp = \Id|_{l(\Lambda_j,N_j)}$, where $\Id$ is the identity operator.
The two-level decomposition and reconstruction framelet filter bank transforms are the processes using the one-level twice, as depicted by the diagram in Figure~\ref{fig:multi-level-FMT}. Similarly, the multi-level framelet filter bank transforms are recursive use of the one-level. The detailed algorithmic steps of the decomposition and reconstruction are described in Algorithms~\ref{algo:decomp.multi.level} and \ref{algo:reconstr.multi.level} in Section~\ref{sec:fmt}.

\begin{figure}[htb]
\begin{minipage}{\textwidth}
	\centering
\begin{minipage}{\textwidth}
\begin{center}
\begin{tikzpicture}[thick,scale=0.66, every node/.style={scale=0.67}, nonterminal/.style={rectangle, minimum size=5.5mm, very thick, draw=red!50!black!50,top color=white, bottom color=red!50!black!20,font=\itshape},
terminal/.style={rectangle,minimum size=6mm,rounded corners=1mm,very thick,draw=black!50,top color=white,bottom color=black!20,font=\ttfamily},
sum/.style={circle,minimum size=1mm,very thick,draw=black!50,top color=white,bottom color=black!20,font=\ttfamily},
skip loop/.style={to path={-- ++(0,#1) -| (\tikztotarget)}},
hv path/.style={to path={-| (\tikztotarget)}},
vh path/.style={to path={|- (\tikztotarget)}},
,>=stealth',thick,black!50,text=black,
every new ->/.style={shorten >=1pt},
graphs/every graph/.style={edges=rounded corners}]
\matrix[row sep=1mm, column sep=5.5mm] {
& & & & & \node (ma2) [terminal] {$\dconv[j-1]\maska^\star$}; & \node (dsa2) [terminal] {$\downsmp[j-1] \hspace{-0.5mm}$}; & \node (pra2) [nonterminal] {processing};
 & \node (usa2) [terminal] {$\upsmp[j-1] \hspace{-0.5mm} $}; & \node (mas2) [terminal] {$\dconv[j-1]\maska$}; & & & & &\\
& & \node (ma1) [terminal] {$\dconv\maska^\star$}; & \node (dsa1) [terminal] {$\downsmp \hspace{-0.5mm}$}; & \node (p2) [coordinate] {}; & & & & & &
\node (plus2) [sum] {$+_{r}$}; & \node (usa1) [terminal] {$\upsmp \hspace{-0.5mm} $}; & \node (mas1) [terminal] {$\dconv\maska$}; & & \\
& & & & & \node (mb2) [terminal] {$\dconv[j-1]{(\maskb)}^\star$}; &  & \node (prb2) [nonterminal] {processing};
 &  & \node (mbs2) [terminal] {$\dconv[j-1]\maskb$}; & & & & &\\
\node (in) [nonterminal] {input}; & \node (p1) [coordinate] {}; &&&&& &&&&& && \node (plus1) [sum] {$+_{r}$}; & \node (out) [nonterminal] {output};\\
& & \node (mb1) [terminal] {$\dconv{(\maskb)}^\star$}; &  && && \node (prb1) [nonterminal] {processing}; && &&  & \node (mbs1) [terminal] {$\dconv\maskb$}; & & \\
};

\graph [use existing nodes] {
ma2 -> dsa2 -> pra2 -> usa2 -> mas2;
ma1 -> dsa1 -> p2 -> [vh path] {ma2,mb2}; {mas2,mbs2} -> [hv path] plus2 -> usa1 -> mas1;
mb2 -> prb2 -> mbs2;
in -- p1 -> [vh path] {ma1, mb1}; {mas1,mbs1} -> [hv path] plus1 -> out;
mb1 -> prb1 -> mbs1;
};
\end{tikzpicture}
\end{center}\vspace{-1mm}
\end{minipage}
\begin{minipage}{0.8\textwidth}
\caption{Two-level framelet filter bank decomposition and reconstruction based on the filter bank $\{\maska;\maskb[1],\ldots,\maskb[r] \}$.}
\label{fig:multi-level-FMT}
\end{minipage}
\end{minipage}
\end{figure}
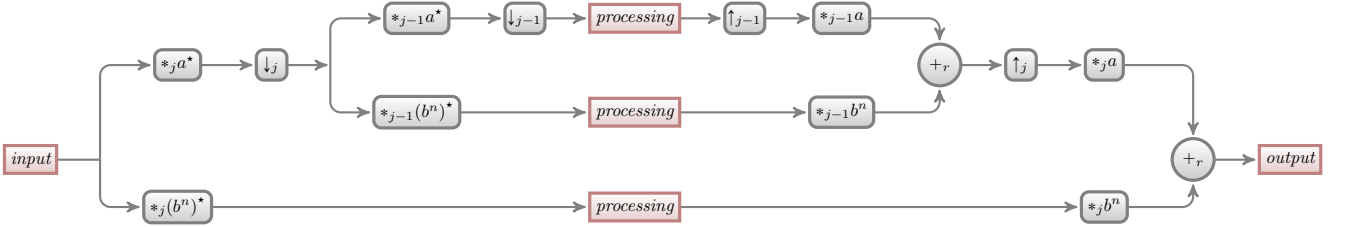

\subsection{Fast framelet filter bank transforms}\label{sec:fmt}
The decomposition in \eqref{def:dec:tz} and the reconstruction in \eqref{eq:reconstr.j.j1} can be rewritten in terms of discrete Fourier transforms (DFTs) and adjoint DFTs on $\mfd$ as
\[
  \fracoev[j-1] = \fft[j-1](\widehat{\fracoev\dconv \maska^\star}),\quad \frbcoev[j-1]{n} = \fft[j](\widehat{\fracoev\dconv {(\maskb)}^\star}),\quad n = 1,\ldots, r
\]
and
\[
  \fracoev = \left(\fft[j]^*(\fracoev[j-1])\right) \dconv \maska + \sum_{n=1}^r \left(\fft[j]^*(\frbcoev[j-1]{n})\right) \dconv \maskb.
\]
The decomposition and reconstruction are thus combinations of discrete Fourier transforms (or the adjoint DFTs) with discrete convolutions. As ${\fracoev\dconv \mask}$ is simply point-wise multiplication in the frequency domain, the computational complexity of the algorithms is determined by the computational complexity of DFTs and adjoint DFTs. Assuming fast discrete Fourier transforms on $\mfd$, the multi-level framelet filter bank transforms can be efficiently implemented in the sense that the computational steps are in proportion to the size of the input data. We say these algorithms \emph{fast framelet filter bank transforms on $\mfd$}, or {\fmt}s.

Let $(\wN[k],\pN[k])_{k=0}^N$ a quadrature rule on $\mfd$, $\fracoev[]=(\fracoev[k])_{k = 0}^ N$ a data sequence with respect to $(\wN[k],\pN[k])_{k=0}^N$ in the time domain, and
$\dfcav[]=(\dfcav[\ell])_{\ell = 0}^M$ the sequence of discrete Fourier coefficients of $\fracoev[]$ in the frequency domain. The discrete Fourier transform for the sequence of Fourier coefficients $\dfcav[]$ on $\mfd$ is given by
\begin{equation}\label{def:fft2}
(\fft[]  \dfcav[])_k :=   \sum_{\ell=0}^M\dfcav[\ell]\: \sqrt{\wN[k]}\eigfm(\pN[k]),\quad k = 0,\ldots, N,
\end{equation}
and the adjoint discrete Fourier transform for the sequence $\fracoev[]$ on $\mfd$ is given by
\begin{equation}\label{def:adjfft2}
(\fft[]^*\fracoev[])_\ell :=    \sum_{k=0}^{N}  \fracoev[k]\: \sqrt{\wN[k]}\:\conj{\eigfm(\pN[k])}, \quad \ell = 0,\ldots, M,
\end{equation}
see \eqref{def:fft} and \eqref{def:adjfft}.
Without loss of generality, we assume  $M\le N$.

By ``fast'' we  mean that the computation of $(\fft[]  \dfcav[])_{k=0}^N$ given ($\dfcav[\ell])_{\ell=0}^M$ in \eqref{def:fft2} (or the computation of $(\fft[]^*\fracoev[])_{\ell=0}^{M}$ in \eqref{def:adjfft2}) can be realized in order $\bigo{}{N}$ flops up to a log factor similar to the standard FFT algorithms on $\R$ ($\eigfm=\exp(-2\pi i \ell\cdot)$ in $\R$). The inverse discrete Fourier transform $\fft[]^{-1}$ can be implemented in the same order $\bigo{}{N}$ by solving the normal equation $\fft[]^*\fft[] \dfcav[]= \fft[]^*\fracoev[]$ using conjugate gradient methods (CG).

Fast algorithms for DFTs and adjoint DFTs exist in typical manifolds, for example, the fast spherical harmonic transforms on the sphere, the fast discrete Fourier transforms on the  torus  and the fast Legendre transforms on the hypercube, see e.g. \cite{DrHe1994,HaTo2014,HeRoKoMo2003,KeKuPo2007,RoTy2006}.

Algorithms~\ref{algo:decomp.multi.level} and \ref{algo:reconstr.multi.level} below show the detailed algorithmic steps for the multi-level {\fmt}s for the decomposition and reconstruction of the framelet coefficient sequences on a manifold assuming the condition of Theorem~\ref{thm:dec:rec}.

We give a brief analysis of the computational complexity analysis of the {\fmt} algorithms (assuming $J\ge J_0=0$), as follows.

In Algorithm~\ref{algo:decomp.multi.level}, the line~1 is of order $\bigo{}{N_J}$; the lines~2--8 together are of order $\bigo{}{r(N_{J-1}+\cdots+N_0)}$; the line~9 is of order $\bigo{}{N_0}$; the total complexity is $\bigo{}{N_J+r(N_{J-1}+\cdots+N_0)+N_0}$.

In Algorithm~\ref{algo:reconstr.multi.level}, the line~1 is of order $\bigo{}{N_0}$; the lines~2--7 together are of order
$\bigo{}{r(N_0+\cdots+N_{J-1})}$; the line~8 is of order $\bigo{}{N_J}$; the total complexity is $\bigo{}{N_0+r(N_0+\cdots+N_{J-1})+N_J}$.

If the numbers of the nodes of the quadrature rules $\QN$ and $\QN[j-1]$ in consecutive levels satisfy $\frac{N_j}{N_{j-1}} \asymp c_0$ for all $j\ge1$ with $c_0>1$, the computational complexities of both the {\fmt} decomposition and reconstruction are of order $\bigo{}{(r+1)N_J}$ for the sequence $\fracoev[\ord]$ of the framelet coefficients of size $N_{\ord}$. Note that $\bigo{}{(r+1)N_J}$ is also the order of the redundancy rate of the {\fmt} algorithms. For example, on the unit sphere $\sph{2}$, using symmetric spherical designs (see \cite{Womersley_ssd_URL}), the number of the quadrature nodes $N_j\sim 2^{2j+1}$, then $\frac{N_j}{N_{j-1}} \asymp 4$, and the FFT on $\sph{2}$ has the complexity $\bigo{}{N\sqrt{\log N}}$ with $N$ the size of the input data, see e.g. \cite{KeKuPo2007}, thus, the {\fmt} on $\sph{2}$ has the computational complexity $\bigo{}{N\sqrt{\log N}}$.

\medskip
\IncMargin{1em}
\begin{algorithm}[H]
\SetKwData{step}{Step}
\SetKwInOut{Input}{Input}\SetKwInOut{Output}{Output}
\BlankLine
\Input{$\fracoev[\ord]$ -- a $(\Lambda_J,N_J)$-sequence}

\Output{$\bigl(\{\frbcoev[\ord-1]{n},\frbcoev[\ord-2]{n},\dots,\frbcoev[J_0]{n}\}_{n=1}^{r},\fracoev[J_0]\bigr)$ as in \eqref{def:dec:tr:J}}

$\fracoev[\ord] \longrightarrow \dfcav[\ord]$ \tcp*[f]{adjoint FFT}\\

\For{$j\leftarrow \ord$ \KwTo $J_0+1$}{
    $\dfcav[j-1] \longleftarrow \dfcav[j,\cdot] \: \conj{\FT{\maska}}\left(2^{-j}\eigvm[\cdot]\right)$ \tcp*[f]{downsampling \& convolution}\\[1mm]

\For{$n\leftarrow 1$ \KwTo $r$}{

    $\dfcbv[j-1]{n} \longleftarrow  \dfcav[j,\cdot] \: \conj{\FT{\maskb}}\left(2^{-j}\eigvm[\cdot]\right)$ \tcp*[f]{convolution}\\[1mm]

    $\frbcoev[j-1]{n} \longleftarrow \dfcbv[j-1]{n}$ \tcp*[f]{ FFT}\\
    }
}

$\fracoev[J_0] \longleftarrow \dfcav[J_0]$ \tcp*[f]{ FFT}

\caption{Multi-Level {\fmt}: Decomposition}
\label{algo:decomp.multi.level}
\end{algorithm}
\DecMargin{1em}\vspace{3mm}

\IncMargin{1em}
\begin{algorithm}[H]
\SetKwData{step}{Step}
\SetKwInOut{Input}{Input}\SetKwInOut{Output}{Output}
\BlankLine
\Input{$\bigl(\{\frbcoev[\ord-1]{n},\frbcoev[\ord-2]{n},\dots,\frbcoev[J_0]{n}\}_{n=1}^{r},\fracoev[J_0]\bigr)$ as in \eqref{def:dec:tr:J}}

\Output{$\fracoev[\ord]$ -- a $(\Lambda_{J},N_J)$-sequence}

$\dfcav[J_0] \longleftarrow \fracoev[J_0]$ \tcp*[f]{adjoint   FFT }\\

\For{$j\leftarrow J_0+1$ \KwTo $\ord$}{

\For{$n\leftarrow 1$ \KwTo $r$}{

$\dfcbv[j-1]{n}\longleftarrow \frbcoev[j-1]{n}$ \tcp*[f]{adjoint  FFT}\\

}

$\dfcav \longleftarrow (\dfcav[j-1,\cdot])\:\FT{\maska}\left(2^{-j}\eigvm[\cdot]\right)  + \sum_{n=1}^{r}\dfcbv[j,\cdot]{n} \;\FT{\maskb}\left(2^{-j}\eigvm[\cdot]\right)$\tcp*[f]{upsampling \& convolution}\\

}

$\fracoev[\ord] \longleftarrow \dfcav[\ord]$ \tcp*[f]{FFT}\\
\caption{Multi-Level {\fmt}: Reconstruction}\label{algo:reconstr.multi.level}
\end{algorithm}
\DecMargin{1em}

\section{Multiscale data analysis on the sphere}
\label{sec:fmtS2expr}

In this section, we construct tight framelets on the sphere $\sph{2}$ and present several examples to demonstrate data analysis on $\sph{2}$ using tight framelets.

\subsection{Framelets on the sphere}
\label{sec:fmtS2}
In this subsection, we give an explicit construction of framelets on $\sph{2}$ to illustrate the results in Section~\ref{sec:tight.framelets}. For simplicity, we consider the filter bank $\filtbk=\{\maska;\maskb[1],\maskb[2]\}$  with  two high-pass filters. We remark that $\filtbk$ can be extended to a filter bank with arbitrary number of high-pass filters in a similar manner.

Define the filter bank $\filtbk:=\{\maska;\maskb[1],\maskb[2]\}$ by their Fourier series as follows.
%
\begin{subequations}\label{eqs:mask.numer.s3}
\begin{align}
  \FT{\maska}(\xi)&: =
  \left\{\begin{array}{ll}
    1, & |\xi|<\frac{1}{8},\\[1mm]
    \cos\bigl(\frac{\pi}{2}\hspace{0.3mm} \nu(8|\xi|-1)\bigr), & \frac{1}{8} \le |\xi| \le \frac{1}{4},\\[1mm]
    0, & \frac14<|\xi|\le\frac12,
    \end{array}\right. \\[1mm]
  \FT{\maskb[1]}(\xi)&:  =\left\{\begin{array}{ll}
    0, & |\xi|<\frac{1}{8},\\[1mm]
    \sin\bigl(\frac{\pi}{2}\hspace{0.3mm} \nu(8|\xi|-1)\bigr), & \frac{1}{8} \le |\xi| \le \frac{1}{4},\\[1mm]
    \cos\bigl(\frac{\pi}{2}\hspace{0.3mm} \nu(4|\xi|-1)\bigr), & \frac14<|\xi|\le\frac12.
    \end{array}\right.\\[1mm]
  \FT{\maskb[2]}(\xi)&:
  =\left\{\begin{array}{ll}
    0, & |\xi|<\frac{1}{4},\\[1mm]
    \sin\bigl(\frac{\pi}{2}\hspace{0.3mm} \nu(4|\xi|-1)\bigr), & \frac{1}{4} \le |\xi| \le \frac{1}{2},
    \end{array}\right.
\end{align}
\end{subequations}
where
\begin{equation*}
  \nu(t) := \scalg[3](t)^{2} = t^{4}(35 - 84t + 70t^{2} - 20 t^{3}), \quad t\in\Rone,
\end{equation*}
as in \cite[Chapter~4]{Daubechies1992}. It can be verified that
\[
|\FT{\maska}(\xi)|^2+|\FT{\maskb[1]}(\xi)|^2+|\FT{\maskb[2]}(\xi)|^2 = 1 \quad\forall \xi\in[0,1/2],
\]
which implies \eqref{thmeq:cfr.tightness.a.mask.cond}. The associated framelet generators $\Psi=\{\scala; \scalb^1,\scalb^2\}$ satisfying \eqref{eq:refinement} and \eqref{thmeq:cfr.scal.j.j1} is explicitly given by
\begin{subequations}\label{eqs:scal.numer.s3}
\begin{align}
  \FT{\scala}(\xi)&=
  \left\{\begin{array}{ll}
    1, & |\xi|<\frac{1}{4},\\[1mm]
    \cos\bigl(\frac{\pi}{2}\hspace{0.3mm} \nu(4|\xi|-1)\bigr), & \frac{1}{4} \le |\xi| \le \frac{1}{2},\\[1mm]
    0, & \hbox{else},
    \end{array}\right. \\[1mm]
  \FT{\scalb^1}(\xi)&= \left\{\begin{array}{ll}
    \sin\left(\frac{\pi}{2}\hspace{0.3mm} \nu(4|\xi|-1)\right), & \frac{1}{4}\le|\xi|<\frac{1}{2},\\[1mm]
    \cos^2\left(\frac{\pi}{2}\hspace{0.3mm} \nu(2|\xi|-1)\right), & \frac{1}{2} \le |\xi| \le 1,\\[1mm]
    0, & \hbox{else},
    \end{array}\right.\\[1mm]
  \FT{\scalb^2}(\xi)&= \left\{\begin{array}{ll}
   0, &|\xi|<\frac{1}{2},\\[1mm]
    \cos\left(\frac{\pi}{2}\hspace{0.3mm} \nu(2|\xi|-1)\right) \sin\left(\frac{\pi}{2}\hspace{0.3mm} \nu(2|\xi|-1)\right), & \frac{1}{2} \le |\xi| \le 1,\\[1mm]
    0, & \hbox{else}.
    \end{array}\right.
\end{align}
\end{subequations}
Then, $\FS{\maska},\FS{\maskb[1]},\FS{\maskb[2]},\FT{\scala},\FT{\scalb^{1}},\FT{\scalb^{2}}$ are all in $\CkR[4-\epsilon]$ with arbitrarily small and positive $\epsilon$ \cite[p.~119]{Daubechies1992}, and $\supp\FT{\scala}\subseteq [0,1/2]$ and $\supp\FT{\scalb^{n}}\subseteq [1/4,1]$, $n=1,2$. Also, the refinable function $\FT{\scala}$ satisfies \eqref{thmeq:cfr.scala.lim}.

Figure~\ref{fig:mask} shows the pictures of the filters $\FS{\maska}$, $\FS{\maskb[1]}$ and $\FS{\maskb[2]}$ of \eqref{eqs:mask.numer.s3}. Figure~\ref{fig:scal} shows the corresponding functions $\FT{\scala}$, $\FT{\scalb^1}$ and $\FT{\scalb^2}$, whose supports are subsets of $[0,1/2]$, $[1/4,1]$ and $[1/2,1]$.
\begin{figure}
\begin{minipage}{\textwidth}
\centering	
\begin{minipage}{\textwidth}
  \centering
  \begin{minipage}{0.47\textwidth}
  \includegraphics[trim = 0mm 0mm 0mm 0mm, width=0.9\textwidth]{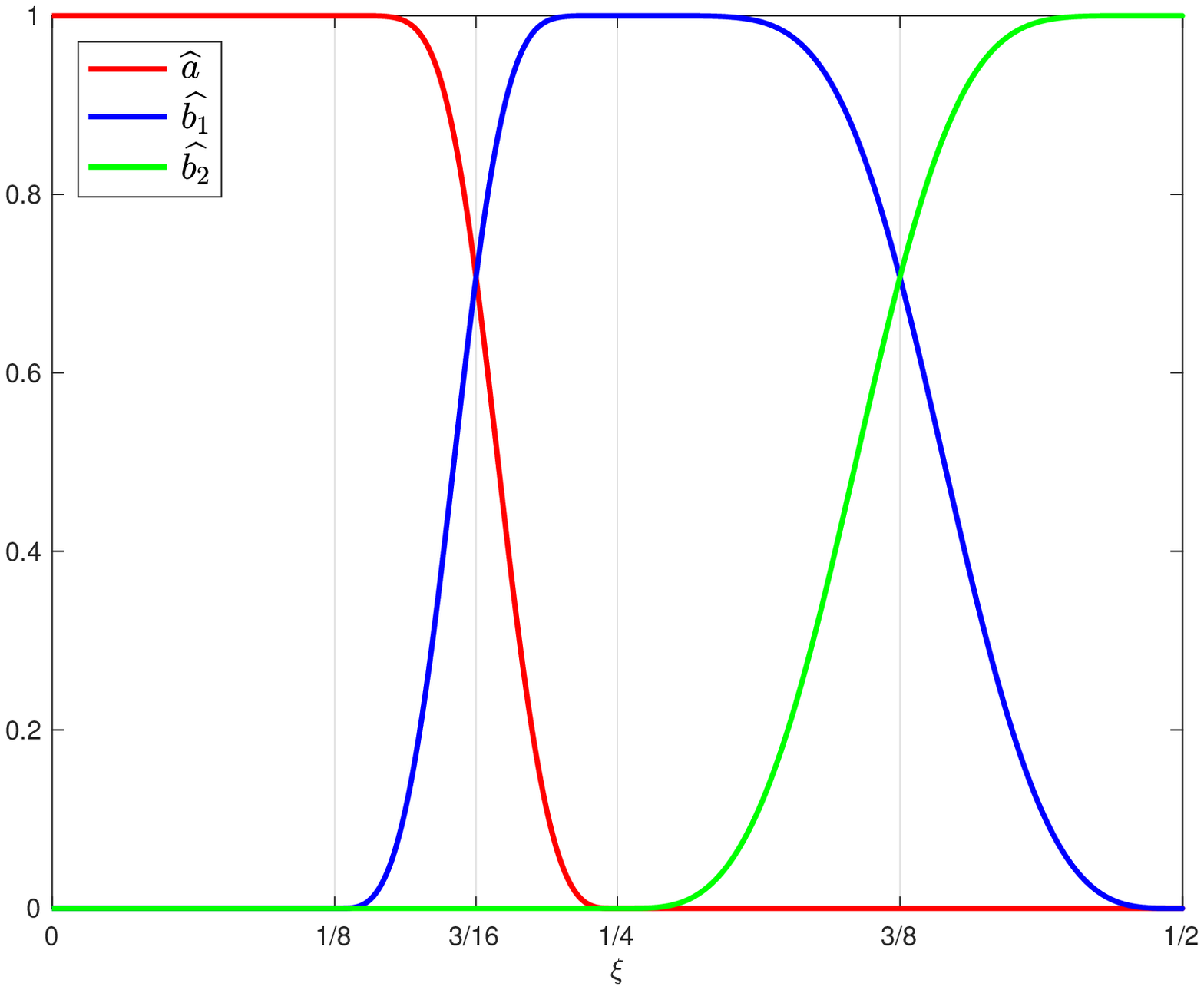}
  \subcaption{Filters $\FS{\maska}$, $\FS{\maskb[1]}$ and $\FS{\maskb[2]}$}\label{fig:mask}
  \end{minipage}
  \begin{minipage}{0.47\textwidth}
  \includegraphics[trim = 0mm 0mm 0mm 0mm, width=0.9\textwidth]{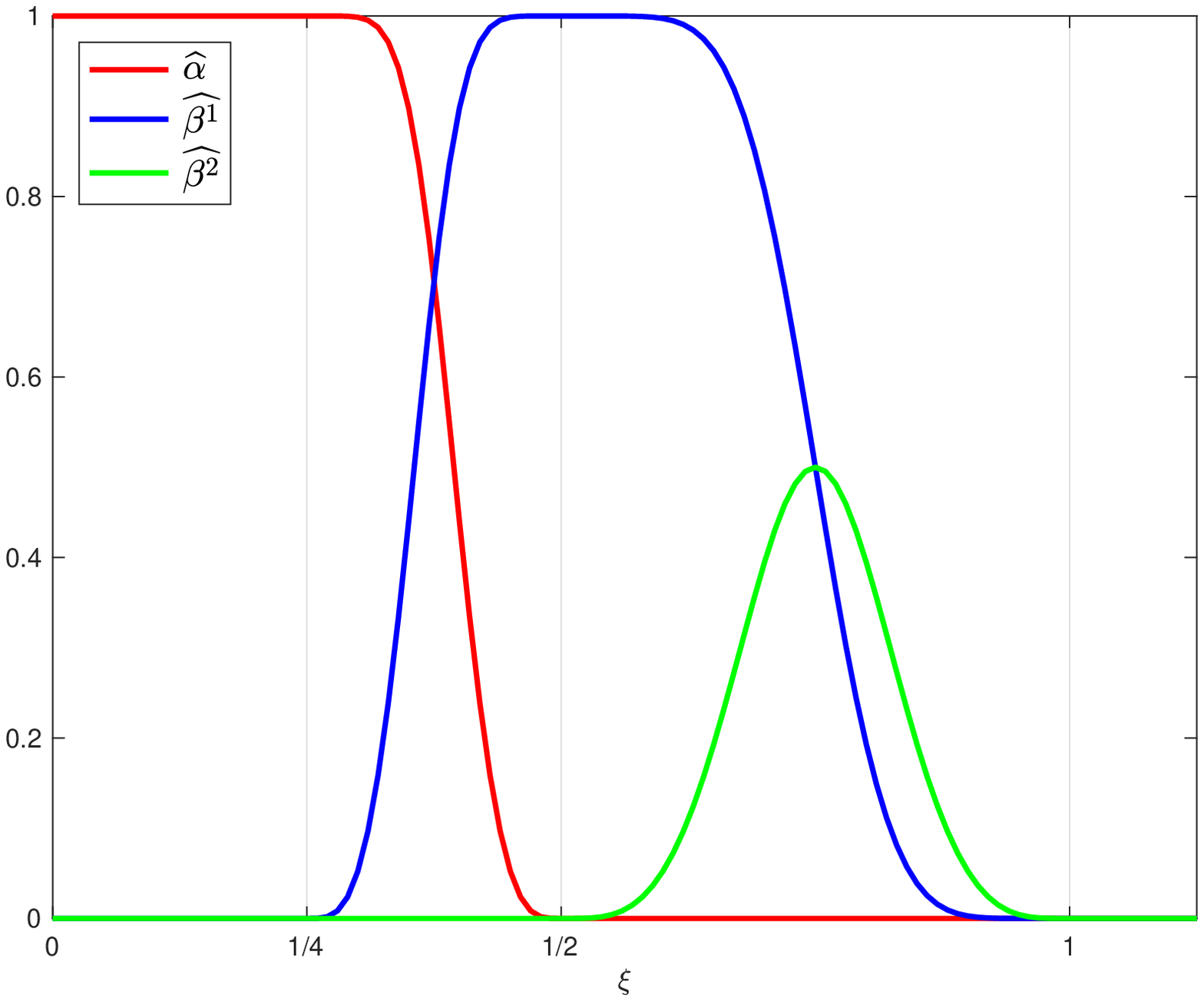}
  \subcaption{Functions $\FT{\scala}$, $\FT{\scalb^{1}}$, and $\FT{\scalb^{2}}$}\label{fig:scal}
  \end{minipage}
  \end{minipage}
\begin{minipage}{0.8\textwidth}
\vspace{1mm}
\caption{Filters $\{\FS{\maska};\FS{\maskb[1]},\FS{\maskb[2]}\}$ and functions $\{\FT{\scala};\FT{\scalb^{1}},\FT{\scalb^{2}}\}$.}\label{fig:mask.scal}
\end{minipage}
\end{minipage}
\end{figure}

For the unit sphere $\sph{2}\subset\Rd[3]$, the Laplace-Beltrami operator $\LBm$ has the spherical harmonics $\{\shY: \ell\in\N_0, |m|\le \ell\}$ as eigenfunctions with (negative) eigenvalues $-\eigvs^{2}=-\ell(\ell+1)$:
\begin{equation*}
    \LBm\: \shY = -\ell(\ell+1) \shY,\quad  m=-\ell,\ldots, \ell,\; \ell=0,1,\ldots,
\end{equation*}
 see e.g. \cite[Chapter~1]{DaXu2013} for details.
Let $(\theta,\varphi)$ with $\theta\in[0,\pi]$ and $\varphi\in [0,2\pi)$ be the spherical coordinates for $\PT{x}\in\sph{2}$, satisfying
$\PT{x}=(\cos\theta\sin\varphi,\cos\theta\sin\varphi,\sin\theta)$.
Using the spherical coordinates, the spherical harmonics can be explicitly written as
\begin{equation*}
  \shY(\PT{x}):=\shY(\theta,\varphi) := \sqrt{\frac{2\ell+1}{4\pi}\frac{(\ell-m)!}{(\ell+m)!}} \:\aLegen{m}(\cos\theta)\: e^{\imu m\varphi}, \quad m=-\ell,\ldots, \ell,\; \ell=0,1,\dots,
\end{equation*}
where $\aLegen{m}(t)$, $-1\le t\le 1$, is the associated Legendre polynomial of degree $\ell$ and order $m$, see e.g. \cite{DaXu2013}. Let $\memf$ be the surface measure on the sphere $\sph{2}$ satisfying $\memf(\sph{2})=1$.  Then $\{(\shY,\eigvs)\}_{m\le|\ell|,\ell\in\N_0}$ forms an orthonormal eigen-pair for $\Lpm[\sph{2},\memf]{2}:=\Lpm[\sph{2}]{2}$.
The (diffusion) polynomial space $\Pi_n$ is given by
\begin{equation}\label{eq:PnS2}
\Pi_n:=\spann\{\shY\setsep \lambda_{\ell,m}\le n\}=\spann\{\shY\setsep \ell<n,\: m=-\ell,\dots,\ell\}.
\end{equation}

The continuous framelets $\cfra(\PT{x}), \cfrb{1}(\PT{x})$ and $ \cfrb{2}(\PT{x})$ on the sphere $\sph{2}$ are
\begin{align*}
  \cfra(\PT{x})&:= \sum_{\ell=0}^{\infty}\sum_{m=-\ell}^{\ell} \FT{\scala}\left(\frac{\eigvs}{2^{j}}\right)\conj{\shY(\PT{y})}\shY(\PT{x}),\\
  \cfrb{n}(\PT{x})&:= \sum_{\ell=0}^{\infty}\sum_{m=-\ell}^{\ell} \FT{\scalb^n}\left(\frac{\eigvs}{2^{j}}\right)\conj{\shY(\PT{y})}\shY(\PT{x}), \quad n = 1,2.
\end{align*}
By (iv) or (v) of Theorem~\ref{thm:cfr.tightness} and the construction of $\Psi$ and $\filtbk$ in \eqref{eqs:scal.numer.s3} and \eqref{eqs:mask.numer.s3}, the continuous framelet system $\cfrsys(\Psi)=\cfrsys(\{\scala;\scalb^1,\scalb^2\})$ on $\sph{2}$ is a tight frame for $\Lpm[\sph{2}]{2}$ for any $J\in\Z$.

Given $\QN$ a quadrature rule on $\sph{2}$, the discrete framelets $\fra(\PT{x}), \frb{1}(\PT{x})$ and $\frb{2}(\PT{x})$ on the sphere $\sph{2}$ are
\begin{align*}
  \fra(\PT{x})&:= \sqrt{\wN}\sum_{\ell=0}^{\infty}\sum_{m=-\ell}^{\ell} \FT{\scala}\left(\frac{\eigvs}{2^{j}}\right)\conj{\shY(\pN)}\shY(\PT{x}),\\
  \frb{n}(\PT{x})&:= \sqrt{\wN[j+1,k]}\sum_{\ell=0}^{\infty}\sum_{m=-\ell}^{\ell} \FT{\scalb^n}\left(\frac{\eigvs}{2^{j}}\right)\conj{\shY(\pN[j+1,k])}\shY(\PT{x}),\quad n =1,2.
\end{align*}
%
%
As the supports of $\FT\scala$, $\FT{\scalb^1}$ and $\FT{\scalb^2}$ are subsets of $[0,1/2]$, $[0,1]$ and $[0,1]$, $\fra\in\Pi_{2^{j-1}}$ and $\frb{1}, \frb{2}\in\Pi_{2^{j}}$.
If $\QN$ is a polynomial-exact quadrature rule of degree $2^j$ for all $j\in\Z$, then by Corollary~\ref{cor:framelet.tightness2}, the framelet system $\frsys(\Psi,\QQ)=\frsys(\{\scala;\scalb^1,\scalb^2\},\{\QN\}_{j\ge J})$ is a  semi-discrete tight frame for $\Lpm[\sph{2}]{2}$ for all $J\in\Z$.

Figure~\ref{fig:fr.S2} shows the pictures of framelets $\fra[6,\PT{y}]$, $\frb[6,\PT{y}]{1}$ and $\frb[6,\PT{y}]{2}$ on $\sph{2}$ at scale $j=6$ and with translation at $\PT{y}=(0,0,1)$. It shows that $\frb[6,\PT{y}]{1}$ and $\frb[6,\PT{y}]{2}$ are more ``concentrated'' at the north pole, which enables them to carry more detailed information in data analysis.

\begin{figure}[htb]
\begin{minipage}{\textwidth}
\centering
\begin{minipage}{\textwidth}
  \centering
  \begin{minipage}{0.30\textwidth}
  \centering
  \includegraphics[trim = 10mm 0mm 0mm 1mm, width=0.8\textwidth]{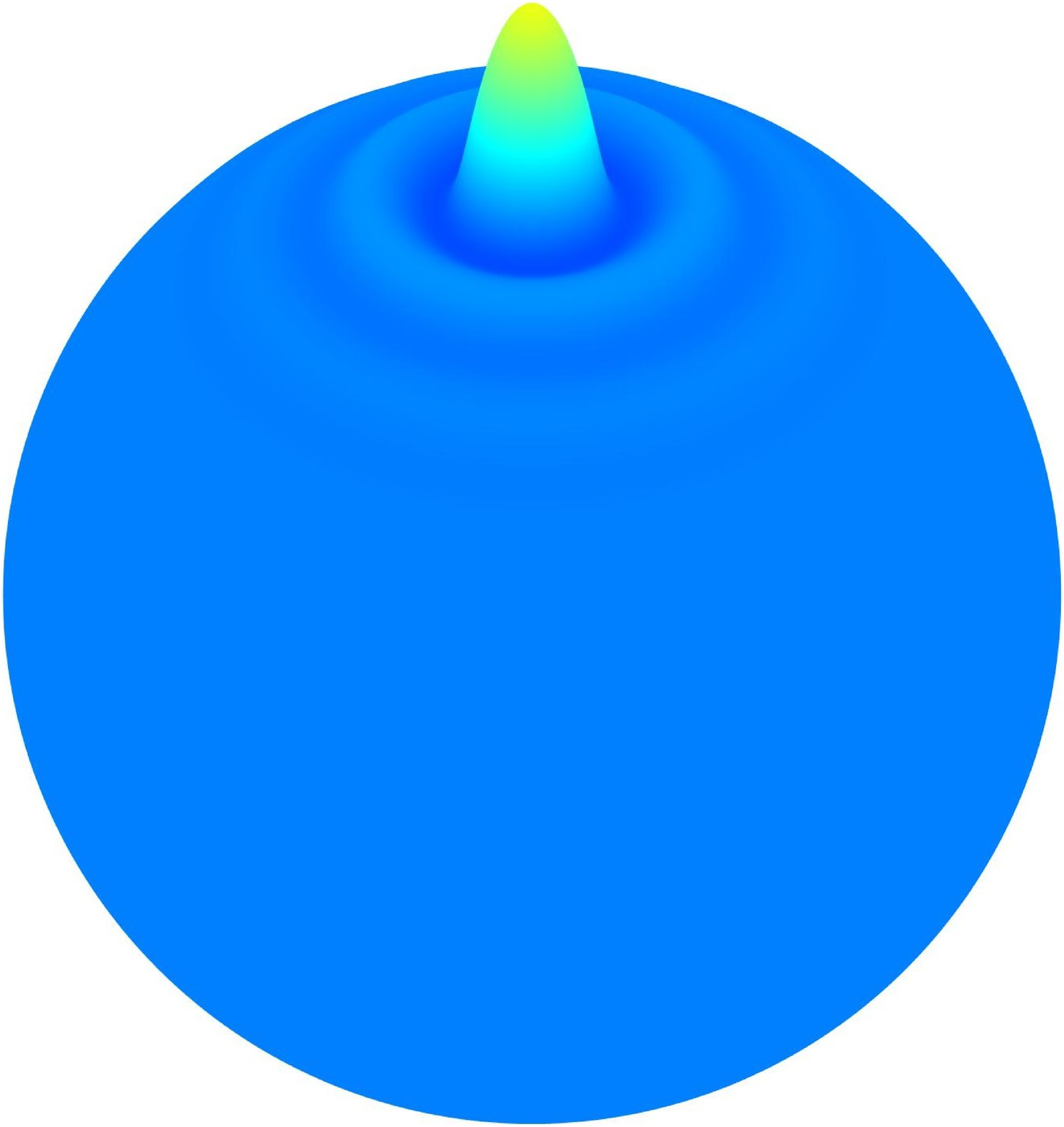}\\
  \subcaption{Framelet $\fra[6,\PT{y}]$}\label{fig:fra.S2}
  \end{minipage}
  \begin{minipage}{0.30\textwidth}
  \centering
  \includegraphics[trim = 10mm 0mm 0mm 1mm, width=0.8\textwidth]{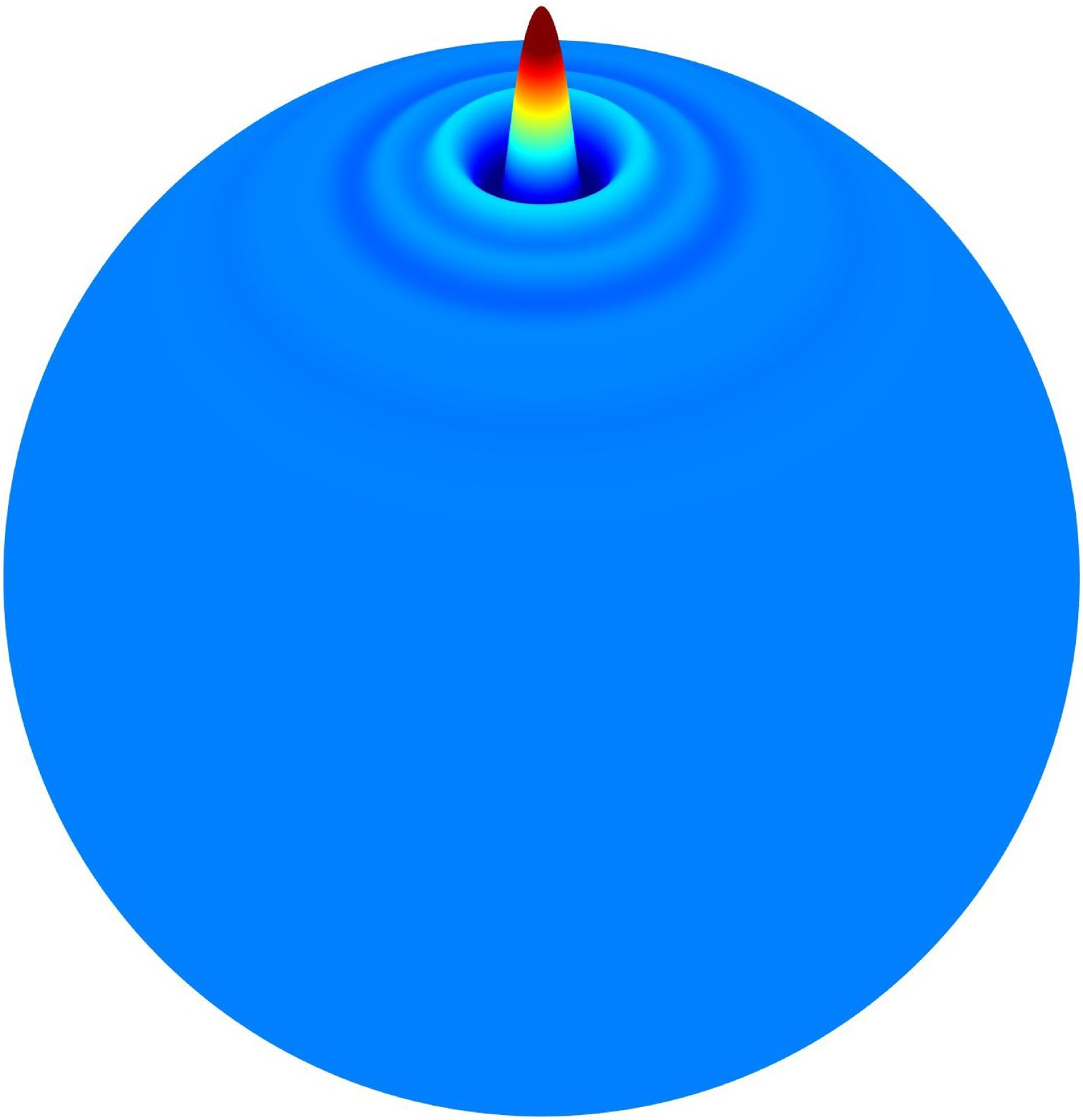}\\
  \subcaption{Framelet $\frb[6,\PT{y}]{1}$}\label{fig:frb.S2}
  \end{minipage}
  \hspace{2mm}
  \begin{minipage}{0.30\textwidth}
  \centering
  \includegraphics[trim = 0mm 0mm 0mm 1mm, width=0.98\textwidth]{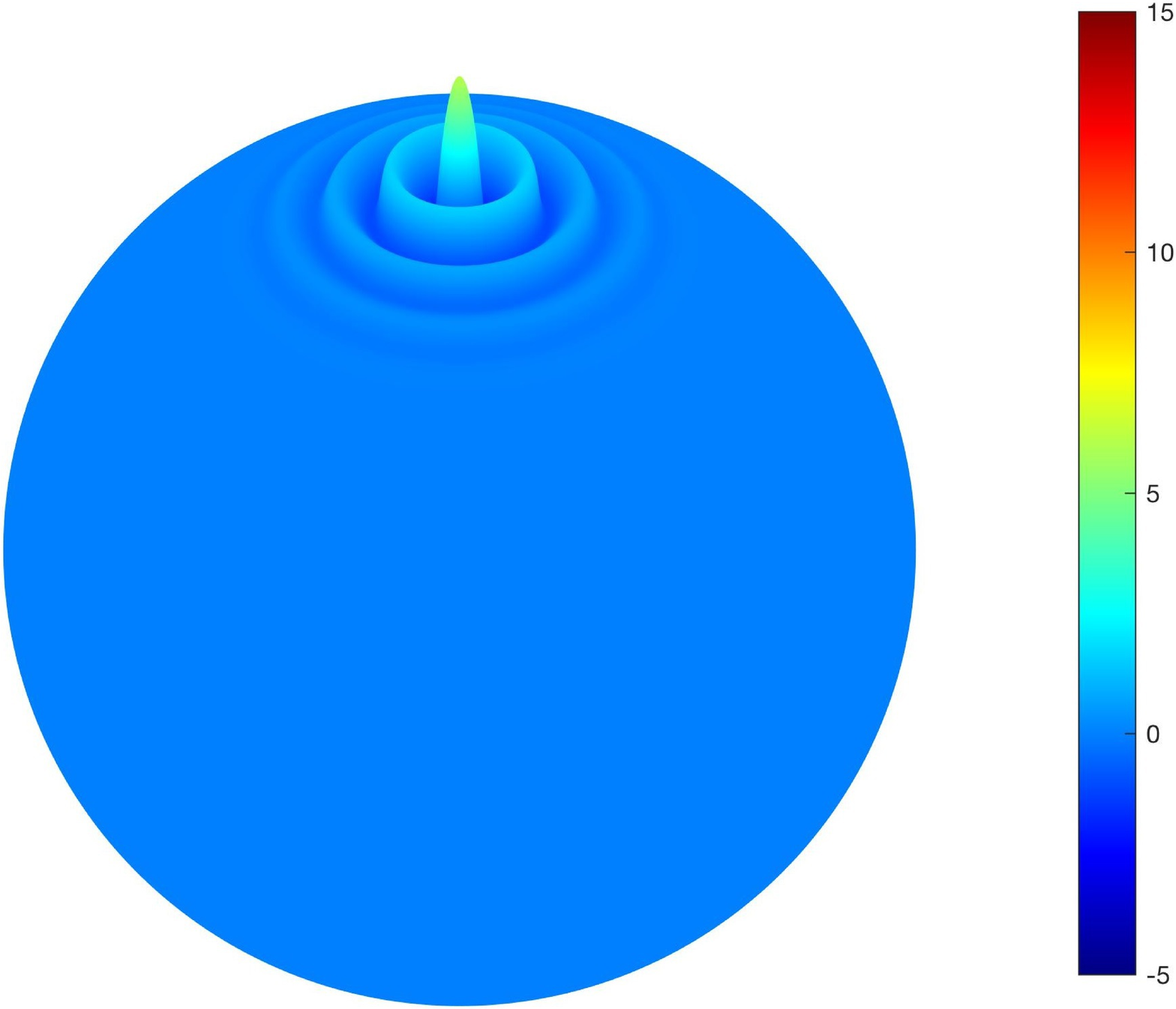}\\
  \subcaption{Framelet $\frb[6,\PT{y}]{2}$}\label{fig:frb.S2}
  \end{minipage}
\end{minipage}
\begin{minipage}{0.8\textwidth}
\vspace{1mm}
\caption{Framelets on $\sph{2}$, scale $j=6$ and $\PT{y}=(0,0,1)$}\label{fig:fr.S2}
\end{minipage}
\end{minipage}
\end{figure}

\subsection{Numerical examples}\label{sec:numer}
In this subsection, we show three numerical examples on $\mfd=\sph{2}$ of the {\fmt} algorithms using the framelet system $\frsys(\Psi,\QQ)=\frsys(\{\scala;\scalb^1,\scalb^2\},\QQ)$ as presented in Subsection~\ref{sec:fmtS2}. The three examples  illustrate for {\fmt}s: the approximation for smooth functions, the multiscale decomposition for a topological data set and the computational complexity for CMB data.

Let $\Psi=\{\scala;\scalb^1,\scalb^2\}$ be the framelet generators associated with the filter bank $\filtbk=\{\maska;\maskb[1],\maskb[2]\}$ given in Section~\ref{sec:fmtS2}, and $\QQ=\{\QN[N_j]\}_{ j={\ord[0]}}^J$ a sequence of point sets on the sphere. We can define a sequence of framelet systems $\frsys[j](\Psi,\QQ)$, $j={\ord[0]},\ldots,J$, as \eqref{eq:intro.frsys}, which can be used to process data on the sphere as described in Algorithms~\ref{algo:decomp.multi.level} and~\ref{algo:reconstr.multi.level}.
A data sequence $\fracoev[]$ sampled from a function on  $\QN[N_J]$ at the finest scale $J$ may not be a $(\Lambda_J,N_J)$-sequence as required by our decomposition and reconstruction algorithms. We can preprocess the data by projecting $\fracoev[]$ onto $\Pi_{2^{J}}$ to obtain a $(\Lambda_J,N_{J})$-sequence using the inverse discrete Fourier transform on the manifold, which splits the data sequence into the approximation coefficient sequence  $\fracoev[\ord]$ at the finest scale  $J$ and the projection error sequence $\projerr=  \fracoev[]-\fracoev[J]$. More precisely, the data sequence $\fracoev[]$ is projected onto $\Pi_{2^{J}}$ by $\fracoev[\ord] = (\fft[\ord]^*\fft[\ord])^{-1}\fft[\ord]^*\fracoev[]$ using the spherical harmonic transform $\fft[\ord]$ and the adjoint spherical harmonic transform $\fft[\ord]^*$. Both of $\fft[\ord]$ and $\fft[\ord]^{*}$ can be implemented fast, in order  $\bigo{}{N_J\sqrt{\log N_J}}$, see e.g. Keiner, Kunis and Potts \cite{KeKuPo2007}. See Example~\ref{ex:1}.

When $\QN[N_j]$ is a (polynomial-exact) quadrature rule of order $2^j$ for all $j=J_0,\ldots, J$, as guaranteed by Theorem~\ref{thm:dec:rec}, we can decompose the $(\Lambda_J,N_J)$-sequence $\fracoev[J]$ and obtain the framelet  coefficient sequences $\frbcoev[\ord-1]{1}$, $\frbcoev[\ord-1]{2}$, $\ldots$, $\frbcoev[{\ord[0]}]{1}$, $\frbcoev[{\ord[0]}]{2}$, $\fracoev[{\ord[0]}]$ by the {\fmt} decomposition in  Algorithm~\ref{algo:decomp.multi.level}. Furthermore,
by using adjoint FFT transforms, we can exactly reconstruct $\fracoev[\ord]$ from the decomposed coefficient sequences $(\frbcoev[\ord-1]{1},\dots, \frbcoev[\ord-1]{r}, \ldots, \frbcoev[{\ord[0]}]{1},\dots,\frbcoev[{\ord[0]}]{r}, \fracoev[{\ord[0]}])$ using the {\fmt} reconstruction in Algorithm~\ref{algo:reconstr.multi.level}.
Once $\fracoev[\ord]$ is obtained, the sequence $\fracoev[]=\fracoev[\ord]+\projerr$ will be constructed with the pre-computed projection error $\projerr$. See Example~\ref{ex:2} for illustration of these steps.

For comparison and illutration of our algorithms in practice, we also show numerical examples of fast framelet algorithms with non-polynomial-exact quadrature rules.
When $\QN[N_j]$ are \emph{not} polynomial-exact quadrature rules, e.g. SP (generalized spiral points) or HL (HEALPix points) in Figure~\ref{fig:QN}, inverse FFT instead of adjoint FFT (see Lines 1 and 4 of Algorithm~\ref{algo:reconstr.multi.level}) is needed to obtain the discrete Fourier coefficients. In this case, errors may appear in each stage of the fast algorithms as the framelets might not be tight, due to the numerical integration errors for polynomials of the point sets. But in practice, one could record such error in each stage.  As this paper is focused on polynomial-exact quadrature rules, we do not get into details on errors for framelets with non-polynomial-exact rules.

We use four types of point sets on $\sph{2}$ as follows.
\begin{enumerate}[(1)]
\item \emph{Gauss-Legendre tensor product rule} (GL) \cite{HeWo2012}. The Gauss-Legendre tensor product rule is a (polynomial-exact but not equal area) quadrature rule $\QN[N]= \{(\wN[k],\pN[k]): k=0,\ldots,N\}$ on the sphere generated by the tensor product of the Gauss-Legendre nodes on the interval $[-1,1]$ and equi-spaced nodes on the longitude with non-equal weights. The GL rule is a polynomial-exact quadrature rule of degree $n$ satisfying $N = n \times (\lfloor (n-1)/2\rfloor+1)$ ($\lfloor (n-1)/2\rfloor+1$ nodes on $[-1,1]$ and $n$ nodes on longitude). Figure~\ref{fig:GL} shows the GL rule with  $n=32$ and  $N=512$.

\item \emph{Symmetric spherical designs} (SD) \cite{Womersley_ssd_URL}. The symmetric spherical design is a (polynomial-exact) quadrature rule $\QN[N]= \{(\wN[k],\pN[k]): k=0,\ldots,N\}$ on the sphere $\sph{2}$ with equal weights $\wN[k]=1/N$. The points are ``equally'' distributed on the sphere. The SD rule is a polynomial-exact quadrature rule of degree $n$ with $N\sim n(n-1)/2$. Figure~\ref{fig:SD} shows the SD rule with $n=32$ and $N=498$.

\item \emph{Generalized spiral points} (SP) \cite{Bauer2000}.
  The rule of generalized spiral points $\QN[N] = \{(1/N,\pN[k]): k = 0,\ldots, N\}$ is given by $\pN[k] = (\cos(1.8\sqrt{N}\theta_k) \sin\theta_k, \sin(1.8\sqrt{N}\theta_k)\sin \theta_k, \cos\theta_k)$ where $\theta_k=\arccos( 1-(2k-1)/N)$ for $k=0,\ldots,N$. We assign equal weights to the SP nodes as they are equal area. SP with equal weights is, however, not a polynomial-exact quadrature rule on the sphere. In the numerical test, to compare with polynomial-exact quadrature rules, we use the SP points with $N=2^{2j+1}$ nodes at scaling level $j$. Figure~\ref{fig:SP} shows the SP points with $N=512$.

\item \emph{HEALPix points\footnote{\url{http://healpix.sourceforge.net}}} (HL) \cite{Gorski_etal2005}. HL is a hierarchical equal area isolatitude point configuration on the sphere.
At each resolution $k$ where $k$ is a positive integer, the number of HL points $N(k)=12\times 2^{2k}$, and the HL partition of the resolution $k$ is nested in that of the resolution $k+1$.
As SP, we assign equal weights to the HL points as nodes of SP are equally distributed. HL with equal weights is not a polynomial-exact quadrature rule on the sphere either. For $j\ge0$, let $k_{j}$ be the smallest positive integer such that $2^{2j+1}\le12\times2^{2k_{j}}$. In the numerical test, to compare with polynomial-exact quadrature rules, we use the HL points of resolution $2^{k_{j}}$ with $N_{j}=12\times2^{2k_{j}}$ nodes at the scaling level $j$ for $j\ge0$. Figure~\ref{fig:HL} shows the HL points of resolution $2^4=16$ on $\sph{2}$ with $N(4)=768$.
\end{enumerate}
%

\begin{figure}[htb]
\begin{minipage}{\textwidth}
	\centering
\begin{minipage}{\textwidth}
  \centering
  \begin{minipage}{0.24\textwidth}
  \centering
  \includegraphics[trim = 0mm 0mm 0mm 0mm, width=0.85\textwidth]{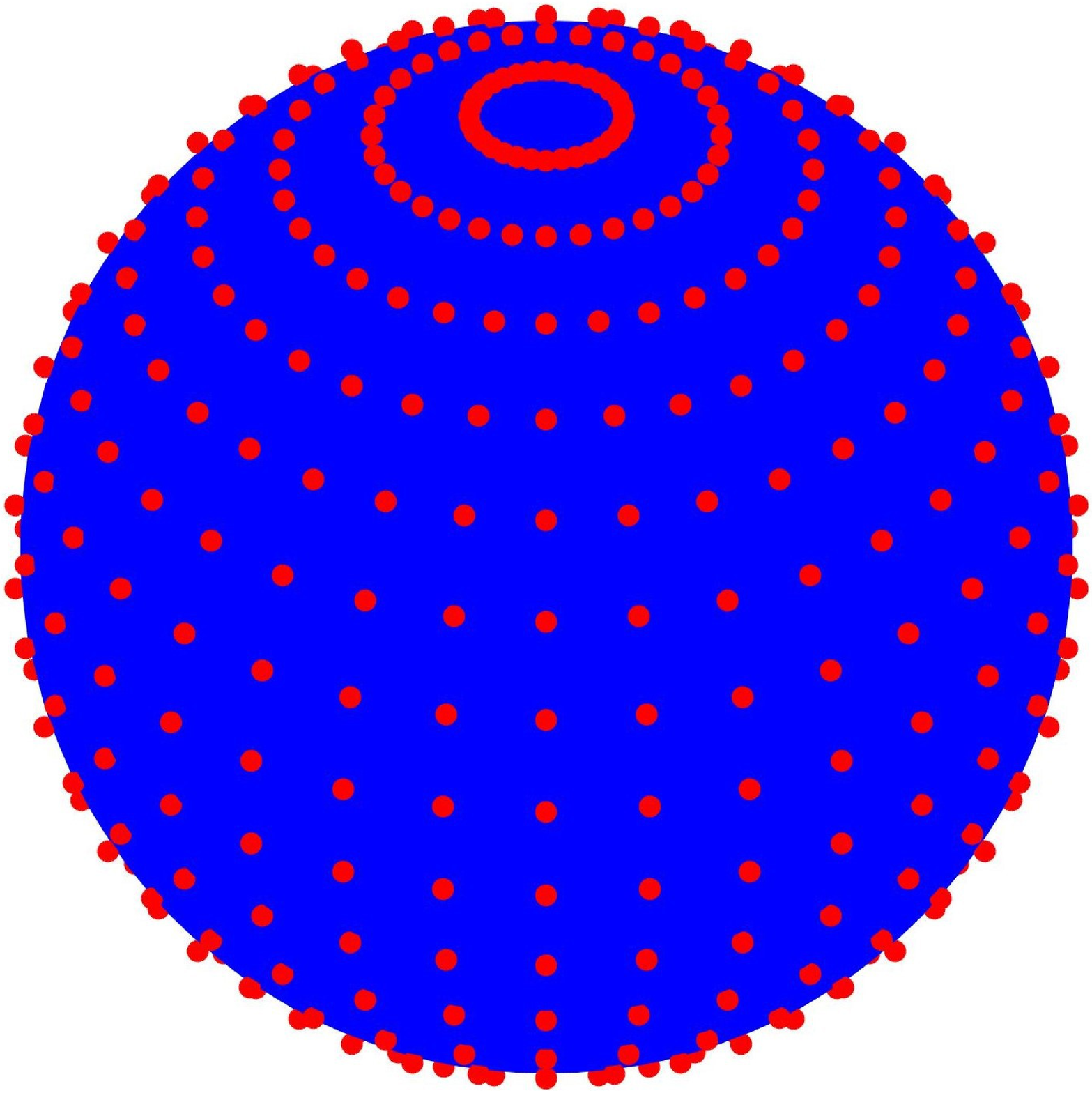}\\[1mm]
  \subcaption{GL, $N=512$}\label{fig:GL}
  \end{minipage}
  \begin{minipage}{0.24\textwidth}
  \centering
  \includegraphics[trim = 0mm 0mm 0mm 0mm, width=0.85\textwidth]{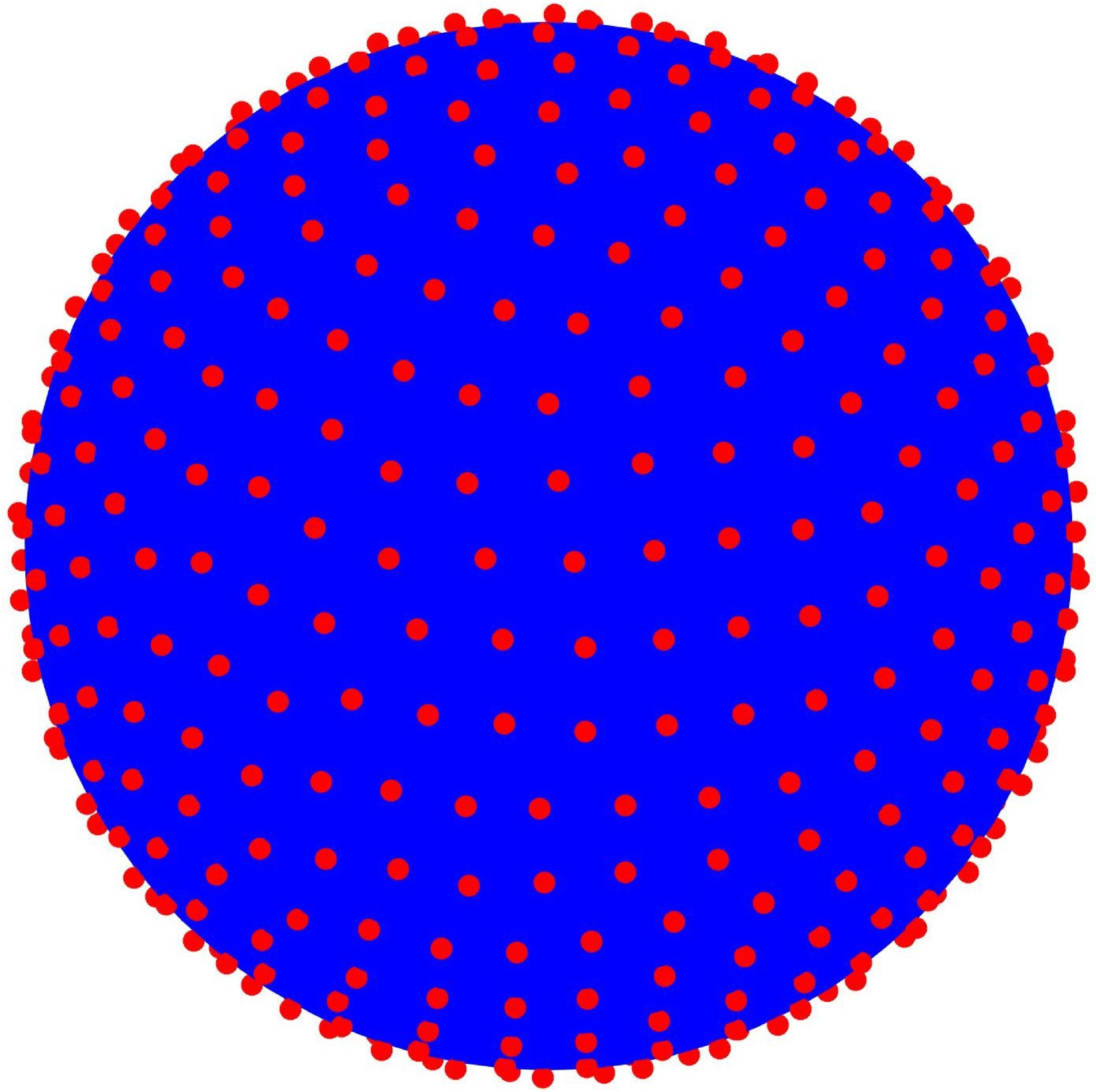}\\[1mm]
  \subcaption{SD, $N=498$}\label{fig:SD}
  \end{minipage}
  \begin{minipage}{0.24\textwidth}
  \centering
  \includegraphics[trim = 0mm 0mm 0mm 0mm, width=0.85\textwidth]{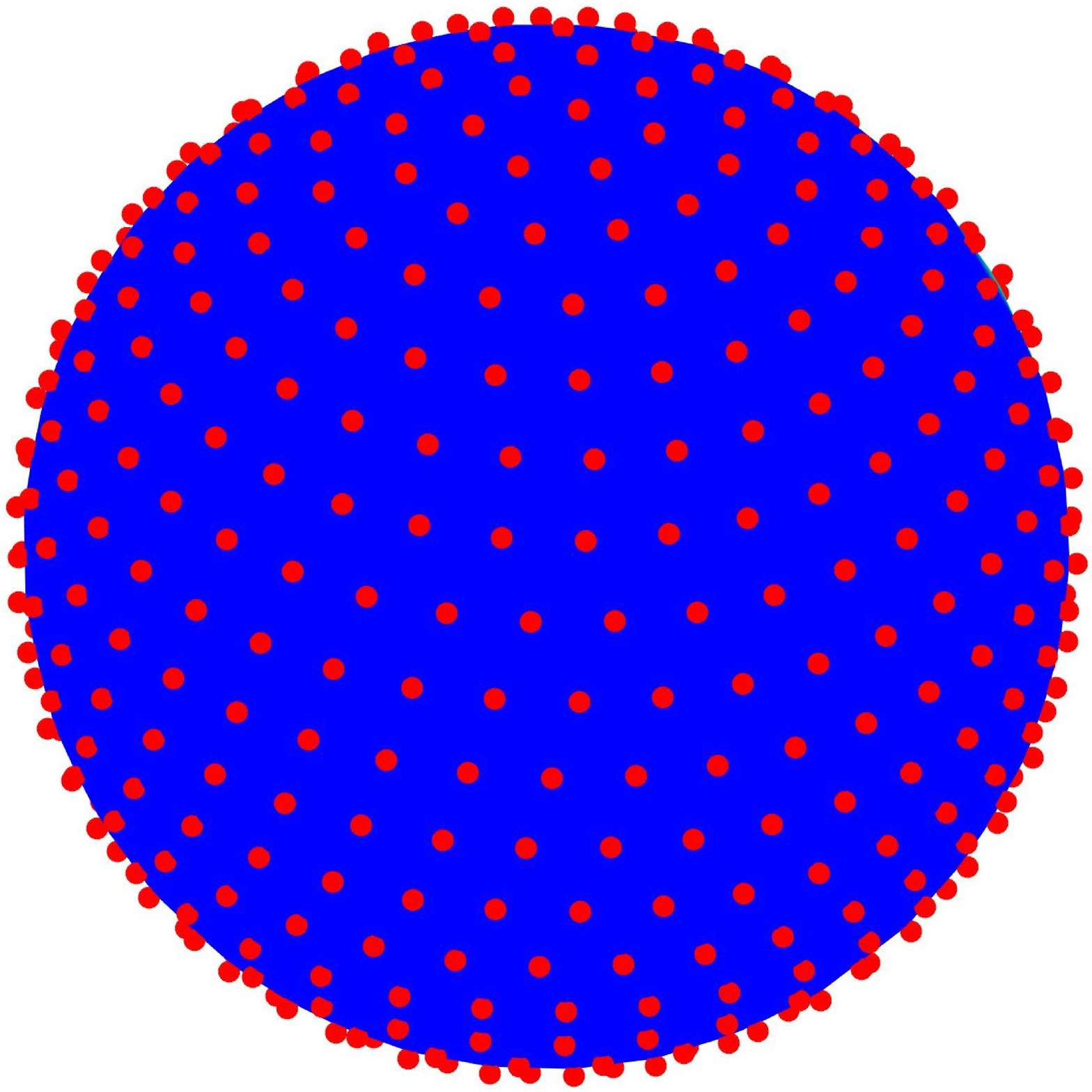}\\[1mm]
  \subcaption{SP, $N = 512$}\label{fig:SP}
  \end{minipage}
  \begin{minipage}{0.24\textwidth}
  \centering
  \includegraphics[trim = 0mm 0mm 0mm 0mm, width=0.85\textwidth]{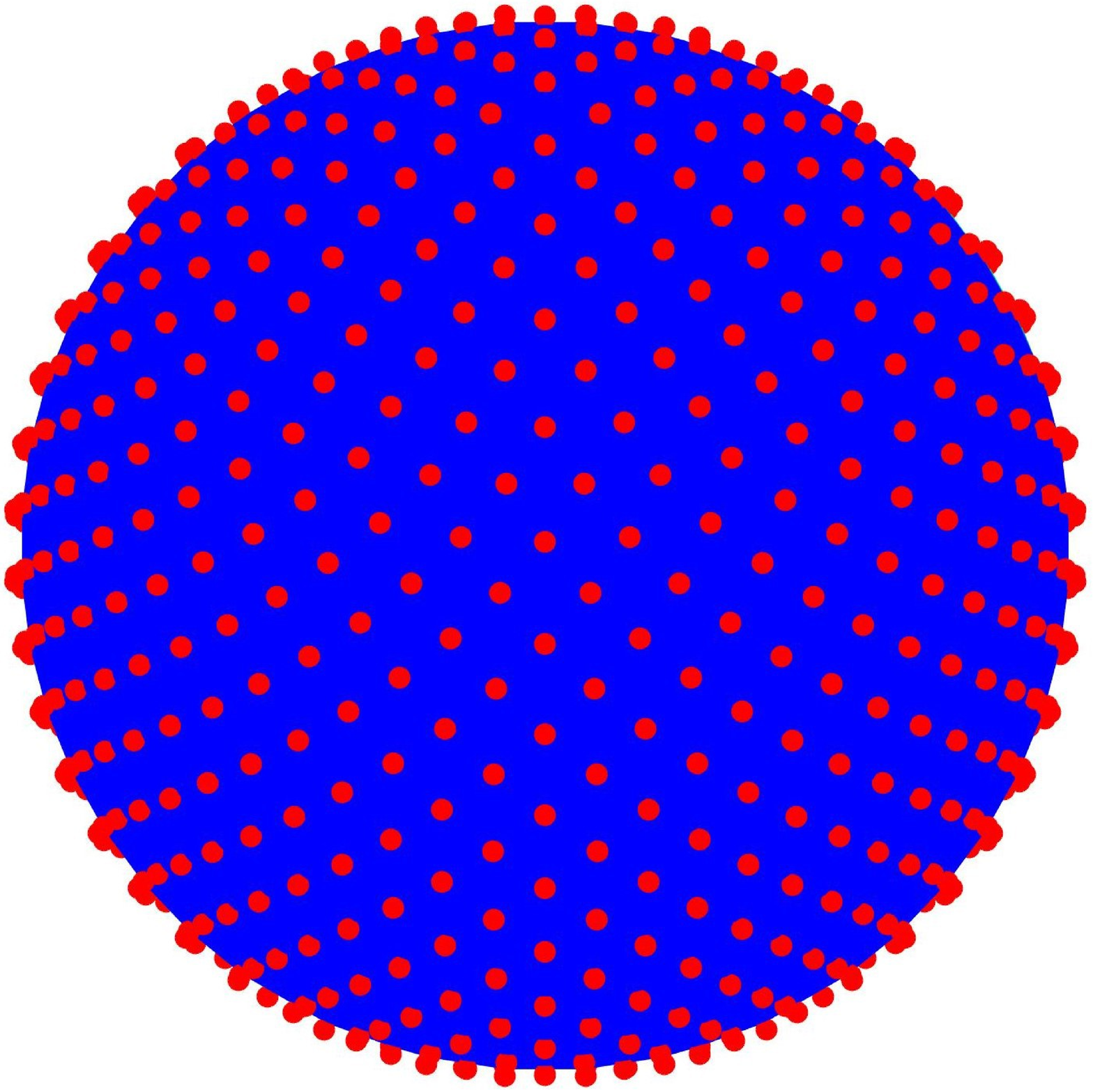}\\[1mm]
  \subcaption{HL, $N=768$}\label{fig:HL}
  \end{minipage}
\end{minipage}
\begin{minipage}{0.8\textwidth}
\vspace{1mm}
\caption{Point sets on the sphere for Gauss-Legendre  rule (GL), symmetric spherical designs (SD), generalized spiral points (SP), and HEALPix (HL).}
\label{fig:QN}
\end{minipage}
\end{minipage}
\end{figure}


\begin{example}[Approximation of smooth functions]
\label{ex:1}
{\rm
We illustrate the approximation ability of $\fra$ in the framelet system $\frsys(\Psi)$ on $\sph{2}$ under different types of point sets for the following test functions of the combinations of normalized Wendland functions \cite{ChSlWo2014}.

Let $(t)_{+}:=\max\{t,0\}$ for $t\in\mathbb{R}$. The original Wendland functions are
\begin{equation*}
  \fWend{n}(t) := \begin{cases}
  (1-t)_{+}^{2}, & n = 0,\\[1mm]
  (1-t)_{+}^{4}(4t + 1), & n = 1,\\[1mm]
  \displaystyle (1-t)_{+}^{6}(35t^2 + 18t + 3)/3, & n = 2,\\[1mm]
  (1-t)_{+}^{8}(32t^3 + 25t^2 + 8t + 1), & n = 3,\\[1mm]
  \displaystyle (1-t)_{+}^{10}(429t^4 + 450t^3 + 210t^2 + 50t + 5)/5, & n = 4.
  \end{cases}
\end{equation*}
The normalized (equal area) Wendland functions are
\begin{equation*}
    \fnWend{n}(t) := \fWend{n}\Bigl(\frac{t}{\tau_{n}}\Bigr),\quad \tau_{n} := \frac{(3n+3)\Gamma(n+\frac{1}{2})}{2\:\Gamma(n+1)},\quad n\geq0.
\end{equation*}
The Wendland functions scaled this way have the property of converging pointwise to a Gaussian as $n\to \infty$, see Chernih et al. \cite{ChSlWo2014}.
%
Let $\PT{z}_{1}:=(1,0,0)$, $\PT{z}_{2}:=(-1,0,0)$, $\PT{z}_{3}:=(0,1,0)$, $\PT{z}_{4}:=(0,-1,0)$, $\PT{z}_{5}:=(0,0,1)$ and $\PT{z}_{6}:=(0,0,-1)$ be six points on $\sph{2}$ and define \cite{LeSlWe2010}
\begin{equation}\label{eq:Phi}
 f_{n}(\PT{x})
  := \sum_{i=1}^{6}\fnWend{n}(|\PT{z}_{i} - \PT{x}|), \quad n\geq0
\end{equation}
so that $\PT{z}_{i}$ are six centers of $f_{n}$,
where $|\cdot|$ is the Euclidean distance.
Le Gia, Sloan and Wendland \cite{LeSlWe2010} proved that $f_{n}\in \sobH[2]{n+\frac{3}{2}}$, where $\sobH[2]{\sigma}:=\{f\in\Lpm[\sph{2}]{2}\setsep \sum_{\ell=0}^\infty\sum_{|m|\le\ell} (1+\ell)^{2\sigma}|\Fcoem[\ell,m]{f}|^2\}<\infty\}$ is the Sobolev space with smooth parameter $\sigma>1$. As the function $f_{n}$ has known smoothness, we can see from the approximation errors the dependence of tight framelets with different points sets on the smoothness of $f_{n}$.

\begin{figure}[htb]
\begin{minipage}{\textwidth}
\centering
\begin{minipage}{\textwidth}
  \centering
  \begin{minipage}{0.24\textwidth}
  \centering
  \includegraphics[trim = 0mm 0mm 0mm 1mm, width=0.9\textwidth]{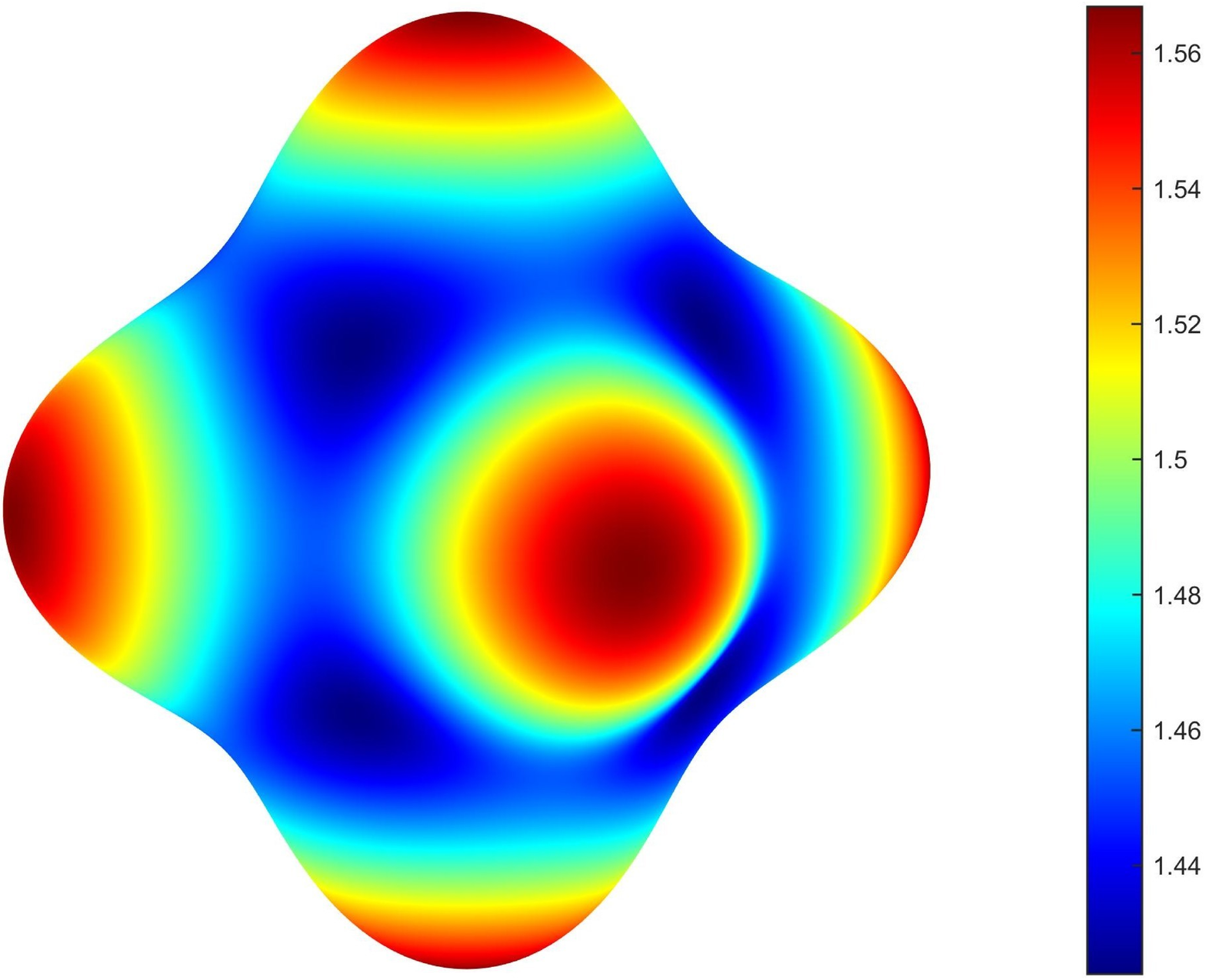}\\[1mm]
    \includegraphics[trim = 0mm 0mm 0mm 1mm, width=1.05\textwidth]{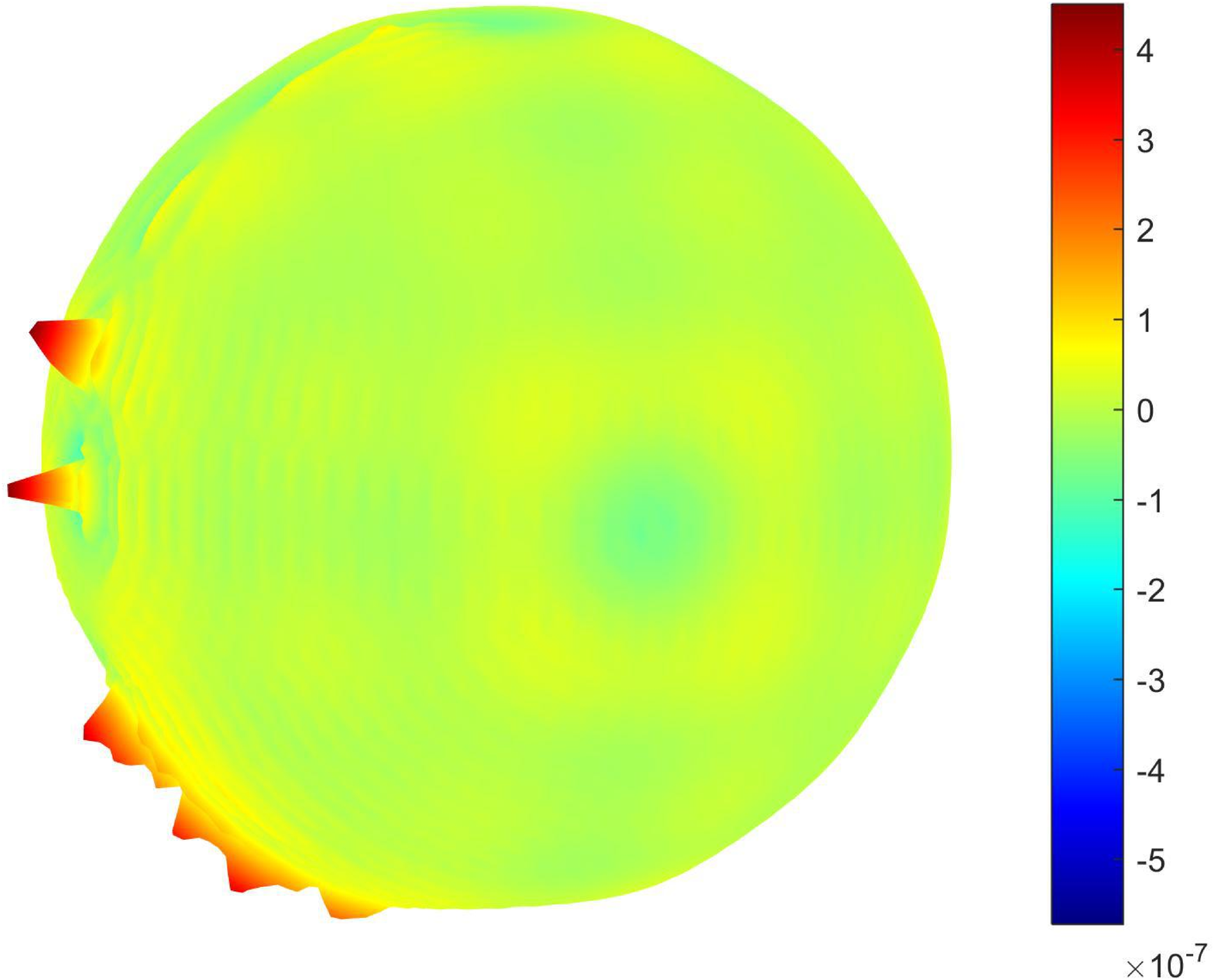}\\[1mm]
      \includegraphics[trim = 0mm 0mm 0mm 1mm, width=1.05\textwidth]{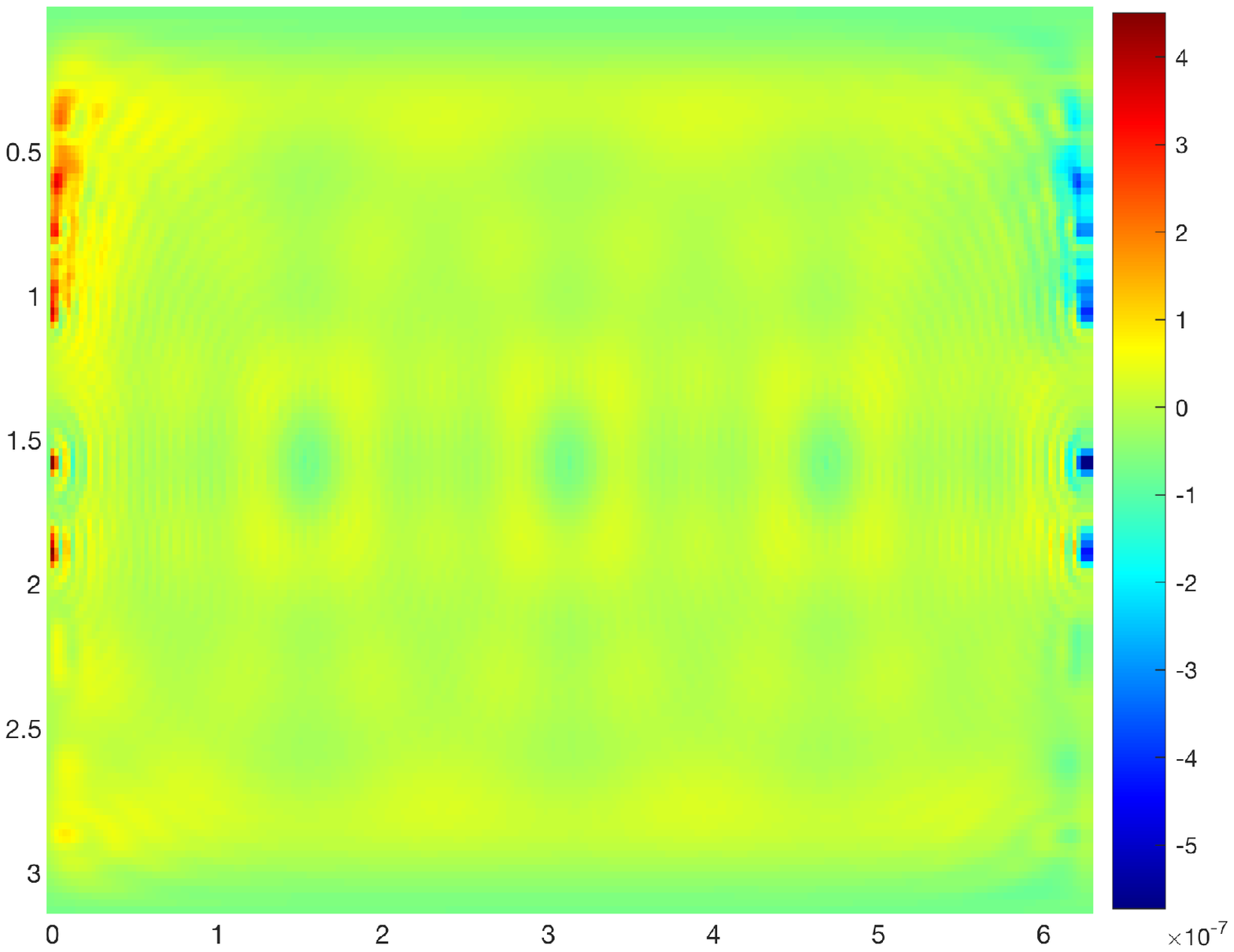}\\[1mm]
  \subcaption{GL ($N_J=32,460$)}\label{fig:RBF.k2.xa}
  \end{minipage}
  \begin{minipage}{0.24\textwidth}
  \centering
  \includegraphics[trim = 0mm 0mm 0mm 1mm, width=0.9\textwidth]{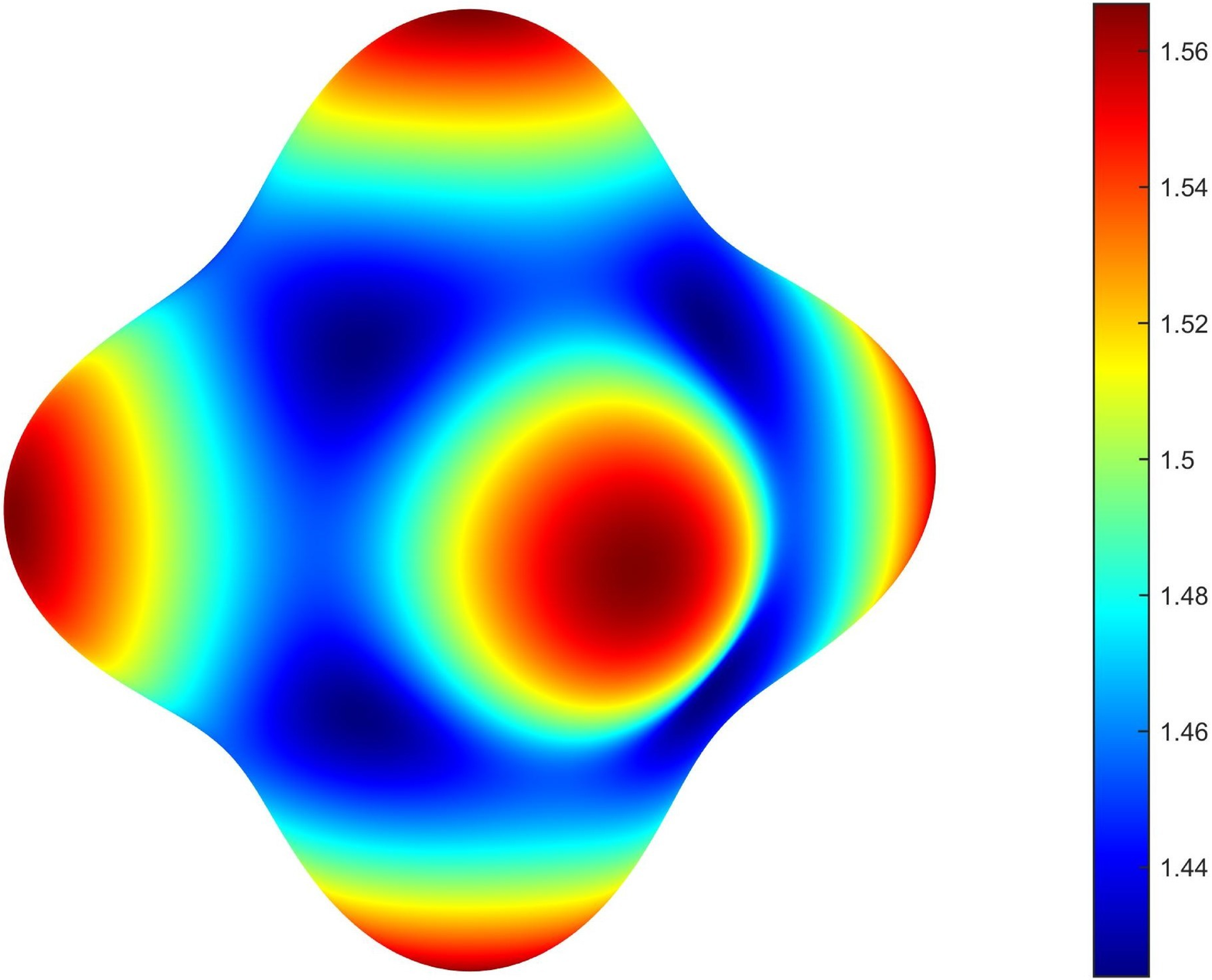}\\[1mm]
    \includegraphics[trim = 0mm 0mm 0mm 1mm, width=1.05\textwidth]{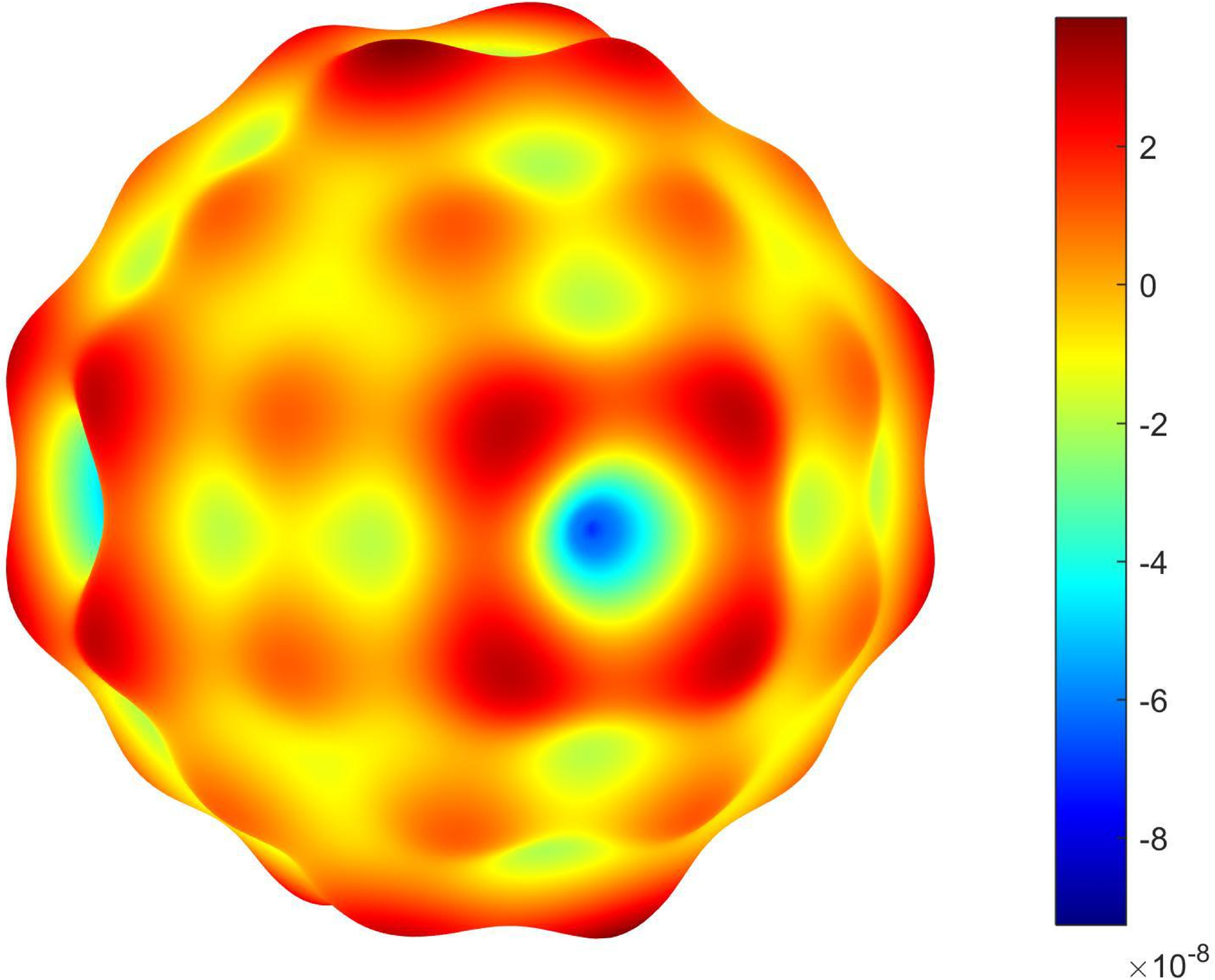}\\[1mm]
    \includegraphics[trim = 0mm 0mm 0mm 1mm, width=1.05\textwidth]{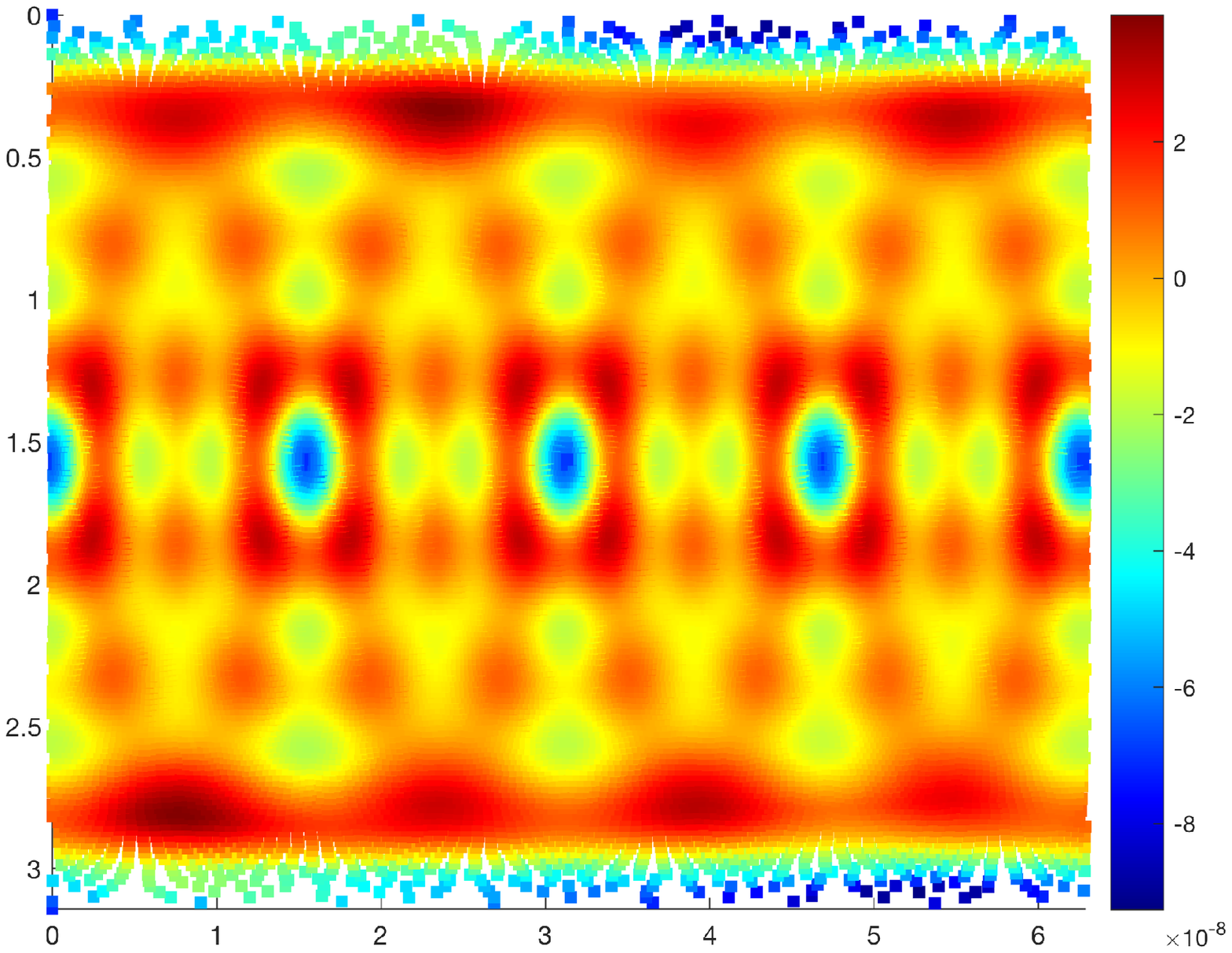}\\[1mm]
  \subcaption{SD ($N_J=32,462$)}\label{fig:RBF.k2.xb}
  \end{minipage}
    \begin{minipage}{0.24\textwidth}
    \centering
  \includegraphics[trim = 0mm 0mm 0mm 1mm, width=0.9\textwidth]{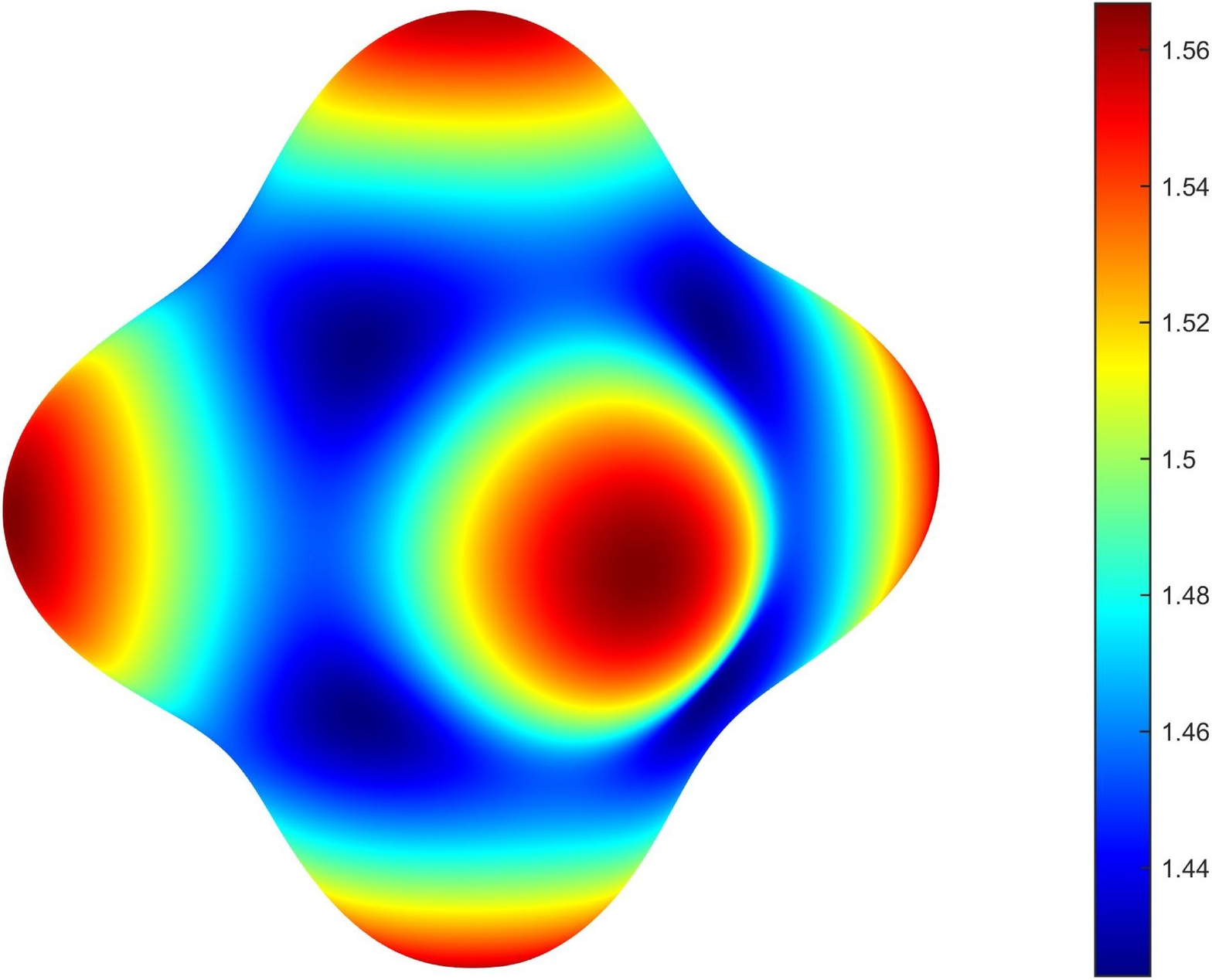}\\[1mm]
    \includegraphics[trim = 0mm 0mm 0mm 1mm, width=1.05\textwidth]{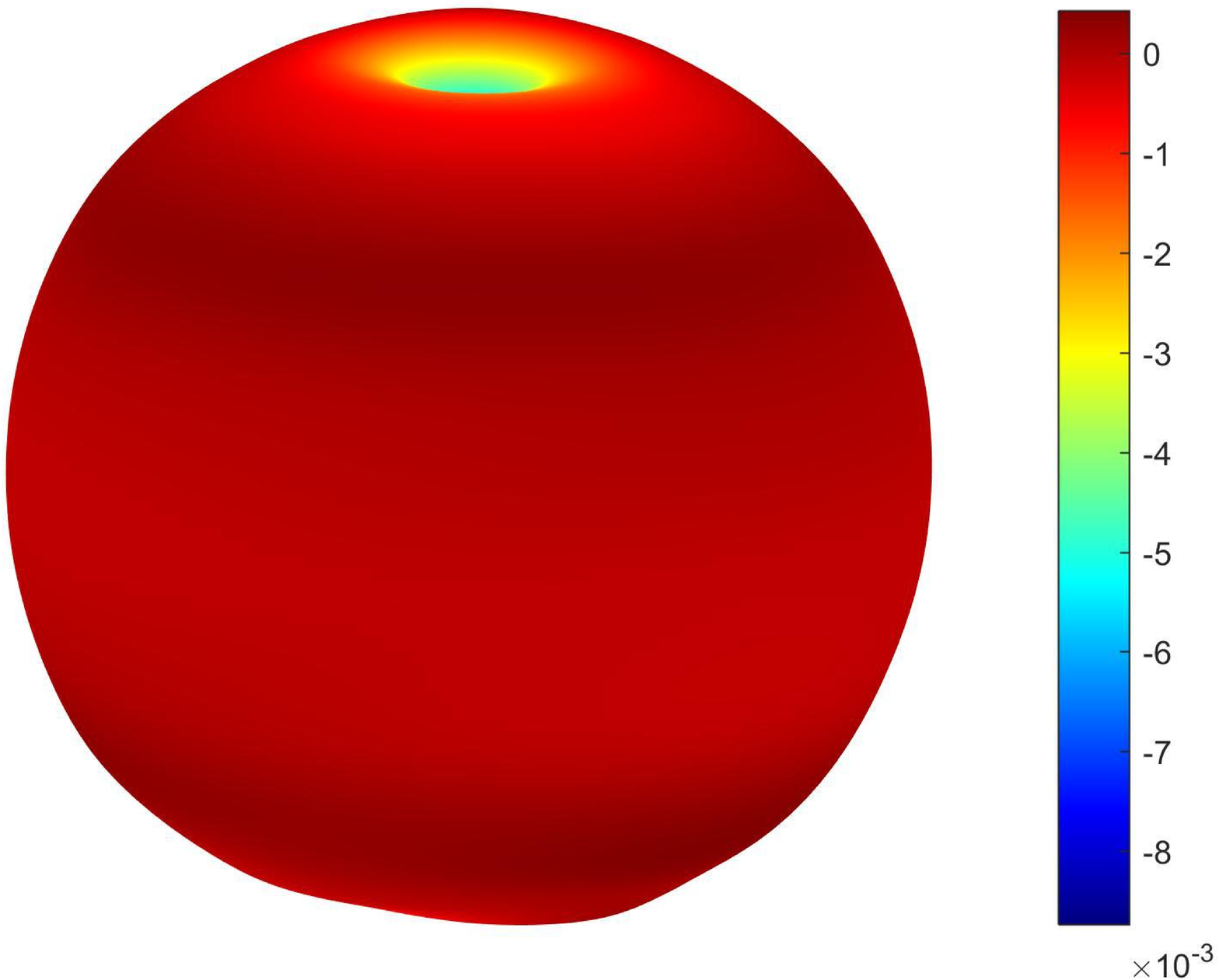}\\[1mm]
    \includegraphics[trim = 0mm 0mm 0mm 1mm, width=1.05\textwidth]{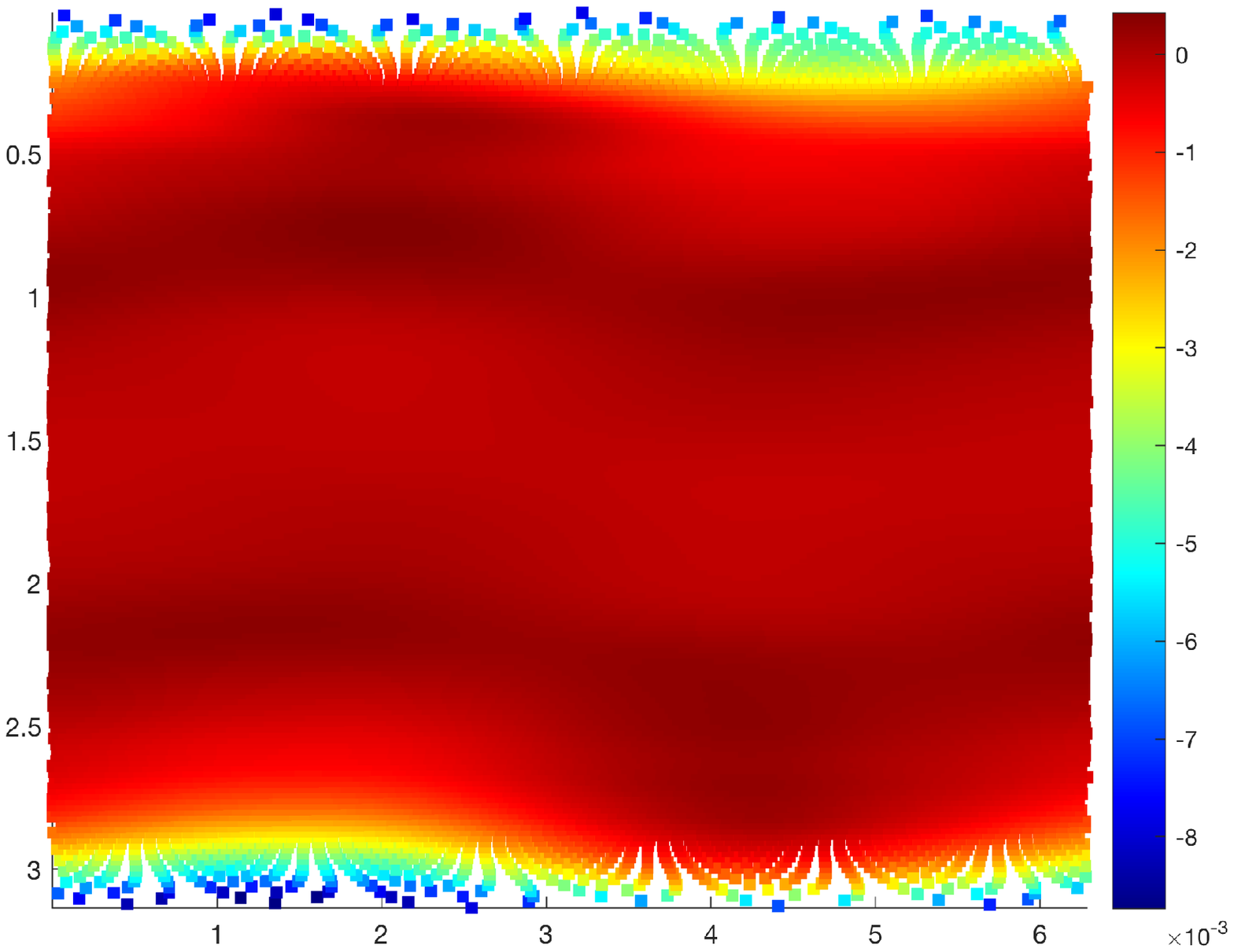}\\[1mm]
  \subcaption{SP ($N_J=32,768$)}\label{fig:RBF.k2.xc}
  \end{minipage}
    \begin{minipage}{0.24\textwidth}
    \centering
  \includegraphics[trim = 0mm 0mm 0mm 1mm, width=0.9\textwidth]{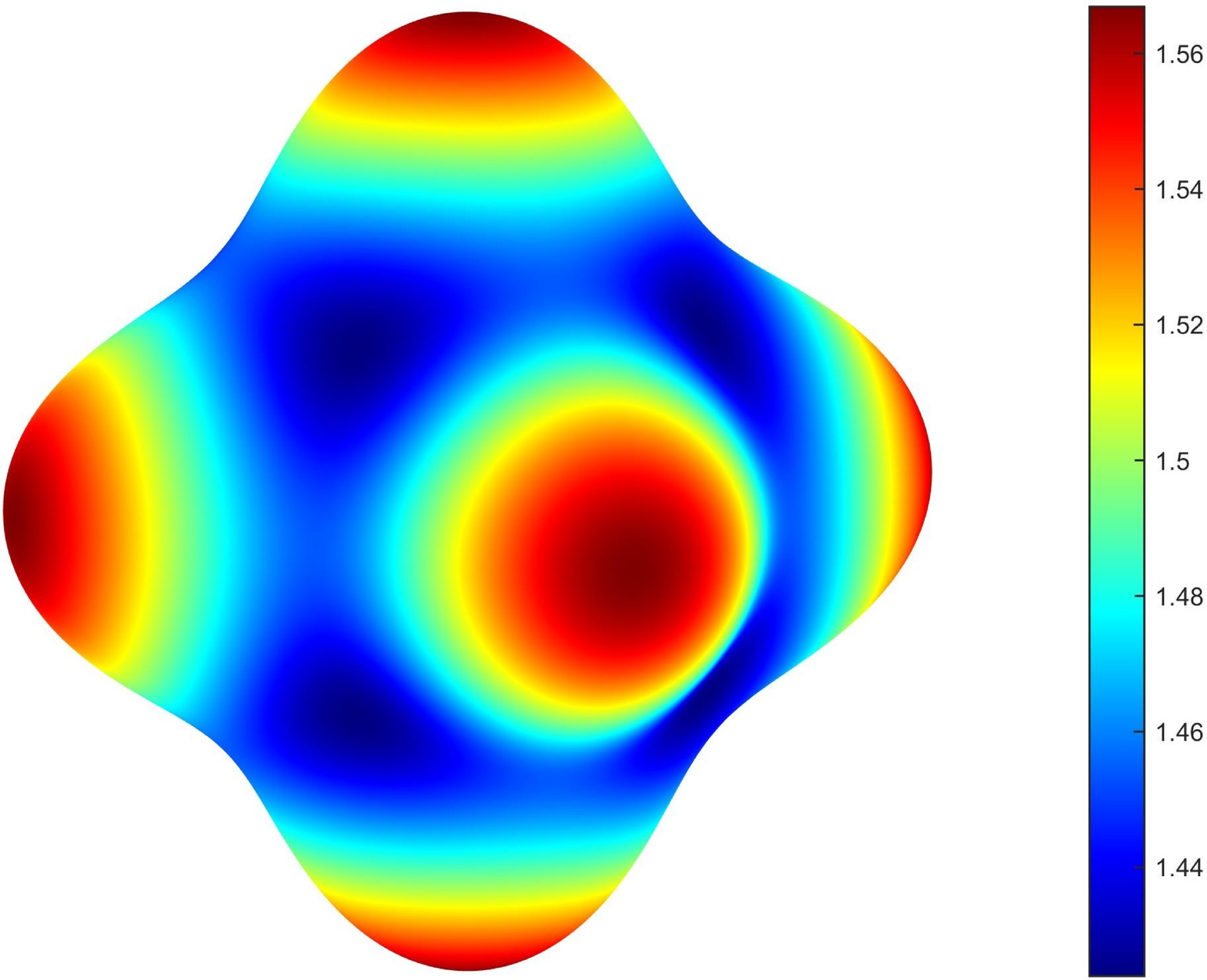}\\[1mm]
    \includegraphics[trim = 0mm 0mm 0mm 1mm, width=1.05\textwidth]{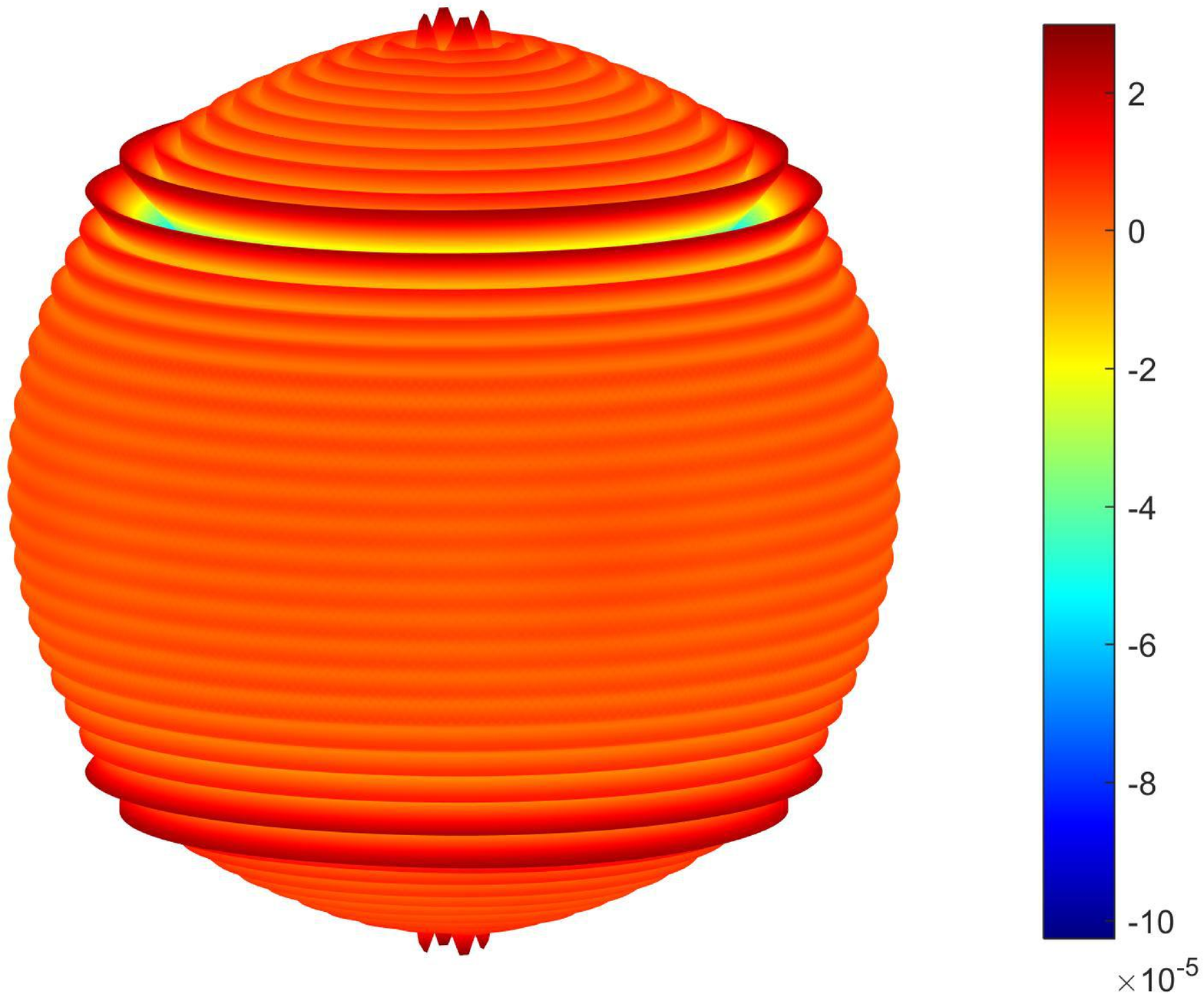}\\[1mm]
    \includegraphics[trim = 0mm 0mm 0mm 1mm, width=1.05\textwidth]{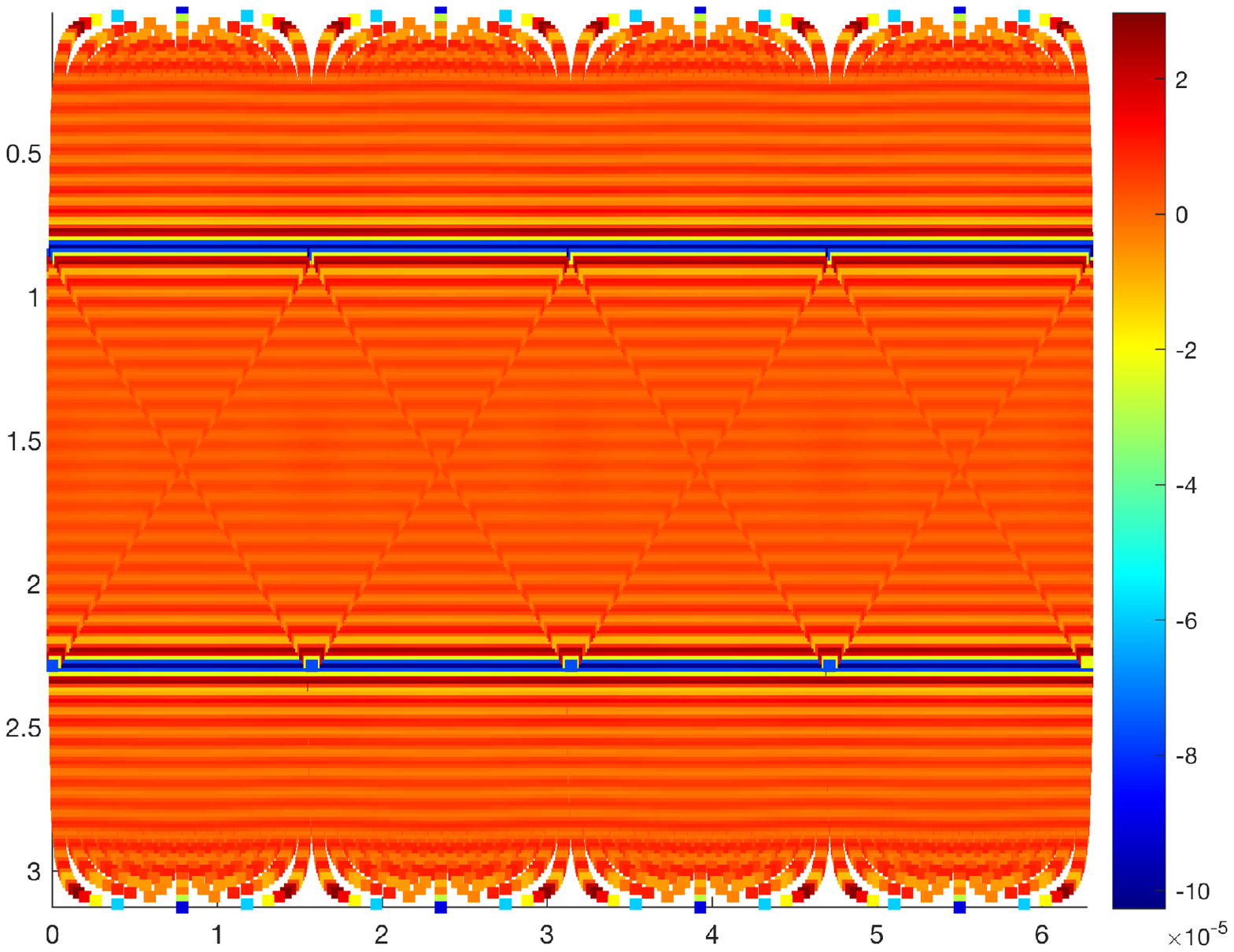}\\[1mm]
  \subcaption{HL ($N_J=49,152$)}\label{fig:RBF.k2.xd}
  \end{minipage}
  \end{minipage}
  \begin{minipage}{0.8\textwidth}
  	\vspace{1mm}
\caption{Projection term $\fracoev[\ord]$ (top row), error term $\frbcoev[\ord]{}$ (middle row), and the equirectangular projection of the error for RBF $f_{2}$ on $\sph{2}$ using different  quadrature rules, $\ord=7$.
}
\label{fig:RBF}
\end{minipage}
\end{minipage}
\end{figure}

Given a point set $\QN[N_J]$, we use $f_n$ on $\QN[N_J]$ as the data sequence $\fracoev[]:=(\fracoev[k])_{k=1}^{N_{\ord}}$, i.e. $\fracoev[k] = f_n(\pN)$, and compute the projection $\fracoev[\ord]$ and projection error $\projerr$ where $\fracoev[]=\fracoev[\ord]+\projerr$.

Figures~\ref{fig:RBF.k2.xa} -- \ref{fig:RBF.k2.xd} show the 3D view pictures of projection $\fracoev[\ord]$ (top row), error $\frbcoev[\ord]{}$ (middle row), and the equirectangular projection of the error (bottom row), using the four types of quadrature rules for $f_2$. We observe that the distributions of errors are partly due to the collective effect of the NFSFT algorithms and the points sets used in {\fmt}. We can observe that the errors by {\fmt} with different quadrature rules show distinct distribution patterns.

Table~\ref{tab:fmt.err} shows the relative $L_2$-error $\frac{\|\fracoev[]-\fracoev[\ord]\|}{\|\fracoev[]\|}$ (with Frobenius-norm) of the projections using the four types of point sets (1)--(4) in Figure~\ref{fig:QN}. The quadrature rules $\QN[N_J]$ with $J=7$ for GL ($N_J = 32,640$) and SD ($N_J=32,642$) are polynomial-exact quadrature rules of degree  $n = 255$.

We observe that SD incurs smaller approximation errors than GL.  The point sets $\QN[N_J]$ with $J=7$  for  SP ($N_J = 32,768$) and HL ($N_J=49,152$) which are \emph{not} polynomial-exact quadrature rules give worse approximation results than GL and SD. This demonstrates that using the polynomial-exact quadrature rules for framelets is more effective than using the non-polynomial-exact quadrature rules. Also, with the increase of the smoothness of the function $f_{n}$, the approximation error of the tight framelets with polynomial-exact quadrature rules (GL and SD) becomes smaller. 

\begin{remark}{
\rm
The fact that the lack of polynomial-exactness of the quadrature for framelets leads to noticeably worse approximation errors was also observed in \cite{LeMh2006}. 
The dependence of $L_{2}$ approximation errors of the tight framelets on smoothness of function space is consistent with that of the filtered approximation on $\sph{2}$, see \cite{Mh2005,NaPeWa2006-1,SlWo2012,WaLeSlWo2017}.
}
\end{remark}

\begin{table}[htb]
\centering
\begin{minipage}{0.9\textwidth}
\centering
\begin{tabular}{l*{10}{c}c}
\toprule
$\QN[N_J]$  &   $f_{0}$    &   $f_{1}$   &   $f_{2}$    &  $f_{3}$     &   $f_{4}$      \\
\midrule
 GL  (32,640)  &  3.9572e-05  & 1.0630e-07 & 1.9294e-08 & 1.6813e-08 &   1.6681e-08  \\
 SD (32,642)  &  6.2013e-05  &  9.8473e-08 &  1.0125e-08 &  3.6211e-09 &  2.9568e-09 \\
 SP (32,768) &  5.0854e-04 &  4.8888e-04 & 4.8297e-04 & 4.8112e-04 &  4.8053e-04 \\
 HL (49,152) &  4.2954e-05 &  1.1370e-05 & 1.1449e-05 & 1.1421e-05  & 1.1453e-05 \\
\bottomrule
\end{tabular}
\end{minipage}
\begin{minipage}{0.8\textwidth}
\vspace{3mm}
\caption{Relative $L_{2}$-errors of {\fmt} for  GL ($N_J=32,640$, exact for degree up to $n=255$), SD ($N_J=32,642$, exact for degree up to $n=255$), SP ($N_J = 32,768$), and HL ($N_J = 49,152$). SP and HL are with equal weights.}
\label{tab:fmt.err}
\end{minipage}
\end{table}
}
\end{example}

\begin{example}[Multiple high-pass filters]
\label{ex:SNR}
{\rm
To illustrate the role of using multiple high-pass filters played in a framelet system, we show a denoising experiment for restoring the signal $f_4$ from a noisy signal $f=f_4+g$ using three different filter banks. Here $f_4$ is given in \eqref{eq:Phi} and $g$ is a Gaussian white noise  $N(0,\sigma^2)$ with standard deviation $\sigma$.

We sample $f_4$ on the GL quadrature rules $\QN[N_J]$ with $J=6$ to obtain a signal $\fracoev[6]$ on the sphere and then add the Gaussian noise $g$ with standard deviation  $\sigma:=\sigma_\theta := \theta\max_{\PT{x}\in \QN[N_J]}{f_4(\PT{x})}$, where $\theta$ is a parameter ranging from $0.05$ to $0.20$ to control the noise level $\sigma_{\theta}$, that is, we choose $\sigma_{\theta}$ to be $5$ to $20$ percent of the maximal value of $f_4$.

Let $\chi_{[c_L,c_R];\epsilon_L,\epsilon_R}$ be the function supported on $[c_L-\epsilon_L,c_R+\epsilon_R]$ as defined in \cite[Eq.~3.1]{HaZhZh2016}. We construct three different filter banks $\filtbk_1, \filtbk_2$ and $\filtbk_3$ with $1, 2$ and $3$ high-pass filters: the filter bank $\filtbk_1=\{a;b^1_1\}$ determined by $\FT{a}:=\chi_{[-3/16,3/16]; 1/16,1/16}$ and $\FT{b}_1^1:=\chi_{[3/16,9/16]; 1/16,1/16}$, the filter bank $\filtbk_2=\{a;b^1_2,b^2_2\}$ by $\FT{b}_2^1:=\chi_{[3/16,3/8]; 1/16,1/8}$ and $\FT{b}_2^1:=\chi_{[3/8,9/16]; 1/16,1/16}$ and the filter bank $\filtbk_3=\{a;b^1_3,b^2_3,b^3_3\}$ by $\FT{b}_3^1:=\chi_{[3/16,5/16]; 1/16,1/16}$, $\FT{b}_3^2:=\chi_{[5/16,7/16]; 1/16,1/16}$, and $\FT{b}_3^3:=\chi_{[7/16,9/16]; 1/16,1/16}$. Sharing a low-pass filter $a$, each filter bank $\filtbk_{i}$ corresponds to a framelet system $\frsys(\Psi_k,\QQ)$ on the sphere, similar to $\frsys(\Psi,\QQ)$ in Subsection~\ref{sec:fmtS2}.

Given the noisy data $\fracoev[\theta]=\fracoev[6]+g$ with noise level $\sigma_\theta$ and a filter bank $\filtbk_{i}$, we apply Algorithm~\ref{algo:decomp.multi.level} to $\fracoev[\theta]$ with $J=6$ and $J_0=4$. We use a simple hard thresholding technique to the corresponding output high-pass (filtered) coefficient sequences with threshold value same as $\sigma_\theta$ and then apply Algorithm~\ref{algo:reconstr.multi.level} to the thresholded coefficient sequences and obtain a reconstructed signal $\rfrav[6]$. The performance of a framelet system for denoising is measured by the signal-to-noise ratio (with unit dB), denoted by $\mathrm{SNR}(\fracoev[6],\rfrav[6]):=20\log_{10}\frac{\|\fracoev[6]\|}{\|\rfrav[6]-\fracoev[6]\|}$. The larger $\mathrm{SNR}$, the more effective the framelet system for denoising is.

The results are reported in Table~\ref{tab:more-filters}. We observe that the filter bank $\filtbk_2$ brings more than $1$ dB improvement compared to $\filtbk_1$ by splitting $b_1^1$ to $b_2^1$ and $b_2^2$, and the use of $\filtbk_3$ brings about $0.5$ dB improvement compared to $\filtbk_2$. Note that we do not make any hard thresholding on the low-pass filter coefficient sequences. The results that $\filtbk_3$ outperforms $\filtbk_2$ and $\filtbk_2$ outperforms $\filtbk_1$ illustrate the advantage of using multiple high-pass filters in a framelet system for denoising. Also, using multiple high-pass filters allows more free parameters in the filter bank and more flexibility of the design of high-pass filters. 

\begin{table}[htb]
\centering
\begin{minipage}{0.9\textwidth}
\centering
\begin{tabular}{l*{10}{c}c}
\toprule
$\theta$ &  $\mathrm{SNR}(\fracoev[6],{\fracoev[\theta]})$ & $\filtbk_{1}$    &   $\filtbk_{2}$   &   $\filtbk_{3}$      \\
\midrule
 0.05  & 17.12 &  19.58  & 20.82 & 21.25   \\
 0.10  & 11.09 &  13.66  & 14.92 & 15.37   \\
 0.15  & 7.57  &  10.25  & 11.54 & 12.00   \\
 0.20  & 5.07  &  7.78   & 9.09  & 9.56    \\
\bottomrule
\end{tabular}
\end{minipage}
\begin{minipage}{0.8\textwidth}
\vspace{3mm}
\caption{Denoising performance in terms of SNR (dB) by the filter banks
$\filtbk_{1}=\{a;b^1_1\}$, $\filtbk_{2}=\{a; b^1_2,b^2_2\}$, $\filtbk_{3}=\{a;b^1_3,b^2_3,b^3_3\}$. The first column $\theta$ ranges from $0.05$ to $0.20$. The second column is the SNR of the original signal $\fracoev[6]$ and the noisy signal $\fracoev[\theta]$. The third, fourth and fifth columns are the SNR of the original signal $\fracoev[6]$ and the reconstructed signal $\rfrav[6]$ for the filter banks $\filtbk_1, \filtbk_2$ and $\filtbk_3$.
}
\label{tab:more-filters}
\end{minipage}
\end{table}

}
\end{example}

\begin{example}[Multiscale analysis]
\label{ex:2}
{\rm
We use the data set ETOPO1 of Earth surface (see Figure~\ref{fig:data}) to illustrate the multiscale decomposition of the {\fmt} algorithm using the GL rules. The data set ETOPO1 for the planar earth is based on $1$ arc-minute global relief model of Earth's surface that integrates land topography and ocean bathymetry by National Centers for Environmental Information (NCEI), see \cite{AmEa2009}.

We sample the data set ETOPO1 at GL points $\QN[N_J]$ to obtain a data sequence $\fracoev[]$ (see Figure~\ref{fig:tpg}) at the scaling level $J=9$ with $N_J = 786,432$ nodes. At level $8$, the GL rule $\QN[N_8]$ has $N_8 = 196,608$ nodes. At level $7$, the GL rule $\QN[N_7]$ has $N_7 = 49,152$ nodes. With the sequence $\QN[] = \{\QN[N_j]: j=7,8,9\}$ of quadrature rules, we can define the sequence of framelet systems $\{\frsys[j](\{\scala;\scalb^1,\scalb^2\})\}_{j=7}^9$ as described in Section~\ref{sec:fmtS2}.

Applying Algorithm~\ref{algo:decomp.multi.level} with the framelet systems $\{\frsys[j](\{\scala;\scalb^1,\scalb^2\})\}_{j=7}^9$,  we  obtain the projection  $\fracoev[9]$ (see Figure~\ref{fig:tpg.fmt}) and the error  $\frbcoev[9]{}$ (see Figure~\ref{fig:tpg.err}) at the finest level $j=9$ satisfying $\fracoev[] = \fracoev[9]+\frbcoev[9]{}$.

At the level $j=8$, the projection  $\fracoev[9]$ is decomposed to the framelet approximation coefficient sequence $\fracoev[8]$ (see Figure~\ref{fig:tpg.v8}) and the framelet detail  coefficient sequences $\frbcoev[8]{1}$ and $\frbcoev[8]{2}$ (see Figures~\ref{fig:tpg.w8.1} and \ref{fig:tpg.w8.2}).

At the level $j= 7$, the approximation $\fracoev[8]$ is further decomposed to
$\fracoev[7]$,  $\frbcoev[7]{1}$, and $\frbcoev[7]{2}$ (see Figures~\ref{fig:tpg.v7} -- \ref{fig:tpg.w7.2}).

The pictures in Figure~\ref{fig:tpg:GL} show that the framelet systems can decompose the input data into a good data approximation and elaborate data details at different resolutions.
The higher-level projection gives the picture with higher resolution and incurs the smaller projection error. The pictures also verify the multiresolution structure of a sequence of tight framelet systems and thus demonstrate the ability of {\fmt} for multiscale data analysis.

\begin{figure}[htb]
\begin{minipage}{\textwidth}
\centering
\begin{minipage}{\textwidth}
\centering
\begin{minipage}{\textwidth}
\centering
  \begin{minipage}{0.31\textwidth}
  \centering
  \includegraphics[trim = 0mm 0mm 0mm 1mm, width=0.9\textwidth]{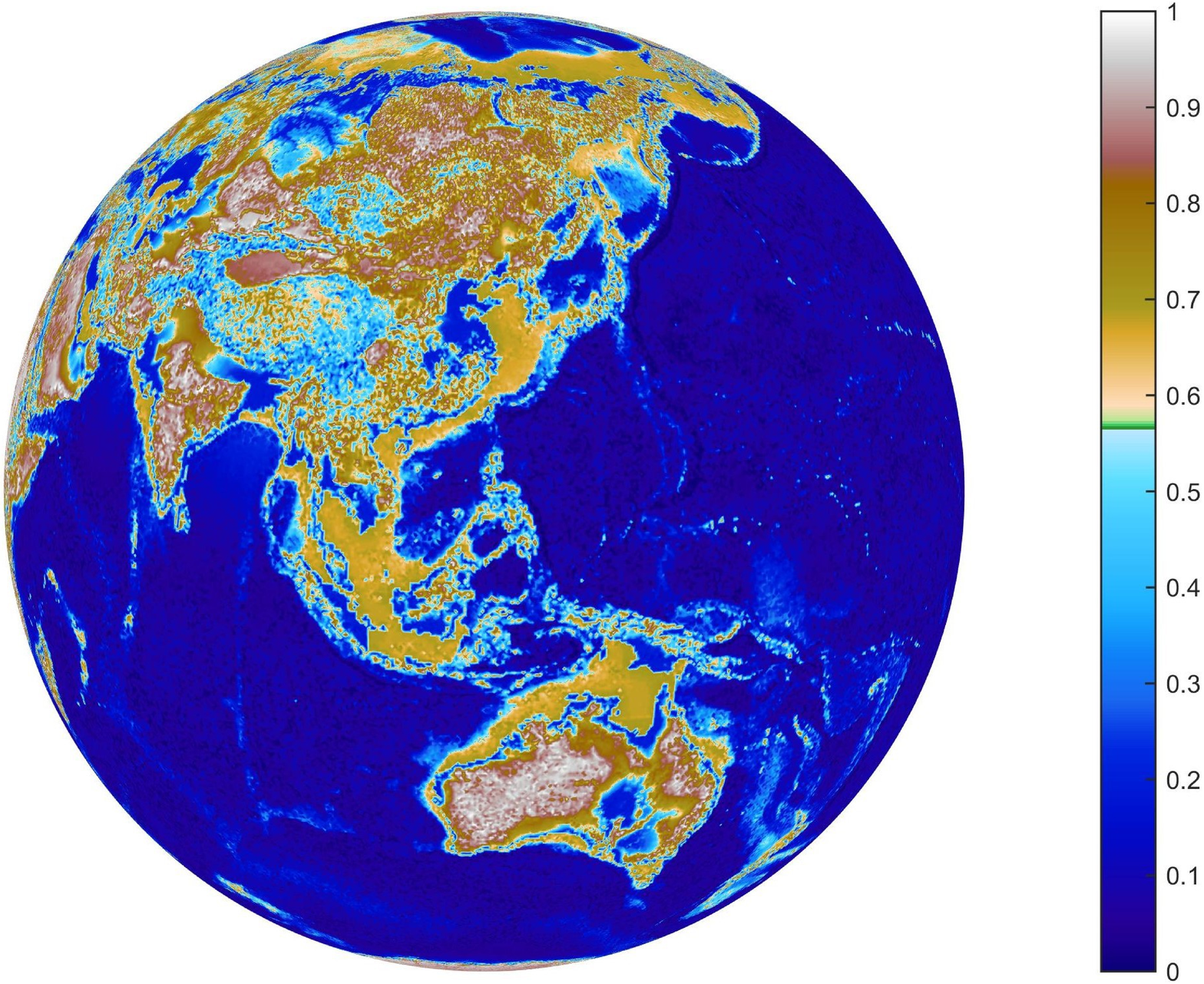}\\
  \subcaption{ETOPO1 data $\fracoev[]$}\label{fig:tpg}
  \end{minipage}
  \begin{minipage}{0.31\textwidth}
  \centering
  \includegraphics[trim = 0mm 0mm 0mm 1mm, width=0.9\textwidth]{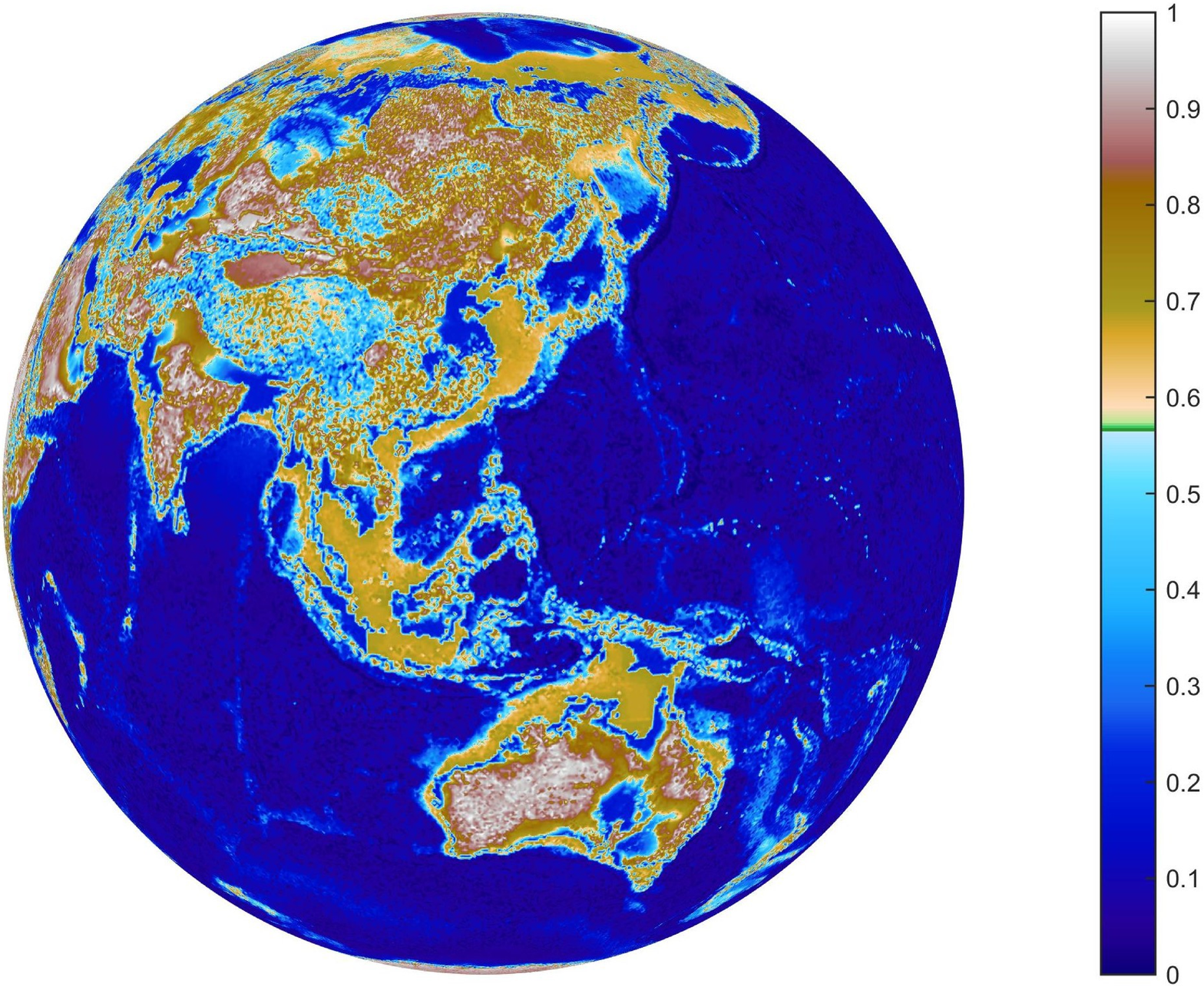}\\
  \subcaption{Projection term $\fracoev[9]$}\label{fig:tpg.fmt}
  \end{minipage}
\begin{minipage}{0.31\textwidth}
\centering
  \includegraphics[trim = 0mm 0mm 0mm 1mm, width=0.9\textwidth]{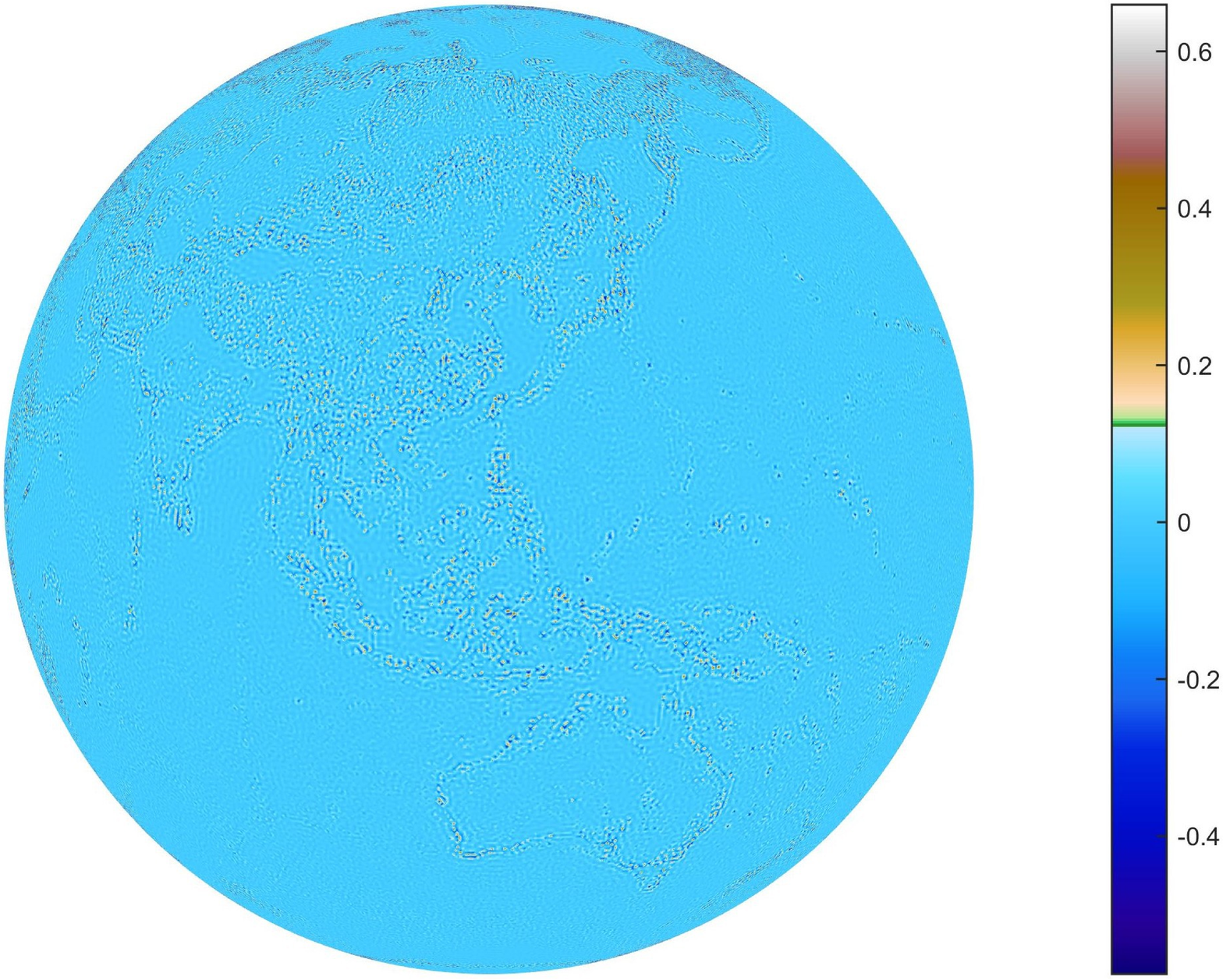}\\
  \subcaption{Error term $\frbcoev[9]{}$}\label{fig:tpg.err}
  \end{minipage}
\end{minipage}
\begin{minipage}{\textwidth}
\vspace{2mm}
\centering
  \begin{minipage}{0.31\textwidth}
  \centering
  \includegraphics[trim = 0mm 0mm 0mm 1mm, width=0.9\textwidth]{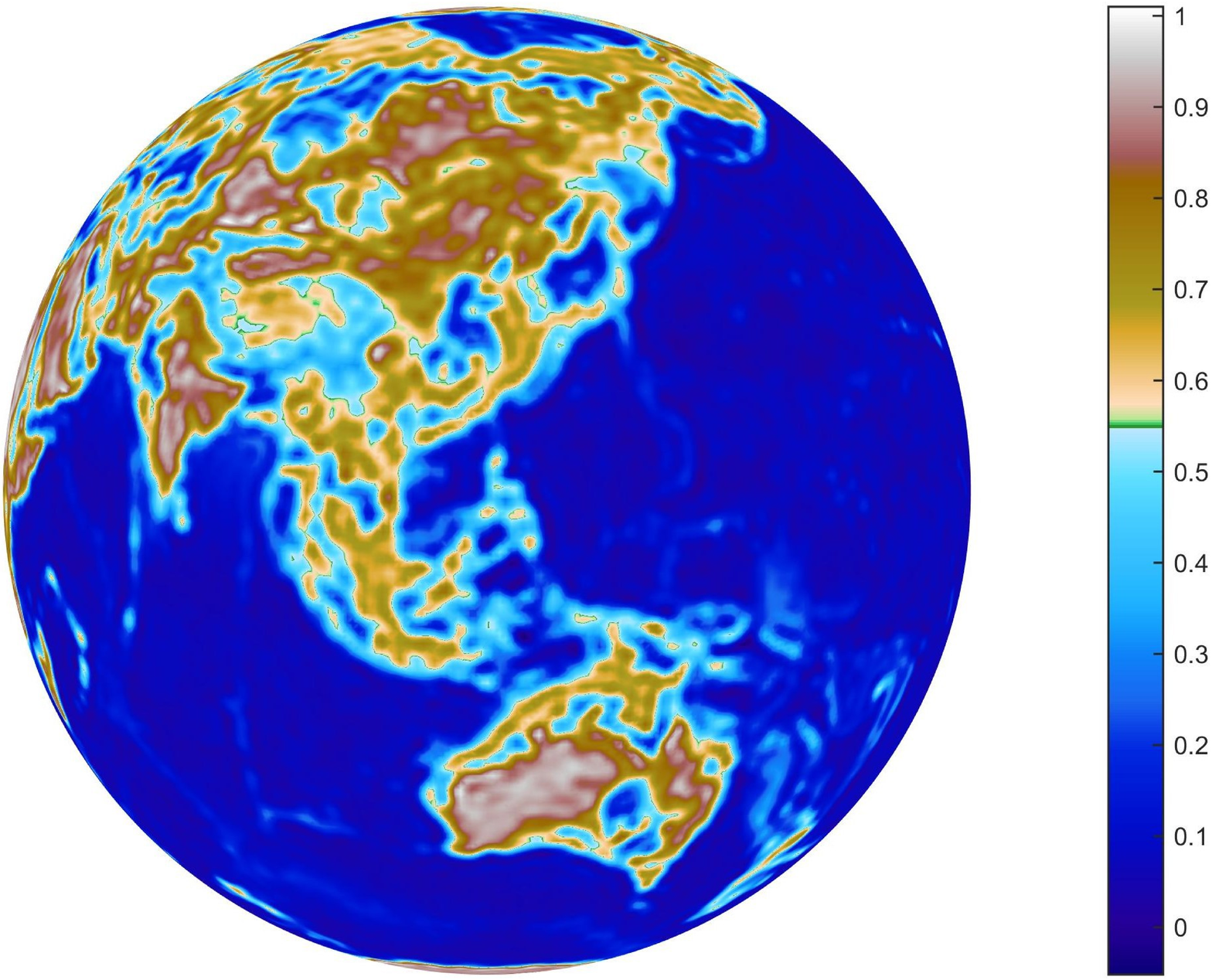}\\
  \subcaption{Approximation $\fracoev[8]$}\label{fig:tpg.v8}
  \end{minipage}
  \begin{minipage}{0.31\textwidth}
  \centering
  \includegraphics[trim = 0mm 0mm 0mm 1mm, width=0.9\textwidth]{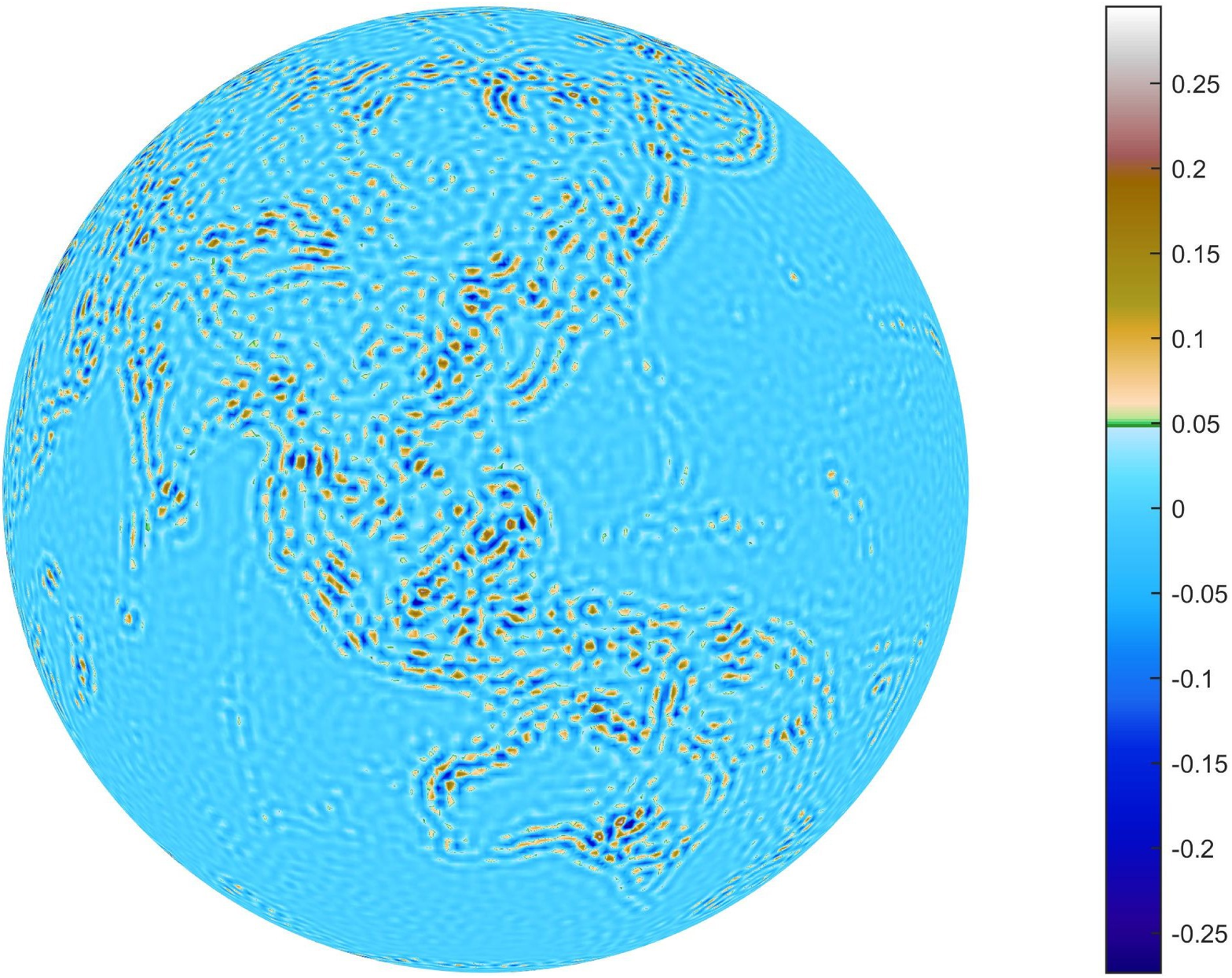}\\
  \subcaption{Detail $\frbcoev[8]{1}$}\label{fig:tpg.w8.1}
  \end{minipage}
\begin{minipage}{0.31\textwidth}
\centering
  \includegraphics[trim = 0mm 0mm 0mm 1mm, width=0.9\textwidth]{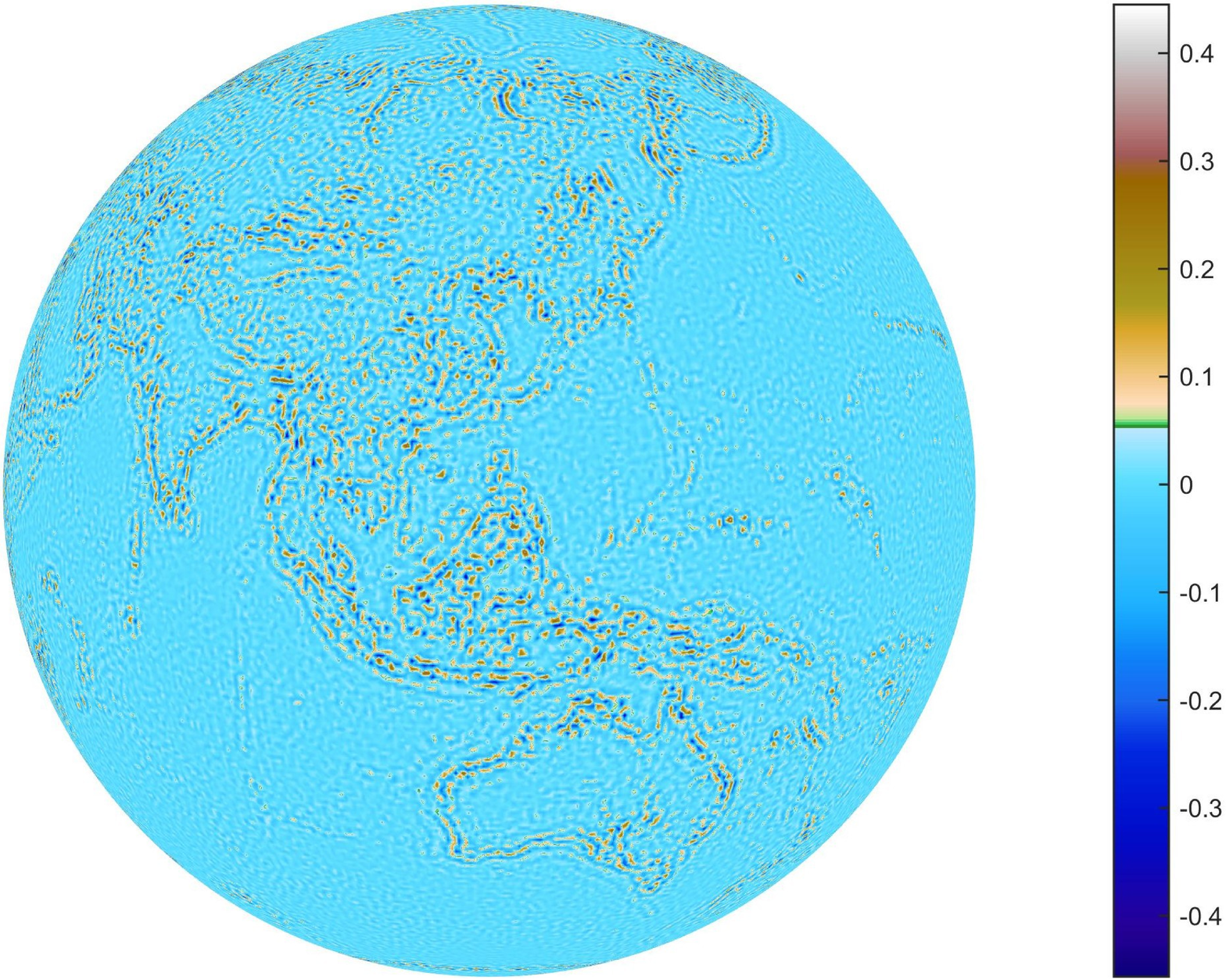}\\
  \subcaption{Detail $\frbcoev[8]{2}$}\label{fig:tpg.w8.2}
  \end{minipage}
\end{minipage}
\begin{minipage}{\textwidth}
\vspace{2mm}
\centering
  \begin{minipage}{0.31\textwidth}
  \centering
  \includegraphics[trim = 0mm 0mm 0mm 1mm, width=0.9\textwidth]{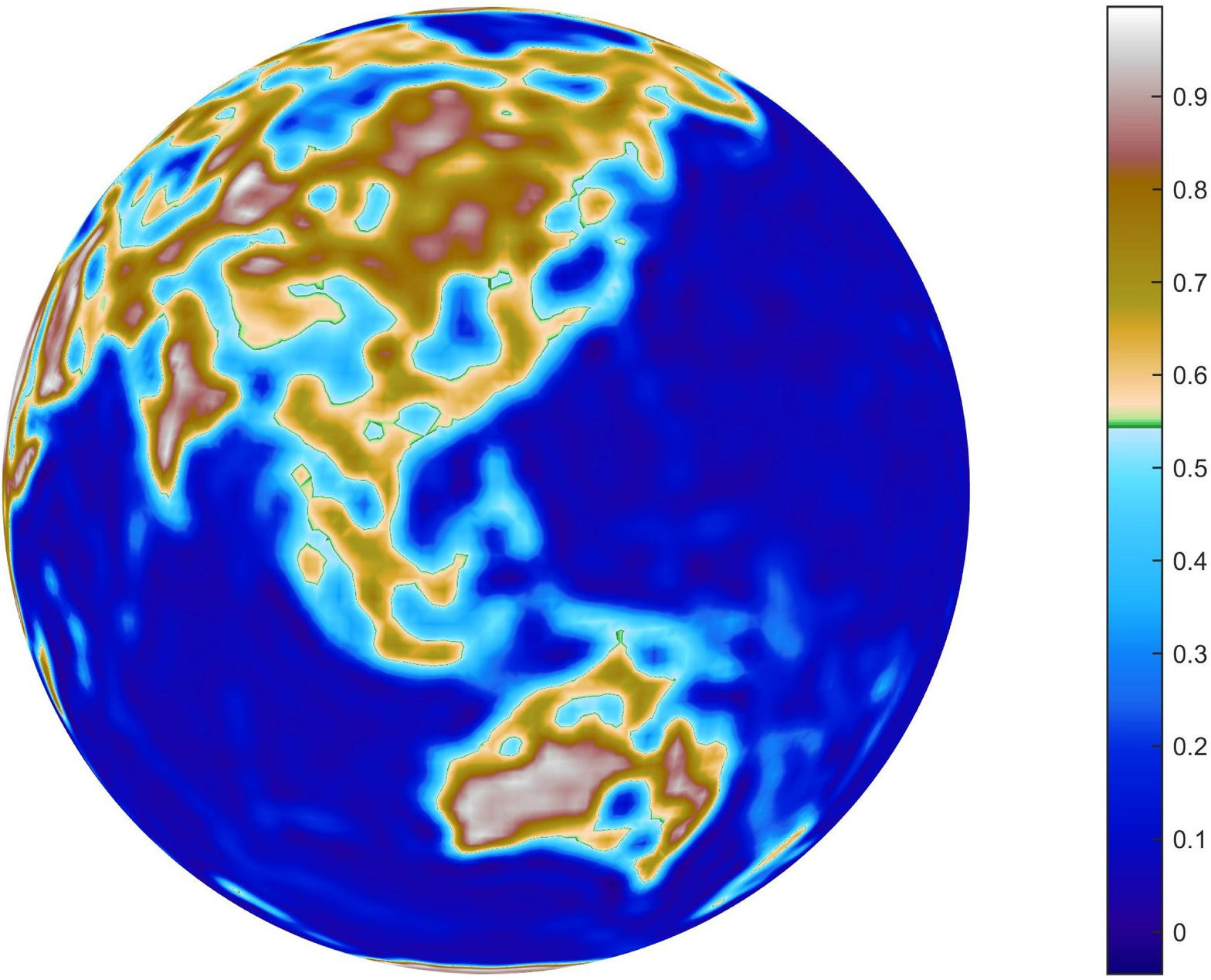}\\
  \subcaption{Approximation $\fracoev[7]$}\label{fig:tpg.v7}
  \end{minipage}
  \begin{minipage}{0.31\textwidth}
  \centering
  \includegraphics[trim = 0mm 0mm 0mm 1mm, width=0.9\textwidth]{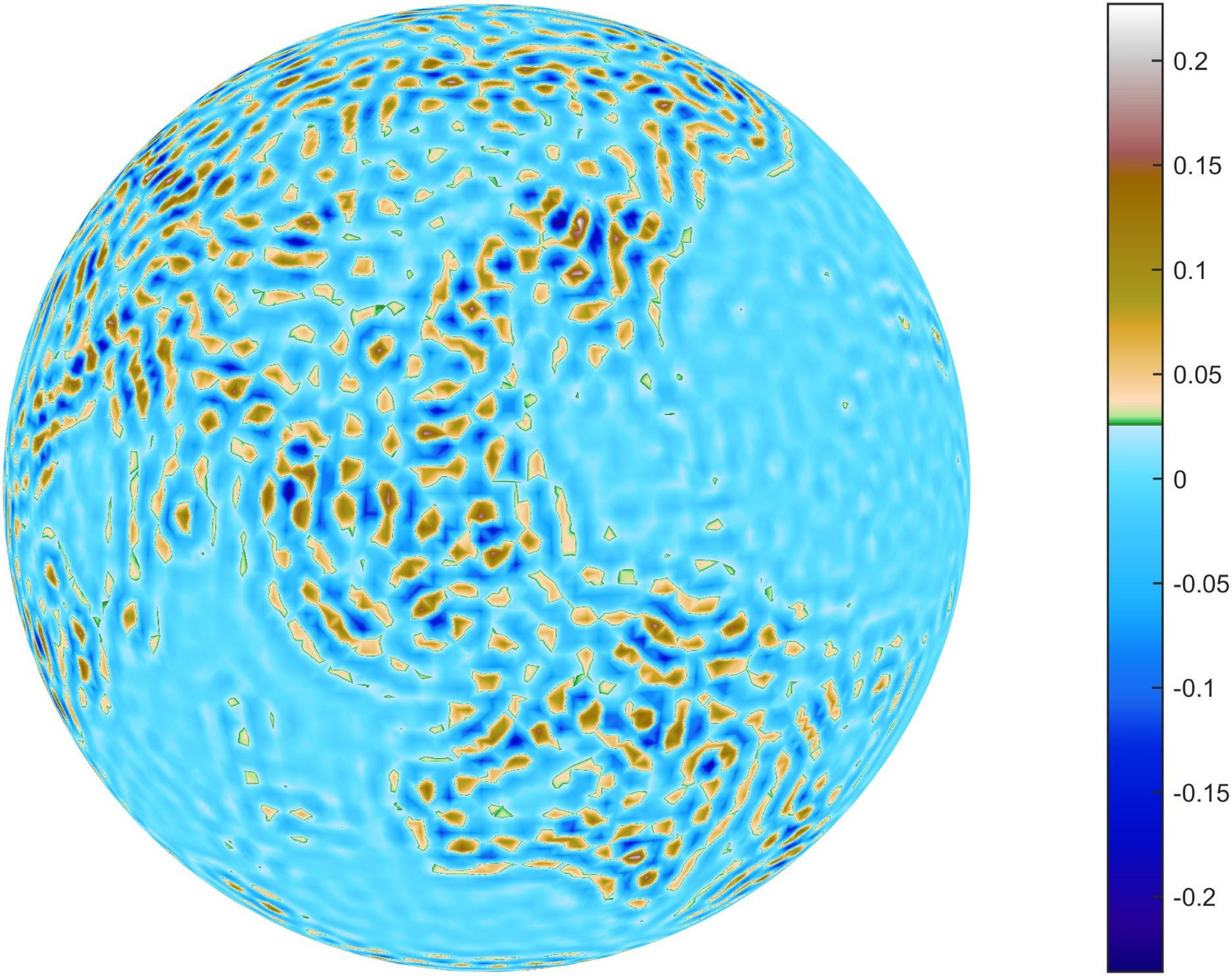}\\
  \subcaption{Detail $\frbcoev[7]{1}$}\label{fig:tpg.w7.1}
  \end{minipage}
\begin{minipage}{0.31\textwidth}
\centering
  \includegraphics[trim = 0mm 0mm 0mm 1mm, width=0.9\textwidth]{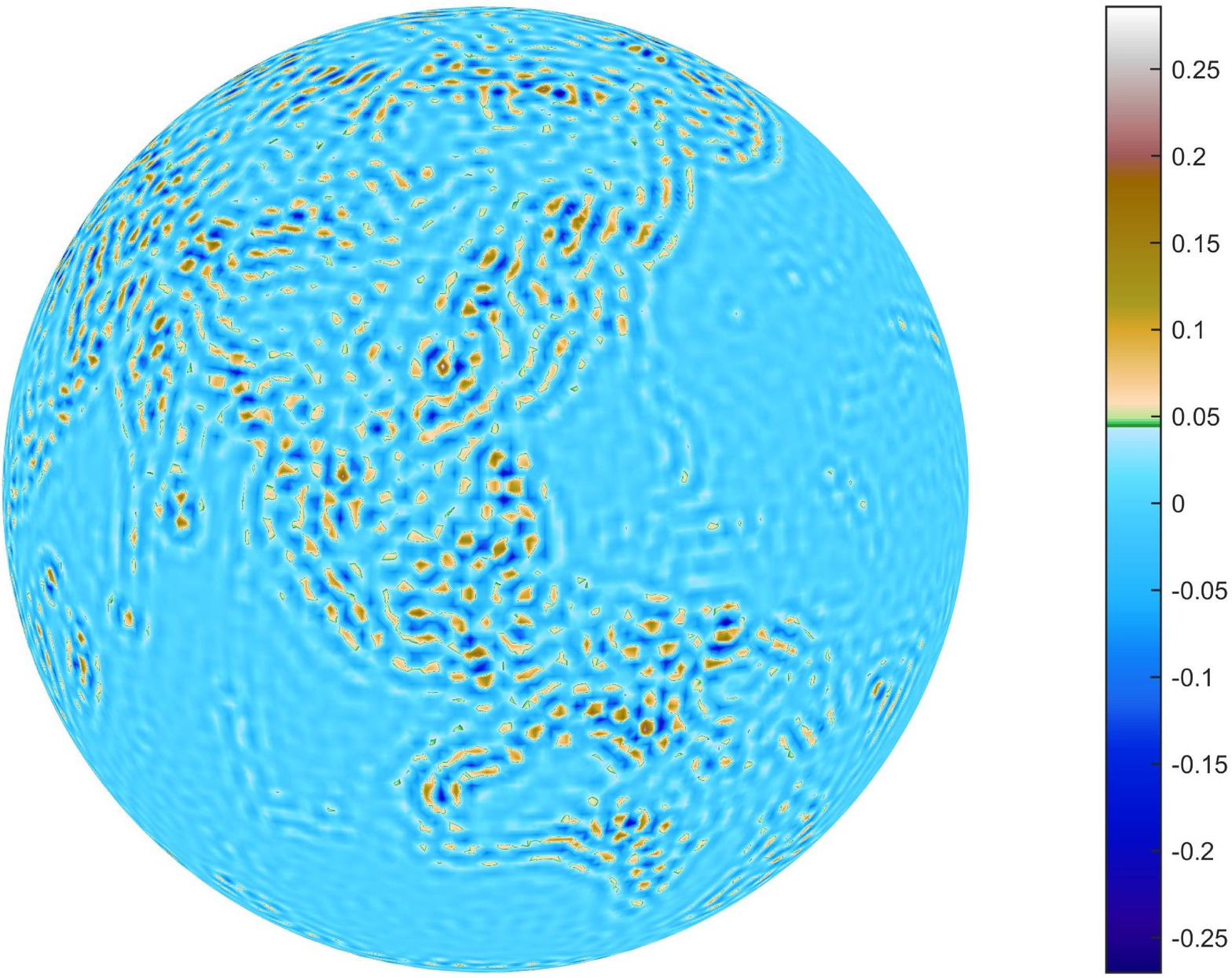}\\
  \subcaption{Detail $\frbcoev[7]{2}$}\label{fig:tpg.w7.2}
  \end{minipage}
\end{minipage}
\end{minipage}
\begin{minipage}{0.8\textwidth}
	\vspace{2mm}
\caption{Multiscale decomposition of the ETOPO1 data $\fracoev[]$ on GL rule $\QN[N_J]$ with $N_J= 523,776$ and $J=9$. $\fracoev[] = \fracoev[9]+\frbcoev[9]{}$ (top row). The projection term  $\fracoev[9]$ is decomposed as $\fracoev[8]+\frbcoev[8]{1}+\frbcoev[8]{2}$ (middle row), and the approximation $\fracoev[8]$ is further decomposed as $\fracoev[7]+\frbcoev[7]{1}+\frbcoev[7]{2}$ (bottom row).}
\label{fig:tpg:GL}
\end{minipage}
\end{minipage}
\end{figure}
}
\end{example}

\begin{example}[Computational complexity]
\label{ex:3}
{\rm
In this example, we use the CMB data set (see Figure~\ref{fig:data}) to illustrate the computational efficiency of the {\fmt} algorithm. The CMB data are collected by Plank at HEALPix (HL) points of resolution $2^{10}=1024$ with $12\times(2^{10})^2=12,582,912$ nodes, see \cite{Planck2015IX}.

The sequence of HL point sets $\QN[N_j]$ for $j=0,1,\ldots,10$ corresponds to a sequence $\{\frsys[j](\Psi)\}_{j={\ord[0]}}^{J}$ of framelet systems with ${\ord[0]}\ge0$ and with $J$ up to $10$. To illustrate the near linearity of the computational complexity for the {\fmt} algorithms, we fix ${\ord[0]}=0$ and change $J$ from $1$ to $10$. At level $J$, we use the CMB data at the nodes of $\QN[N_J]$, $N_J = 12\times (2^J)^2$ as the data sequence $\fracoev[\ord]$.

For each $J\in\{1,2,\ldots,10\}$, we test the total time, the decomposition time and the reconstruction time of the {\fmt}s described in Algorithms~\ref{algo:decomp.multi.level} and \ref{algo:reconstr.multi.level} associated with $\{\frsys[j](\Psi)\}_{j={\ord[0]}}^{J}$ for the data set $\fracoev[\ord]$. The results are reported in Table~\ref{tab:t.N.S2.FMT.CMB.HL}, where the numbers inside the brackets are the ratios of the time at level $\ord$ to the time at level $\ord-1$, $2\le\ord\le10$. From the ratios in Table~\ref{tab:t.N.S2.FMT.CMB.HL}, we observe that the computational time (decomposition, reconstruction or total CPU time) grows almost linearly with respect to the size of the data. This illustrates that the computational steps of {\fmt} are proportional to the size of the input data $\fracoev[\ord]$, which is consistent with the analysis in Section~\ref{sec:fmt}.

\begin{table}[htb]
\centering
\begin{minipage}{\textwidth}
\centering
\scriptsize
\begin{tabular}{l*{12}{c}c}
\toprule
$\ord$   &   1       &   2       &   3       &   4     &  5      &   6     &   7    &   8    & 9   & 10 \\
\midrule
 $N_J$    &   48       &     192   &     768   &    3,072   &   12,288   &   49,152   &   196,608   &   786,432 & 3,145,728  & 12,582,912 \\

 $t$    &   0.007 &0.022 (3.2) &0.048 (2.2) &0.11 (2.2) &0.28 (2.6) &1.07 (3.8) &4.39 (4.1) &20.9 (4.8) &102.4 (4.9)  & 569.8 (5.6)\\

 $t_{\mathrm{de}}$  & 0.003 &0.010 (3.1) &0.021 (2.1) &0.04 (2.0) &0.10 (2.4) &0.35 (3.3) &1.33 (3.8) &5.86 (4.4) &26.8 (4.6) & 129.9 (4.8) \\

 $t_{\mathrm{re}}$  &    0.003 &0.012 (3.4) &0.027 (2.3) &0.06 (2.4) &0.18 (2.8) &0.72 (4.1) &3.06 (4.2) &15.0 (4.9) &75.5 (5.0) & 439.9 (5.8) \\
\bottomrule
\end{tabular}
\end{minipage}
\begin{minipage}{0.9\textwidth}
\vspace{3mm}
\caption{{\fmt} CPU time v.s. number of input for CMB data (rounded).  Quadrature rules in $\{\QN\}_{j={\ord[0]}}^J$ are the HEALPix points with equal weights,  ${\ord[0]}=0$, and $1\le J\le 10$. The row of $t_{\mathrm{de}}$ is the CPU time of decomposition, the row of $t_{\mathrm{re}}$ is the CPU time of reconstruction and the row of $t=t_{\mathrm{de}}+t_{\mathrm{re}}$ is the total CPU  time.
The numbers inside brackets are the ratios $\frac{t(N_{J})}{t(N_{J-1})}$ (or $\frac{t_{\mathrm{de}}(N_{J})}{t_{\mathrm{de}}(N_{J-1})}$ or $\frac{t_{\mathrm{re}}(N_{J})}{t_{\mathrm{re}}(N_{J-1})}$) of  CPU time $t(N_J)$ (or $t_{\mathrm{de}}(N_{J})$ or $t_{\mathrm{re}}(N_{J})$) of level $J$ to the CPU time $t(N_{J-1})$ of level $J-1$. The numerical test is run under Intel Core i7 CPU @ 3.4GHz with 32GB RAM in OS X EI Capitan.
 }
\label{tab:t.N.S2.FMT.CMB.HL}
\end{minipage}
\end{table}

\begin{figure}[htb]
\begin{minipage}{\textwidth}
  \centering
  \begin{minipage}{0.37\textwidth}
  \centering
  \includegraphics[trim = 0mm 0mm 0mm 0mm, width=7cm]{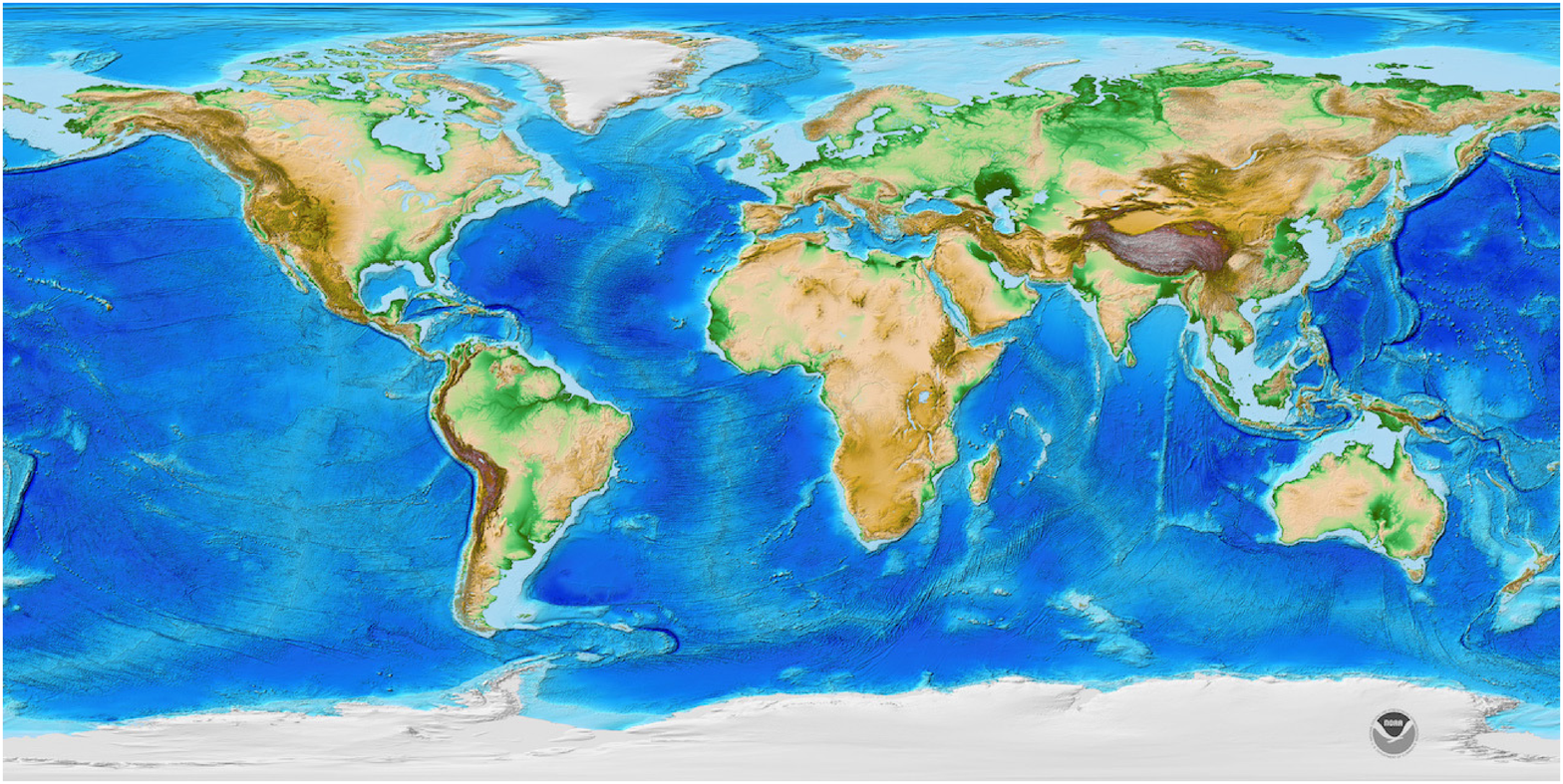}\\[1mm]
  \end{minipage}
  \begin{minipage}{0.05\textwidth}
  \centering
  \hspace{.3\textwidth}
  \end{minipage}
  \begin{minipage}{0.37\textwidth}
  \centering
  \includegraphics[trim =0mm 0mm 0mm 0mm, width=5cm]{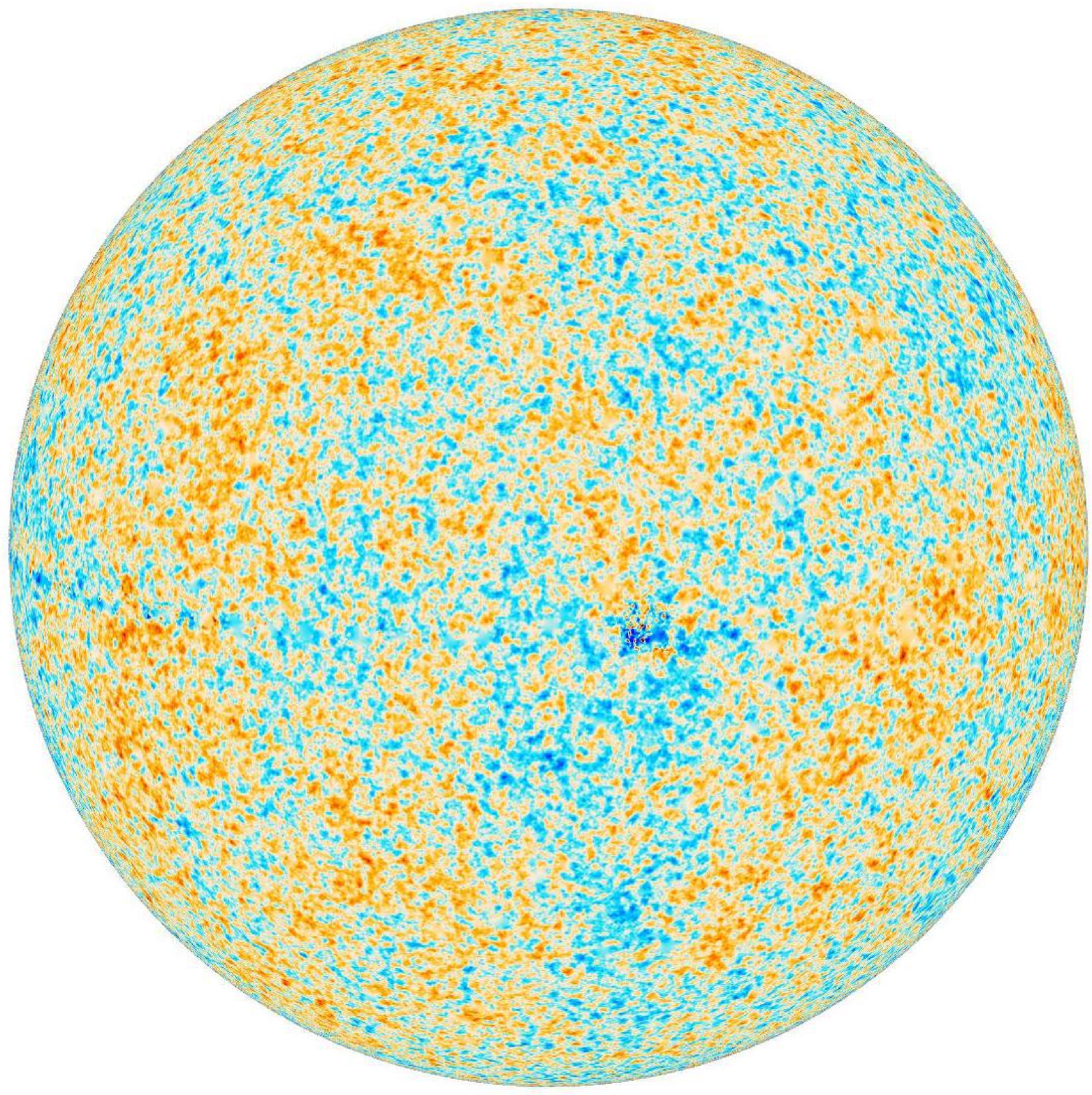}\\[1mm]
  \end{minipage}
\end{minipage}
\caption{ETOPO1 data (left) and CMB data (right)}\label{fig:data}
\end{figure}
}
\end{example}

\section{Final remarks}
\begin{enumerate}[1)]
\item Besides the orthogonal polynomials (Fourier domain) approach, the usual time domain approach and the group-theoretical approach are other two widely used approaches. In the usual time domain approach, wavelets are restricted to intervals, squares, cubes or regular domains in higher dimensions. For instance, Cohen, Daubechies and Vial \cite{CoDaVi1993} constructed the wavelets on a compact interval similar to wavelets on $\R$ by carefully handling  generators on the boundary. Using the lifting schemes, Sweldens \cite{Sweldens1998} constructed wavelets on irregular domains including compact intervals and surfaces.  Continuous wavelets on the two-dimensional sphere $\sph{2}$ can also be achieved using the group of rotations $SO(3)$, as shown by Freeden, Gervens and Schreiner \cite{FrGeSc1998} and Antonio and Vandergheynst et al. \cite{AnVa1998,AnVa1999}.

\item  In high-dimensional (big) data analysis such as in denoising and inpainting of image and video, in order to avoid the boundary effect when estimating convolution with a filter, one usually exploits symmetric extension and periodization techniques for the data. The data can then be regarded as samples on the torus $\mathbb{T}^d$ for which the convolution in time domain are implemented by the classical discrete Fourier transforms in the frequency domain. When data are sampled from the regular integer grid, the framelet filter bank transforms reduce to the classical framelet filter bank transforms in $\Rd$, see e.g. \cite{HaZhZh2016}, and the fast discrete Fourier transforms (FFT) are used for the fast implementation of framelet transforms. Multiscale analysis of data sampled at the regular integer grid has been widely used in inpainting, denoising, debluring, segmentation and so on.
For data sampled at irregular grids, the nonequispaced fast Fourier transforms \cite{Ku2006} provide an algorithmic realization of fast framelet transforms in our setting for framelet systems on $\Lpm[\mathbb{T}^{d}]{2}$, which is also an active area  of ``irregular sampling'' or ``non-uniform sampling''. Besides, it is possible and promising to use lattice rules and QMC designs with low-discrepancy, see e.g. \cite{DiKuSl2013,DiPi2010,SlJo1994}, to construct framelets in a high-dimensional torus.

\item  The polynomial-exact quadrature rule simplifies the conditions and implementation for tight framelets. The choice of quadrature rules for tight framelets in  Theorem~\ref{thm:framelet.tightness} is in fact rather general. One can also consider non-polynomial-exact quadrature rules satisfying one of equivalence conditions (iv) and (v) in Theorem~\ref{thm:framelet.tightness}. These would bring flexibility when designing tight framelet systems on manifolds, as there are many quadrature rules with good geometric property and good approximation for numerical integration without the requirement of polynomial exactness, see e.g. QMC designs on the two-dimensional sphere \cite{BrSaSlWo2014}, minimal energy points on a compact manifold \cite{HaSa2004}. Investigation into the construction of tight framelets and tight framelet filter banks for a manifold with these quadrature rules (not exact for polynomials) is significant.

\item In this paper, we assume $\mfd$ a compact Riemannian manifold and $\{(\eigfm,\eigvm)\}_{\ell=0}^{\infty}$ an orthonormal eigen-pair for $\Lpm{2}$. In fact, our results can be extended to a more general setting, for example, metric measure spaces \cite{MaMh2008}, graphs, meshes, which we will report elsewhere.

\item In the paper, the continuous framelets $\boldsymbol{\varphi}_{j,y}$ and $\boldsymbol{\psi}_{j,y}^n$ as well as the inner product for $L_2(\mfd)$ can be complex-valued. In implementation, taking square-root does not affect the numerical results of the algorithm for framelet systems. On the other hand, to avoid the square-root of negative numbers, one can simply move the square-root of the weights from the {\fmt} decomposition to the {\fmt} reconstruction. In this way, weights are computed in one step without splitting $\omega_{j,k}$ to $\sqrt{\omega_{j,k}}\cdot\sqrt{\omega_{j,k}}$. Such treatment of weights as well as the further relaxation on the filter banks can be done (in both theory and practice) by using dual framelets.

\item Directional wavelets on the sphere based on group representations are given by \cite{Ig2017,WiMcVaBl2008}. We do not consider directional sensitivity of the framelet systems. It is desirable to incorporate directionality into our framelets on manifolds.

\end{enumerate}

\section*{Acknowledgements}
The authors thank anonymous reviewers for their valuable comments and suggestions that help the improvement of the quality of this paper.
The authors thank Ian H. Sloan, Robert S. Womersley and Quoc T. Le Gia for their helpful comments on the paper.
Some of the results in this paper have been derived using the HEALPix \cite{Gorski_etal2005}.


{\footnotesize 

}


\begin{thebibliography}{10}

\bibitem{Planck2015IX}
R.~Adam, P.~Ade, N.~Aghanim, M.~Arnaud, M.~Ashdown, J.~Aumont, C.~Baccigalupi,
  A.~J. Banday, R.~B. Barreiro, J.~G. Bartlett, N.~Bartolo, et~al.
\newblock {Planck 2015 results. IX. Diffuse component separation: CMB maps}.
\newblock {\em Astron. Astrophys.}, 594:A9, 2016.

\bibitem{AmEa2009}
C.~Amante and B.~W. Eakins.
\newblock {ETOPO1} 1 arc-minute global relief model: Procedures, data sources
  and analysis.
\newblock {\em {NOAA Technical Memorandum NESDIS NGDC-24. National Geophysical
  Data Center, NOAA}}, 2009.

\bibitem{AnVa1998}
J.-P. Antoine and P.~Vandergheynst.
\newblock Wavelets on the {$n$}-sphere and related manifolds.
\newblock {\em J. Math. Phys.}, 39(8):3987--4008, 1998.

\bibitem{AnVa1999}
J.-P. Antoine and P.~Vandergheynst.
\newblock Wavelets on the {$2$}-sphere: a group-theoretical approach.
\newblock {\em Appl. Comput. Harmon. Anal.}, 7(3):262--291, 1999.

\bibitem{Bauer2000}
R.~Bauer.
\newblock Distribution of points on a sphere with application to star catalogs.
\newblock {\em J. Guidance Control Dynamics}, 23(1):130--137, 2000.

\bibitem{Bennett_etal2013}
C.~L. {\relax Bennett et al}.
\newblock Nine-year {Wilkinson} microwave anisotropy probe {(WMAP)}
  observations: Final maps and results.
\newblock {\em Astrophys. J. Suppl. Ser.}, 208(2):20, 2013.

\bibitem{BoKuZh2015}
B.~G. Bodmann, G.~Kutyniok, and X.~Zhuang.
\newblock Gabor shearlets.
\newblock {\em Appl. Comput. Harmon. Anal.}, 38(1):87--114, 2015.

\bibitem{Br_etal2014}
L.~Brandolini, C.~Choirat, L.~Colzani, G.~Gigante, R.~Seri, and G.~Travaglini.
\newblock Quadrature rules and distribution of points on manifolds.
\newblock {\em Ann. Sc. Norm. Super. Pisa Cl. Sci. (5)}, 13(4):889--923, 2014.

\bibitem{BrDiSaSlWaWo2014}
J.~S. Brauchart, J.~Dick, E.~B. Saff, I.~H. Sloan, Y.~G. Wang, and R.~S.
  Womersley.
\newblock Covering of spheres by spherical caps and worst-case error for equal
  weight cubature in {S}obolev spaces.
\newblock {\em J. Math. Anal. Appl.}, 431(2):782--811, 2015.

\bibitem{BrSaSlWo2014}
J.~S. Brauchart, E.~B. Saff, I.~H. Sloan, and R.~S. Womersley.
\newblock {QMC} designs: optimal order quasi {M}onte {C}arlo integration
  schemes on the sphere.
\newblock {\em Math. Comp.}, 83(290):2821--2851, 2014.

\bibitem{CaDeDoYi2006}
E.~Cand{\`e}s, L.~Demanet, D.~Donoho, and L.~Ying.
\newblock Fast discrete curvelet transforms.
\newblock {\em Multiscale Model. Simul.}, 5(3):861--899, 2006.

\bibitem{Chavel1984}
I.~Chavel.
\newblock {\em Eigenvalues in {R}iemannian geometry}.
\newblock Academic Press, Inc., Orlando, FL, 1984.

\bibitem{ChSlWo2014}
A.~Chernih, I.~H. Sloan, and R.~S. Womersley.
\newblock Wendland functions with increasing smoothness converge to a {G}aussian.
\newblock {\em Adv. Comput. Math.}, 40(1):185--200, 2014.

\bibitem{Chui1992}
C.~K. Chui.
\newblock {\em An introduction to wavelets}.
\newblock Academic Press, Inc., Boston, MA, 1992.

\bibitem{ChHeSt2002}
C.~K. Chui, W.~He, and J.~St{\"o}ckler.
\newblock Compactly supported tight and sibling frames with maximum vanishing
  moments.
\newblock {\em Appl. Comput. Harmon. Anal.}, 13(3):224--262, 2002.

\bibitem{Ch1997}
F. R. K. Chung.
\newblock Spectral graph theory, volume 92.
\newblock American Mathematical Soc., 1997.

\bibitem{CoDaVi1993}
A.~Cohen, I.~Daubechies, and P.~Vial.
\newblock Wavelets on the interval and fast wavelet transforms.
\newblock {\em Appl. Comput. Harmon. Anal.}, 1(1):54--81, 1993.

\bibitem{CoMa2006}
R.~R. Coifman and M.~Maggioni.
\newblock Diffusion wavelets.
\newblock {\em Appl. Comput. Harmon. Anal.}, 21(1):53--94, 2006.

\bibitem{wiki_IntelTickTock}
Intel Corporation.
\newblock {Tick-Tock} model.
\newblock {\em
  \newline\emph{\url{https://en.wikipedia.org/wiki/Tick-Tock_model}}}, 2016.

\bibitem{DaXu2013}
F.~Dai and Y.~Xu.
\newblock {\em Approximation theory and harmonic analysis on spheres and
  balls}.
\newblock Springer, New York, 2013.

\bibitem{Daubechies1992}
I.~Daubechies.
\newblock {\em Ten lectures on wavelets}.
\newblock SIAM, Philadelphia, PA, 1992.

\bibitem{DaHaRoSh2003}
I.~Daubechies, B.~Han, A.~Ron, and Z.~Shen.
\newblock Framelets: {MRA}-based constructions of wavelet frames.
\newblock {\em Appl. Comput. Harmon. Anal.}, 14(1):1--46, 2003.

\bibitem{DiKuSl2013}
J.~Dick, F.~Y. Kuo, and I.~H. Sloan.
\newblock High-dimensional integration: the quasi-{M}onte {C}arlo way.
\newblock {\em Acta Numer.}, 22:133--288, 2013.

\bibitem{DiPi2010}
J.~Dick and F.~Pillichshammer.
\newblock {\em Digital nets and sequences. Discrepancy theory and quasi-Monte
  Carlo integration}.
\newblock Cambridge University Press, Cambridge, 2010.

\bibitem{Dong2017}
B.~Dong.
\newblock Sparse representation on graphs by tight wavelet frames and applications.
\newblock {\em Appl. Comput. Harmon. Anal.}, 42(3):452--479, 2017.

\bibitem{DrHe1994}
J.~R. Driscoll and D.~M. Healy, Jr.
\newblock Computing {F}ourier transforms and convolutions on the {$2$}-sphere.
\newblock {\em Adv. in Appl. Math.}, 15(2):202--250, 1994.

\bibitem{FiMh2011}
F.~Filbir and H.~N. Mhaskar.
\newblock Marcinkiewicz-{Z}ygmund measures on manifolds.
\newblock {\em J. Complexity}, 27(6):568--596, 2011.

\bibitem{FiPr1997}
B. Fischer and J. Prestin.
\newblock Wavelets based on orthogonal polynomials.
\newblock {\em Math. Comput.}, 66 (220):1593--1618, 1997.

\bibitem{FrGeSc1998}
W.~Freeden, T.~Gervens, and M.~Schreiner.
\newblock {\em Constructive approximation on the sphere. {W}ith applications to
  geomathematics}.
\newblock The Clarendon
  Press, Oxford University Press, New York, 1998.




\bibitem{Gorski_etal2005}
K.~M. G\'{o}rski, E.~Hivon, A.~J. Banday, B.~D. Wandelt, F.~K. Hansen,
  M.~Reinecke, and M.~Bartelmann.
\newblock {HEALPix}: A framework for high-resolution discretization and fast
  analysis of data distributed on the sphere.
\newblock {\em Astrophys. J.}, 622(2):759, 2005.

\bibitem{Grieser2002}
D.~Grieser.
\newblock Uniform bounds for eigenfunctions of the {L}aplacian on manifolds
  with boundary.
\newblock {\em Comm. Partial Differential Equations}, 27(7-8):1283--1299, 2002.

\bibitem{HaTo2014}
N. Hale and A. Townsend.
\newblock A fast, simple, and stable Chebyshev--Legendre transform using an asymptotic formula.
\newblock {\em SIAM J. Sci. Comput.}, 36:A148--A167,  2014.

\bibitem{HaVaGr2011}
D.~K. Hammond, P.~Vandergheynst, and R.~Gribonval.
\newblock Wavelets on graphs via spectral graph theory.
\newblock {\em Appl. Comput. Harmon. Anal.}, 30(2):129 -- 150, 2011.

\bibitem{Han2010}
B.~Han.
\newblock Pairs of frequency-based nonhomogeneous dual wavelet frames in the
  distribution space.
\newblock {\em Appl. Comput. Harmon. Anal.}, 29(3):330--353, 2010.

\bibitem{Han2012}
B.~Han.
\newblock nonhomogeneous wavelet systems in high dimensions.
\newblock {\em Appl. Comput. Harmon. Anal.}, 32(2):169--196, 2012.

\bibitem{HaZhZh2016}
B.~Han, Z.~Zhao, and X.~Zhuang.
\newblock Directional tensor product complex tight framelets
with low redundancy.
\newblock {\em Appl. Comput. Harmon. Anal.}, 41(2):603--637, 2016.

\bibitem{HaZh2015}
B.~Han and X.~Zhuang.
\newblock Smooth affine shear tight frames with {MRA} structure.
\newblock {\em Appl. Comput. Harmon. Anal.}, 39(2):300--338, 2015.

\bibitem{HaSa2004}
D.~P. Hardin and E.~B. Saff.
\newblock Discretizing manifolds via minimum energy points.
\newblock {\em Notices Amer. Math. Soc.}, 51(10):1186--1194, 2004.

\bibitem{HeRoKoMo2003}
D.~M. Healy, Jr., D.~N. Rockmore, P.~J. Kostelec, and S.~Moore.
\newblock F{FT}s for the 2-sphere-improvements and variations.
\newblock {\em J. Fourier Anal. Appl.}, 9(4):341--385, 2003.

\bibitem{HeWo2012}
K.~Hesse and R.~S. Womersley.
\newblock Numerical integration with polynomial exactness over a spherical cap.
\newblock {\em Adv. Comput. Math.}, 36(3):451--483, 2012.

\bibitem{Ig2017}
I. Iglewska-Nowak.
\newblock Frames of directional wavelets on $n$-dimensional spheres.
\newblock {\em  Appl. Comput. Harmon. Anal.},  43 (1):148--161, 2017.

\bibitem{KeKuPo2007}
J.~Keiner, S.~Kunis, and D.~Potts.
\newblock Efficient reconstruction of functions on the sphere from scattered
  data.
\newblock {\em J. Fourier Anal. Appl.}, 13(4):435--458, 2007.

\bibitem{KoRe2015}
O. Kounchev and H. Render.
\newblock A new cubature formula with weight functions on the disc, with error estimates.
\newblock  arXiv:1509.00283.

\bibitem{Ku2006}
S.~Kunis, Nonequispaced FFT
Generalisation and Inversion, Ph.D. Thesis, 2006.



\bibitem{LeMh2006}
Q. T. Le Gia and H. N. Mhaskar.
\newblock Polynomial operators and local approximation of solutions of pseudo-differential equations on the sphere. 
\newblock {\em Numer. Math.}, 103(2):299--322, 2006.


\bibitem{LeMh2008}
Q.~T. Le~Gia and H.~N. Mhaskar.
\newblock Localized linear polynomial operators and quadrature formulas on the
  sphere.
\newblock {\em SIAM J. Numer. Anal.}, 47(1):440--466, 2008.

\bibitem{LeSlWe2010}
Q.~T. Le~Gia, I.~H. Sloan, and H.~Wendland.
\newblock Multiscale analysis in {S}obolev spaces on the sphere.
\newblock {\em SIAM J. Numer. Anal.}, 48(6):2065--2090, 2010.

\bibitem{MaMh2008}
M.~Maggioni and H.~N. Mhaskar.
\newblock Diffusion polynomial frames on metric measure spaces.
\newblock {\em Appl. Comput. Harmon. Anal.}, 24(3):329--353, 2008.

\bibitem{Mallat2009}
S.~Mallat.
\newblock {\em A wavelet tour of signal processing. The sparse way, With
  contributions from Gabriel Peyr{\'e}}.
\newblock Elsevier/Academic Press, Amsterdam, 2009.

\bibitem{Meyer1990}
Y.~Meyer.
\newblock {\em Ondelettes et op\'erateurs. {I}}.
\newblock Actualit\'es Math\'ematiques. [Current Mathematical Topics]. Hermann,
  Paris, 1990.

\bibitem{Mh2005}
H.~Mhaskar.
\newblock On the representation of smooth functions on the sphere using finitely many bits.
\newblock {\em Appl. Comput. Harmon. Anal.}, 18(3):215 -- 233, 2005.


\bibitem{Mhaskar2010}
H.~Mhaskar.
\newblock Eignets for function approximation on manifolds.
\newblock {\em Appl. Comput. Harmon. Anal.}, 29(1):63 -- 87, 2010.

\bibitem{MhNaWa2001}
H.~N. Mhaskar, F.~J. Narcowich, and J.~D. Ward.
\newblock Spherical {M}arcinkiewicz-{Z}ygmund inequalities and positive
  quadrature.
\newblock {\em Math. Comp.}, 70(235):1113--1130, 2001.

\bibitem{MhPr2004}
H. N. Mhaskar  and J. Prestin.
\newblock Polynomial frames: a fast tour.
\newblock {\em Approximation Theory XI},  Gatlinburg, 101--132, 2004.


\bibitem{wiki_Moore}
G.~Moore.
\newblock Moore's law.
\newblock {\em
  \newline\emph{\url{https://en.wikipedia.org/wiki/Moore\%27s_law}}}, 2016.


\bibitem{NaPeWa2006-1}
F.~Narcowich, P.~Petrushev, and J.~Ward.
\newblock Decomposition of {B}esov and {T}riebel-{L}izorkin spaces on the
  sphere.
\newblock {\em J. Funct. Anal.}, 238(2):530--564, 2006.

\bibitem{NaPeWa2006-2}
F.~J. Narcowich, P.~Petrushev, and J.~D. Ward.
\newblock Localized tight frames on spheres.
\newblock {\em SIAM J. Math. Anal.}, 38(2):574--594, 2006.

\bibitem{RoTy2006}
V.~Rokhlin and M.~Tygert.
\newblock Fast algorithms for spherical harmonic expansions.
\newblock {\em SIAM J. Sci. Comput.}, 27(6):1903--1928, 2006.

\bibitem{RoSh1997}
A.~Ron and Z.~Shen.
\newblock Affine systems in {$L_2(\bold R^d)$}: the analysis of the analysis
  operator.
\newblock {\em J. Funct. Anal.}, 148(2):408--447, 1997.


\bibitem{RoSa2000}
S. T. Roweis and L. K. Saul.
\newblock Nonlinear dimensionality reduction by locally linear embedding.
\newblock {\em Science}, 290(5500):2323--2326, 2000

\bibitem{ShWiHoVa2015}
D.~I. Shuman, C.~Wiesmeyr, N.~Holighaus, and P.~Vandergheynst.
\newblock Spectrum-adapted tight graph wavelet and vertex-frequency frames.
\newblock {\em IEEE Trans. Signal Process.}, 63(16):4223--4235, 2015.

\bibitem{Si2006}
A. Singer.
\newblock From graph to manifold Laplacian: The convergence rate.
\newblock {\em Appl. Comput. Harmon. Anal.}, 21(1):128--134, 2006.

\bibitem{SlJo1994}
I. H. Sloan and S. Joe.
\newblock {\em Lattice methods for multiple integration.}
\newblock The Clarendon Press, Oxford University Press, New York, 1994.

\bibitem{SlWo2004}
I.~H. Sloan and R.~S. Womersley.
\newblock Extremal systems of points and numerical integration on the sphere.
\newblock {\em Adv. Comput. Math.}, 21(1-2):107--125, 2004.

\bibitem{SlWo2012}
I. H. Sloan and R. S. Womersley.
\newblock Filtered hyperinterpolation: a constructive polynomial approximation on the sphere.
\newblock {\em Int. J. Geomath.}, 3(1):95--117, 2012.

\bibitem{Sweldens1998}
W.~Sweldens.
\newblock The lifting scheme: a construction of second generation wavelets.
\newblock {\em SIAM J. Math. Anal.}, 29(2):511--546, 1998.

\bibitem{TeDeLa2000}
J. B. Tenenbaum, V. De Silva, and J. C. Langford.
\newblock A global geometric framework for nonlinear
dimensionality reduction.
\newblock {\em Science}, 290(5500):2319-2323, 2000.

\bibitem{ThTo1996}
P.~Thompson and AW.~Toga.
\newblock A surface-based technique for warping three-dimensional images of the brain.
\newblock {\em IEEE Trans Med Imaging.},  15(4):402--17, 1996.

\bibitem{WaLeSlWo2017}
Y. G. Wang, Q. T. Le Gia, I. H. Sloan and R. S. Womersley.
\newblock Fully discrete needlet approximation on the sphere.
\newblock {\em Appl. Comput. Harmon. Anal.}, 43(2):292--316, 2017.

\bibitem{Weyl1912}
H.~Weyl.
\newblock Das asymptotische Verteilungsgesetz der Eigenwerte linearer partieller Differentialgleichungen (mit einer Anwendung auf die Theorie der Hohlraumstrahlung).
\newblock {\em Math. Ann.}, 71(4):441--479.


\bibitem{WiMcVaBl2008}
Y. Wiaux,  J. D. McEwen, P. Vandergheynst, and O. Blanc.
\newblock Exact reconstruction with directional wavelets on the sphere.
\newblock {\em Mon. Not. R. Astron. Soc.}, 388(2):770--788, 2008.


\bibitem{Womersley_ssd_URL}
R.~S. Womersley.
\newblock Efficient spherical designs with good geometric properties.
\newblock {\em
  \newline\emph{\url{http://web.maths.unsw.edu.au/~rsw/Sphere/EffSphDes/}}},
  2015.

\bibitem{Zhou2003}
D.-X. Zhou,
\newblock Capacity of reproducing kernel spaces in learning theory.
\newblock {\em IEEE Tran. Info. Theory},
49(7):  1743--1752, 2003.

\bibitem{Zhuang2015}
X.~Zhuang.
\newblock Smooth affine shear tight frames: digitization and applications.
\newblock In {\em SPIE Optical Engineering+ Applications}, pages
  959709--959709, 2015.

\bibitem{Zhuang2016}
X.~Zhuang.
\newblock Digital affine shear transforms: fast realization and applications in
  image/video processing.
\newblock {\em SIAM J. Imaging Sci.}, 9(3):1437--1466, 2016.

\end{thebibliography}
\end{document}